\documentclass{amsart}[a4paper,twoside]

\usepackage{graphicx}
\usepackage{amssymb, amsthm, amsmath, amsfonts}
\usepackage{enumitem} 
\usepackage{mathtools}
\usepackage{xspace}
\usepackage{cleveref} 
\usepackage{xcolor}
\usepackage[foot]{amsaddr}
\usepackage[round]{natbib}

\theoremstyle{plain}%
\newtheorem{theorem}{Theorem}[subsection]%
\newtheorem*{proposition*}{Proposition}%
\newtheorem{lemma}[theorem]{Lemma}%
\newtheorem*{lemma*}{Lemma}%
\newtheorem{corollary}[theorem]{Corollary}%

\theoremstyle{remark}%
\newtheorem{example}[theorem]{Example}%
\newtheorem{remark}[theorem]{Remark}%
\newtheorem{conjecture}[theorem]{Conjecture}%
\theoremstyle{definition}%
\newtheorem{definition}[theorem]{Definition}%

\numberwithin{equation}{section}

\newcommand{\mc}[1]{\ensuremath{\mathcal{#1}}} 
\newcommand{\mr}[1]{\ensuremath{\mathrm{#1}}} 
\newcommand{\dd}{\ensuremath{\partial}}
\newcommand{\dr}{\ensuremath{\mathrm{d}}}
\newcommand{\res}[1]{\ensuremath{\vert_{#1}}}
\newcommand{\one}{\ensuremath{\mathbf{1}}}
\newcommand{\eps}{\varepsilon} 
\DeclareMathOperator*{\Tr}{Tr} 
\DeclareMathOperator*{\tr}{trace} 
\DeclareMathOperator{\vspan}{span}

\newcommand{\gF}{\mathfrak{F}}
\newcommand{\gV}{\mathfrak{V}}
\newcommand{\gE}{\mathfrak{E}}

\newcommand{\cW}{\mathcal{W}}
\newcommand{\cE}{\mathcal{E}}
\newcommand{\cP}{\mathcal{P}}
\newcommand{\cS}{\mathcal{S}}
\newcommand{\cM}{\mathcal{M}}

\newcommand{\iid}{\stackrel{i.i.d.}{\sim}}

\newcommand{\EE}{\mathbb{E}}
\newcommand{\NN}{\mathbb{N}}
\newcommand{\RR}{\mathbb{R}}

\newcommand{\spd}{\ensuremath{\mathcal{P}}}
\newcommand{\sym}{\ensuremath{\mathcal{S}}}
\newcommand{\W}{\ensuremath{\mathcal{W}}}
\newcommand{\F}{\ensuremath{\mathcal{W}}}
\newcommand{\G}[1]{\ensuremath{\mathcal{G}_{#1}}} 

\newcommand{\Exp}{\ensuremath{\mathrm{Exp}}}
\newcommand{\Log}{\ensuremath{\mathrm{Log}}}

\newcommand{\st}{\mbox{ such that }}
\newcommand{\Span}{\ensuremath{\mathrm{span}}}
\newcommand{\rank}{\ensuremath{\mathrm{rank}}}
\newcommand{\col}{\ensuremath{\mathrm{col}}}

\DeclareMathOperator*{\argmin}{\mbox{\rm argmin}}

\newcommand{\waldtop}{wald topology\xspace}
\newcommand{\waldtops}{wald topologies\xspace}
\newcommand{\waldspace}{wald space\xspace}
\newcommand{\Waldspace}{Wald space\xspace}
\newcommand{\wald}{wald\xspace}
\newcommand{\walds}{w\"alder\xspace}

\newcommand{\grove}{grove\xspace}
\newcommand{\groves}{groves\xspace}
\newcommand{\bhvspace}{BHV space\xspace}

\begin{document}

\title[Wald Space]{Foundations of the Wald Space for Phylogenetic Trees}


\author{Jonas Lueg}
\address[Jonas Lueg and Stephan~F. Huckemann]{Felix-Bernstein-Institute for Mathematical Statistics in the Biosciences\\ Georg-August-Universit\"at\\ G\"ottingen\\ Germany}
\email{jonas.lueg@stud.uni-goettingen.de}
\author{Maryam~K. Garba}
\address[Maryam~K. Garba and Tom~M.~W. Nye]{School of Mathematics, Statistics and Physics\\ Newcastle University\\ UK}
\email{m.k.garba1@ncl.ac.uk}
\author{Tom~M.~W. Nye}
\email{tom.nye@ncl.ac.uk}
\author{Stephan~F. Huckemann}
\email{huckeman@math.uni-goettingen.de}

\date{Sep. 5, 2022}

\begin{abstract}

Evolutionary relationships between species are represented by phylogenetic trees, but these relationships are subject to uncertainty due to the random nature of evolution. 
A geometry for the space of phylogenetic trees is necessary in order to properly quantify this uncertainty during the statistical analysis of collections of possible evolutionary trees inferred from biological data.
Recently, the wald space has been introduced: a length space for trees which is a certain subset of the manifold of symmetric positive definite matrices.
In this work, the wald space is introduced formally and its topology and structure is studied in detail. 
In particular, we show that wald space has the topology of a disjoint union of open cubes, it is contractible, and by careful characterization of cube boundaries, we demonstrate that wald space is a Whitney stratified space of type (A). 
Imposing the metric induced by the affine invariant metric on symmetric positive definite matrices, we prove that wald space is a geodesic Riemann stratified 
space.  
A new numerical method is proposed and investigated for construction of geodesics, computation of Fr\'echet means and calculation of curvature in wald space.
This work is 
intended to serve as a mathematical foundation for further geometric and statistical research on this space. 

\end{abstract}


\maketitle


\section{Introduction}

\subsection{Background}

Over billions of years, evolution has been driven by unobserved random processes. 
Inferences about evolutionary history, which by necessity are largely based on observations of present-day species, are therefore always subject to some level of uncertainty. 
Phylogenetic trees are used to represent possible evolutionary histories relating a set of species, or taxa, which form the leaves of each tree. 
Internal vertices on phylogenetic trees usually represent speciation events, and edge lengths represent the degree of evolutionary divergence over any given edge. 
Trees are typically inferred from genetic sequence data from extant species, and a variety of well-established statistical methods exist for phylogenetic inference \citep{felsenstein_inferring_2003}. 
These generally output a sample of trees -- a collection of possible evolutionary histories compatible with the data. 
Moreover, evolutionary relationships can vary stochastically from one gene to another, giving a further source of random variation in samples of trees \citep{maddison1997gene}. 
It is then natural to pose statistical questions about such samples: for example identifying a sample mean, identifying principal modes of variation in the sample, or testing differences between samples. 
This, in turn, calls for the design of suitable metric spaces in which each element is a phylogenetic tree on some fixed set of taxa, and which are ideally both biologically substantive and computationally tractable.

The design of these \emph{tree spaces} is aggravated by the continuous and combinatorial nature of phylogenetic trees and furthermore, a metric space that is also a geodesic space (so that distance corresponds to the length of shortest paths, also called geodesics) is to be preferred, as it facilitates computation of statistics like the Fr\'echet mean significantly.
The first geodesic space of phylogenetic trees was introduced by \cite{billera_geometry_2001} and is called the \bhvspace, where BHV is an acronym of the authors Billera, Holmes and Vogtmann.
For a fixed set of species $L=\{1,\dots,N\}$, also called taxa or labels, with $3\leq N\in\NN$, \bhvspace is constructed via embedding all phylogenetic trees into a Euclidean space $\RR^M$, where $M\in\NN$ is exponentially growing in $N$, and then taking the infinitesimally induced intrinsic distance on this embedded subset, giving a metric space.
As a result, \bhvspace features a very rich and computationally tractable geometry as it is CAT(0) space, i.e. 
globally of non-positive curvature, and thus having unique geodesics 
and Fréchet means.
Starting with the development of a polynomial time algorithm for computing geodesics and thereby overcoming the combinatorial difficulties (\cite{owen_fast_2011}), many algorithms have been derived for computing statistics like sample means (\cite{bacak_computing_2014, miller_polyhedral_2015}) and variance (\cite{brown_mean_2020}), confidence regions for the population mean (\cite{willis_confidence_2016}) and principal component analysis \cite{nye_principal_2011, nye_algorithm_2014, nye_principal_2016}, Feragen et al. 2013). 
The BHV paper has had considerable influence more widely on research in phylogenetics (see \cite{suchard2005stochastic} for example), non-Euclidean statistics (\cite{marron2014overview}), algebraic geometry (\cite{ardila2006bergman}), probability theory (\cite{evans2006rayleigh}) and other area of mathematics (\cite{baez}).
In addition to the BHV tree space, a variety of alternative tree spaces have been proposed, both for discrete and continuous underlying point sets of trees. 
For example, in the \emph{tropical tree space} \citep{Speyer2004,Monod2022} edge weights are times, not evolutionary divergence, thus allowing for a distance metric between two trees involving tropical algebra. 
 
The geometries of the BHV and tropical tree spaces are unrelated to the methods used to infer phylogenies from sequence data. 
In contrast, there are substantially different tree spaces that originate via the evolutionary genetic substitution models used by molecular phylogenetic methods for tree inference (see \cite{Yang2006} for details of these).
Evolutionary substitution models are essentially Markov processes on a phylogenetic tree with state space $\Omega$.
For DNA sequence data, the state space is $\Omega=\{\mr{A},\mr{C},\mr{T},\mr{G}\}$. 
Under an appropriate set of assumptions on the substitution model, each tree determines a probability distribution on the set of possible letter patterns at the labelled vertices $L$, (i.e.~a probability mass function $p\colon\Omega^N \to [0, 1]$), $N=|L|$, and this can be used to compute the likelhiood of any tree.
At about the same that \bhvspace was introduced, \cite{kim_slicing_2000} provided a geometrical interpretation of tree estimation methods, where, given the substitution model, an embedding of phylogenetic trees into an $\vert\Omega\vert^N$-dimensional simplex using the likelihoods was discussed informally.
The concept was then picked up by \cite{moulton_peeling_2004}, introducing the topological space known as the \emph{edge-product space}, taking not only into account phylogenetic trees but also forests, characterising each forest via a vector containing correlations between all pairs of labels in $L$ under the induced distribution $p$.
This representation is then an embedding of all phylogenetic forests into a $N(N-1)/2$-dimensional space.
Using the same characterisation of phylogenetic trees via distributions on $\Omega^N$ obtained from a fixed substitution model, \cite{garba_probabilistic_2018} considered probabilistic distances to obtain metrics on tree space, but these metrics do not yield length spaces.
Therefore, in \cite{garba_information_2021}, the fact that all phylogenetic trees with a fixed fully resolved tree topology are a manifold was used to apply the Fisher information geometry for statistical manifolds on each such piece of the space to eventually obtain a metric space that is a length space. 
Additionally, instead of using substitution models with finite state space $\Omega$, \cite{garba_information_2021} considered a Gaussian model with state space $\Omega = \RR$ 
in order to deal with the problem of computational tractability. 
The distributions characterising phylogenetic trees are then zero-mean multivariate Gaussians, and sums over $\Omega^N$ for discrete $\Omega$ are replaced with integrals over $\RR^N$. 
The characterisation with this Gaussian model together with the choice of the Fisher information geometry and the extension to phylogenetic forests ultimately leads to the \emph{\waldspace}, which is essentially an embedding of the phylogenetic forests into the real symmetric $N\times N$-dimensional strictly positive definite matrices $\spd$ (\cite{garba_information_2021}). 
The elements of \waldspace are called \emph{\walds} (``Wald'' is a German word meaning ``forest'').

\begin{figure}[ht]
    \centering%
    \hfill
    \includegraphics[width=0.6\linewidth]{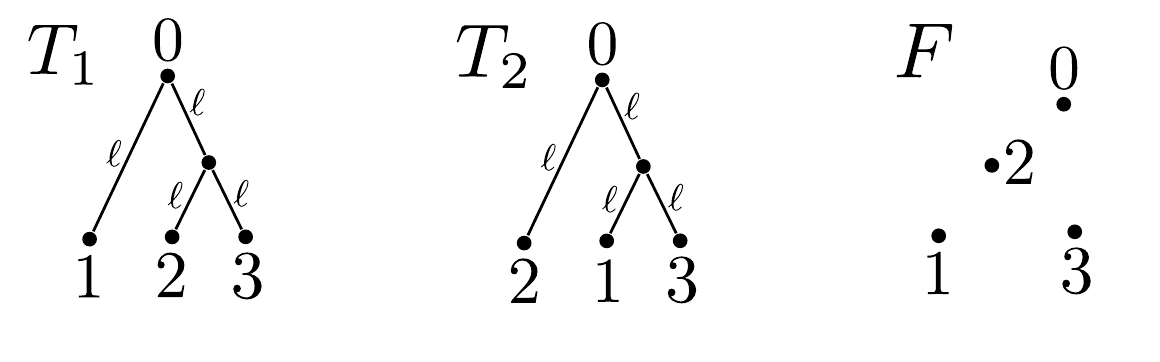}\hfill\hfill
    \caption{Two trees $T_1$ and $T_2$ with positive edge length $\ell\in(0,\infty)$. 
    Letting $\ell\to\infty$, the intuitive limit element for both trees is the forest $F$, as species are considered not related if their evolutionary distance approaches infinity. 
    In \waldspace, the distance between $T_1$ and $T_2$ goes to zero accordingly as $\ell\to\infty$ and their limit is the forest $F$ that is also contained in \waldspace. In \bhvspace however, their distance goes to infinity as $\ell\to\infty$ and $F$ is not an element of the space.}
    \label{fig:introduction-example-biologically-reasonability}
\end{figure}

The geometry of \waldspace is fundamentally different from \bhvspace (\cite{garba_information_2021, lueg_wald_2021}), as illustrated in Fig.~\ref{fig:introduction-example-biologically-reasonability}, which also underlines the biological reasonability of the \waldspace.
Loosely speaking, \waldspace can be viewed topologically as being obtained by compactifying the boundaries at the ``infinities'' of \bhvspace, which comes with the price of fundamentally changing the geometry that is not locally Euclidean anymore. We avoid, though, the compactification at the ``zeroes'' of the edge-product space proposed by \cite{moulton_peeling_2004} which suggests itself by mathematical elegance. It is biologically questionable, however, as it would allow different taxa to agree with one another.
In \cite{garba_information_2021}, apart from defining the \waldspace, certain properties of the space were established, such as showing the distance between any two points to be finite, and algorithms for approximating geodesics were proposed. 
In \cite{lueg_wald_2021}, a compact definition of \waldspace as well as more refined algorithms for approximating geodesics were introduced.


\subsection{Contribution of this paper}

Previous work on \waldspace established the space as a length space, and this paper was originally motivated by the aim of proving the existence of a minimising geodesic betweeen every two points, i.e.~establishing \waldspace as a geodesic metric space, since the existence of geodesics is crucial for performing statistical analysis within the space. 
This aim is achieved in Theorem~\ref{thm:waldspace-is-geodesic-metric-space} below. 
The proof involves three essential characterisations of the elements of wald space (as graph-theoretic forests; as split systems; and as certain symmetric positive definite matrices).  
In turn, these enable a rigorous analysis of the topology of \waldspace, such as Theorem~\ref{thm:whitney} about its stratified structure, in addition to providing a foundation for further research on this space. 

The remainder of the paper is structured as follows. 
In \Cref{scn:Wald-Def} we define the \waldspace $\W$ for a fixed set of labels $\{1,\dots,N\}$ as equivalence classes of partially labelled graph-theoretic forests. 
The topology on $\W$ is obtained by defining a map $\psi$ from $\W$ into the set of $N\times N$ symmetric positive definite matrices and requiring $\psi$ to be a homeomorphism onto its image. 
We then provide an equivalent, but more tractable, definition in terms of \emph{splits} or biparitions of labels, and an equivalent map $\phi$ from split-representations of \walds to symmetric positive definite matrices. 
In particular, we show that \waldspace can be identified topologically with a disjoint union of open unit cubes. 
Each open unit cube is called a \emph{\grove}. 
In \Cref{scn:topology-of-waldspace} we describe the structure or stratification of the \waldspace by investigating on how the \groves are glued together along their respective boundaries.  
This is achieved by first providing in \Cref{scn:inequalities} a detailed characterization of the matrices in the image $\phi(\W)$ in terms of a set of algebraic constraints on the matrix elements. 
Using this characterization, for example, we show that \waldspace is contractible. 
Then in \Cref{scn:groves-end} we use a partial ordering of forest topologies, first introduced by \cite{moulton_peeling_2004} to establish results about the boundaries of groves and the stratification of \waldspace. 
This culminates in \Cref{scn:whitney} in which we prove \waldspace satifies certain axioms at grove boundaries, collectively known as \emph{Whitney condition (A)} \citep{pflaum_analytic_2001}, which ensure that tangent spaces behave well as the boundaries of strata are approached. 
We then go on to consider the induced affine invariant or information geometry on \waldspace in \Cref{scn:geometry}. 
We show the topology induced by the metric is the same as the previous topology defined using $\phi$, and hence show that $\W$ is a geodesic metric space (i.e.~every two points are connected by a minimising geodesic). 
Finally in \Cref{scn:numerical}, we use a new algorithm for computing approximate geodesics to explore the geometry on \waldspace, specifically computing sectional curvatures within \groves and Alexandrov curvatures for fundamental examples. 
We also investigate the behaviour of the sample Fr\'echet mean, in particular with reference to the issue of \emph{stickiness} observed in in \bhvspace (see for example~\cite{hotz_sticky_2013, huckemann_sticky_2015} for a discription).
In \Cref{scn:discussion} we discuss the contributions of the paper and some of the many open questions and unsolved problems about the geometry of \waldspace.

\subsection{Notation}\label{scn:notation}

Throughout the paper we use the following notation and concepts, where points 4--6 below can be found in standard textbooks of differential geometry, e.g. \cite[Chapter XII]{lang_fundamentals_1999}:
\begin{enumerate}
    \item  
    $2\leq N\in \NN$ is a fixed integer defining the \emph{set of labels} $L = \{1,\dots,N\}$. 
    \item $\sqcup_{i=1}^n A_i$ denotes the union if the $A_i$ are pairwise disjoint ($i=1,\ldots,n$).
    \item 
    When we speak of \emph{partitions}, no empty sets are allowed.
    \item 
    For a set $E$, its cardinality is denoted by $\vert E\vert$.
    \item 
    $\cS$ is the Euclidean space of real symmetric $N\times N$ matrices.
    \item 
    $\cP$ is  the space of real symmetric and positive definite $N\times N$ matrices. It is an open cone in $\cS$ and carries the topology and smooth manifold structure  inherited from $\cS$. In particular, every tangent space $T_P\cP$ at $P \in \cP$ is isomorphic to $\cS$.
    \item 
    We equip $\cP$ with the \emph{affine invariant Riemannian metric}, also called information geometry, yielding a Cartan-Hadamard manifold. Its metric tensor is given by 
    $$ \langle X,Y\rangle_P = \tr (P^{-1}XP^{-1}Y)$$
    for $X,Y \in \cS \cong T_P\cP$ and the unique geodesic $\gamma$ through $P=\gamma(0),Q=\gamma(1)\in \cP$  is given by
    $$ (-\infty,\infty)\to \cP,~ t\mapsto \gamma(t) = \sqrt{P} \exp\left(t \log\left(\sqrt{P}^{-1}Q\sqrt{P}^{-1}\right)\right)\sqrt{P} $$
    with the usual matrix exponential and logarithm, respectively. Here, $\sqrt{P}$ denotes the unique positive definite root of $P$.
    \item 
    The Riemannian metric induces a metric on $\cP$ denoted by $d_\cP$ and for a rectifiable curve $\gamma : [a,b] \to \cP$ let $L_{\cP}(\gamma)$ be its length.
\end{enumerate}

In a word of caution we note that the term \emph{topology} appears in two contexts: (i) as a system of open sets defining a topological space and (ii) as a branching structure of a graph-theoretic forest. 
The latter is standard in the phylogenetic literature, despite the potential for confusion.

\section{Definition of Wald Space via Graphs and Splits}
\label{scn:Wald-Def}

\subsection{From a Graph Viewpoint}\label{scn:graphs}

This section recalls definitions and results from \cite{garba_information_2021} and \cite{lueg_wald_2021}.

\begin{definition}
    A \emph{forest} is a triple $(\gV,\gE,\ell)$, where
    \begin{enumerate}[leftmargin=1.1cm, start=1, label={\textnormal{(PF\arabic*)}}]\itemsep0.5em
      \item \label{item:pf1}
        $(\gV,\gE)$ is a graph-theoretical undirected forest with \emph{vertex} set $\gV$ such that $L\subseteq\gV$ and that $v\in \gV\setminus L$ implies $\deg(v)\geq 3$, where $\deg(v)$ is the degree of a vertex $v$, and \emph{edge} set $\gE\subseteq \big\{\{u,v\}:u,v\in \gV,\, u\neq v \big\}$, 
      \item \label{item:pf2}
        and $\ell = (\ell_e)_{e\in\gE} \in (0,\infty)^\gE$.
    \end{enumerate}
\end{definition}

\begin{definition}\label{def:wald-graph-representation}
    Two forests $(\gV, \gE,\ell), (\gV', \gE',\ell')$ are \emph{topologically} equivalent,
    if there is a bijection  $f\colon \gV\to \gV'$ such that\\[-0.5em]
\begin{enumerate}[label=(\roman*)]\itemsep0.5em
  \item[(i)]  $f(u) = u$ for all $u\in L$,
  \item[(ii)] $\{u,v\} \in \gE \Leftrightarrow \{f(u),f(v)\} \in\gE'$. 
  \end{enumerate}
  They are \emph{phylogenetically} equivalent if additionally
 \begin{enumerate} 
    \item[(iii)]  $\ell_{\{u, v\}} = \ell'_{\{f(u), f(v)\}}$ for all edges $\{u, v\}\in \gE$.\\[-0.5em]
\end{enumerate}
Moreover, \begin{enumerate}
    \item 
Every phylogenetic equivalence class is called a \emph{phylogenetic forest} and denoted by $\gF = [\gV, \gE,\ell]$. 
\item $\W$ is the set of all phylogenetic forests.
\item Every topological equivalence class is called a \emph{forest topology} and denoted by $[\gF] = [\gV,\gE]$.
\end{enumerate}
\end{definition}

\begin{definition}
    Let $(\gV,\gE,\ell)$ be a forest. 
    For two leaves $u,v \in L$ let $\gE(u,v)$ be the set of edges in $\gE$ of the unique  path between $u$ and $v$, if $u$ and $v$ are connected, else set $\gE(u,v)=\emptyset$.
    Further define a mapping of forests via 
    $$\psi: (\gV,\gE,\ell) \mapsto  \big(\rho_{uv}\big)_{u,v=1}^N\in \cS$$
    where
    \begin{equation}\label{eq:definition-phi-map-ell}
         \rho_{uv} = \left\{
        \begin{array}{cl}
             \displaystyle\exp\bigg(-\sum_{e\in \gE(u,v)} \ell_e\bigg)& \mbox{ if } u\neq v\mbox{ and } \gE(u,v) \neq \emptyset \\[1.0em]
             0 &  \mbox{ if } u\neq v\mbox{ and } \gE(u,v) = \emptyset \\[1.0em]
             1 &  \mbox{ if } u=v\end{array}
        \right. 
    \end{equation}
    for $1\leq u,v\leq N$.
\end{definition}

By definition, the above matrix is the same for two forests representing the same phylogenetic forest. It is even positive definite and characterizes phylogenetic forests uniquely as the following theorem shows. 

\begin{theorem}[{\cite{garba_information_2021}}, Theorem~4.1]\label{thm:phi-well-defined-injective-into-spd} 
    For every forest $(\gV,\gE,\ell)$, we have $$\psi(\gV,\gE,\ell)\in \cP$$
    and for any two forests $(\gV,\gE,\ell)$ and $(\gV',\gE',\ell')$ we have
    $$\psi(\gV,\gE,\ell)=\psi(\gV',\gE',\ell')$$
    if and only if
    $$[\gV,\gE,\ell]=[\gV',\gE',\ell']\,.$$
\end{theorem}

In consequence of Theorem \ref{thm:phi-well-defined-injective-into-spd}, $\psi$ induces a well defined injection from $\W$ into $\cP$. 
In slight abuse of notation we denote this mapping also by $\psi$, that is
\begin{equation}\label{eq:definition-phi-map-ell-wald-space}
    \psi\colon\W\to\spd,\quad \gF=[\gV,\gE,\ell]\mapsto\psi(\gF)\coloneqq\psi(\gV,\gE,\ell).
\end{equation}
\begin{definition}\label{def:wald-space}
    The \emph{\waldspace} is the topological space $\W$ equipped with the unique topology under which the map $\psi\colon\W\to\spd$ from \Cref{eq:definition-phi-map-ell-wald-space}
    is a homeomorphism onto its image.
\end{definition}

\subsection{From a Split Viewpoint}\label{scn:splits}

If $(\gV,\gE,\ell)$ is a representative of a phylogenetic forest $\gF$, there is $K\in \NN$ such that the graph-theoretic  forest $(\gV,\gE)$ decomposes into $K$ disjoint nonempty graph-theoretic trees
$$ (\gV_1,\gE_1),\ldots,(\gV_K,\gE_K)\,.$$
In particular, this decomposition induces a partition $L_1,\ldots,L_K$ of the leaf set $L$ with $L_\alpha \subseteq \gV_\alpha$, $1\leq \alpha \leq K$.

Furthermore for $1\leq \alpha \leq K$, taking away an edge $e \in \gE_\alpha$ decomposes $(\gV_\alpha,\gE_\alpha)$ into two disjoint graph-theoretic trees that \emph{split} the leaf set $L_\alpha$ into two disjoint subsets $A$ and $B$.

The representation of phylogenetic trees via splits is more abstract than as graphs but more tractable. 
We first introduce the weighted split representation and then show equivalence of the concepts.

\begin{definition}\label{def:wald-split-representation}
    A tuple $F = (E,\lambda)$ with $E\neq \emptyset$ is a \emph{split-based phylogenetic forest} if
    \begin{enumerate}[label=(\roman*)]
        \item 
            there is $1\leq K \leq N$ and a partition $L_1,\ldots,L_K$ of the leaf set $L$;
        \item 
            every element $e\in E$ is of the form $e=\{A,B\}$,  called a \emph{split}, where for some $1\leq \alpha\leq K$, $A,B$ is a partition of $L_\alpha$; $E_\alpha$ denotes the elements in $E$ that are splits of $L_\alpha$; for notational ease we write interchangeably 
            $$ e=\{A,B\} = A \vert B = a_1\ldots a_r\vert b_1\ldots b_s = a_1\ldots a_r \vert B = A\vert b_1\ldots b_s\,, $$
            whenever $A=\{a_1,\ldots,a_r\}$, $B=\{b_1,\ldots,b_s\}$;
        \item 
            all splits in $E_\alpha$ ($1\leq \alpha \leq K$) are pairwise \emph{compatible} with one another, where two splits $A\vert B$ and $C\vert D$ of $L_\alpha$ are \emph{compatible} with one another if one of the sets below is empty:
        $$ A\cap C,\quad A\cap D,\quad B \cap C,\quad B\cap D\,;$$ 
        \item 
            for all distinct $u,v\in L_\alpha$, $1\leq \alpha \leq K$, there exists a split $e=A\vert B\in E_\alpha$ such that $u\in A$ and $v\in B$;
        \item 
            $\lambda\coloneqq (\lambda_e)_{e\in E} \in (0,1)^{ E}$. 
    \end{enumerate}
    Moreover $F_\infty$ with $E=\emptyset$ and void array $\lambda$ is the completely disconnected \emph{split-based phylogenetic forest}  with leaf partion $\{1\}, \ldots, \{N\}$.
    
\end{definition}

The partition $L_1,\ldots,L_K$ is not mentioned explicitly in the definition of  a split-based phylogenetic forest $F = (E,\lambda)$ since it can be derived from $E$ via $\{L_1,\dots,L_{\tilde{K}}\} \coloneqq \big\{A\cup B\colon A\vert B \in E\big\}$, where $\tilde{K}\leq K$, and for all $u\in L\setminus\bigcup_{\alpha=1}^{\tilde{K}} L_\alpha$, the singleton $\{u\}$ is added to the collection to obtain $L_1,\dots,L_K$.

\begin{theorem}\label{thm:definitions-of-phylo-forest-equivalent}
    There is a one-to-one correspondence between split-based phylogenetic forests $F = (E,\lambda)$ from \Cref{def:wald-split-representation} and phylogenetic forests $\gF = [\gV, \gE, \ell]$ from \Cref{def:wald-graph-representation} with $\ell$ and $\lambda$ related by
       \begin{equation}\label{eq:lambda-ell}
            \lambda_s \coloneqq 1 - \exp\big(-\ell_e\big) \,.
        \end{equation}
     with an arbitrary but fixed representative  $(\gV,\gE,\ell)$. Furthermore, there is a one-to-one correspondence between compatible split sets $E$ from \Cref{def:wald-split-representation} (i) - (iv), and phylogenetic forest topologies $[\gV, \gE]$.
\end{theorem}

\begin{proof}
   Case I. Suppose $K=1$, i.e. $\gF$ comprises only one tree: We take recourse to \cite[Theorem~3.1.4]{semple_phylogenetics_2003} who establish a one-to-one correspondence of compatible split sets $E$ from \Cref{def:wald-split-representation} (i) - (iv), and phylogenetic forest topologies $[\gV, \gE]$, in case these are taken from graph-theoretic trees. Indeed, our phylogenetic forest topologies correspond to \emph{isomorphic X-trees} there (our $L$ is $X$ there and the \emph{labelling map} from \cite[Definition~2.1.1]{semple_phylogenetics_2003} is the identity in our case) and for every representative $(\gV, \gE)\in [\gV, \gE]$ there is a unique compatible split set $E$ from \Cref{def:wald-split-representation} (i) - (iv) ((iv). is a consequence of $L\subseteq\gV$, $K = 1$ and injectivity of the labelling map). Vice versa, there is a bijection $e \mapsto s_e, \gE \to E$ that, removing the edge $e$ from $\gE$ produces two disconnected trees, yields a unique split $s=s_e=A\vert B$ of the leaf set $L= A\cup B$. This yields the second assertion, namely a one-to-one correspondence between compatible split sets $E$  from \Cref{def:wald-split-representation} (i) - (iv) and phylogenetic forest topologies $[\gV,\gE]$ in case of underlying graph-theoretic trees. The first assertion follows from the correspondence in \eqref{eq:lambda-ell},  which thus yields, due to phylogenetic equivalence in Definition \ref{def:wald-graph-representation} (iii),  a one-to-one correspondence between split based phylogenetic forests $F=(E,\lambda)$ and phylogenetic forests $[\gV,\gE,\ell]$, in case of underlying graph-theoretic trees.

    Case II. Suppose $\gF$ comprises several $K>1$ trees: Here,  consider two phylogenetic forests representatives $(\gV,\gE,\ell),(\gV',\gE',\ell')\in[\gV,\gE,\ell]$.
        Due to \Cref{def:wald-graph-representation}, (i) and (ii), both $(\gV,\gE,\ell)$ and $(\gV',\gE',\ell')$ have the same number of connected components, each of which is a graph-theoretic tree and the bijection $f$ from \Cref{def:wald-graph-representation} restricts to bijections between the corresponding graph-theoretic trees. For each of these, Case I ($K=1$) is applicable, thus yielding the assertion in the general case. 
\end{proof}

In consequence of \Cref{thm:definitions-of-phylo-forest-equivalent} we introduce the following additional notation. 

\begin{definition}\label{def:split-based-advanced}
    From now on, we identify split-based phylogenetic forests $F = (E,\lambda)$ with phylogenetic forests $\gF = [\gV,\gE,\ell]$ and say that $F$ is a \emph{\wald}, in plural \emph{\walds}, so that $F\in \W$, and use interchangeably the name split and edges for the elements of $E$ (as they are ``edges'' in equivalence classes). In particular, the $\lambda_e, e\in E$, from \Cref{def:wald-split-representation}, 5., are called \emph{edge weights}. Furthermore, 
    \begin{enumerate} 
        \item
        $[F] \coloneqq E$ also denotes the \emph{topology} $[\gV,\gE]$ of $F$ and 
    $$\cE := \{E : \exists \lambda \in (0,1)^E\st (E,\lambda) \mbox{ is a split-based phylogenetic forest}\}\cup\{\emptyset\}\,$$
    denotes the set of all possible topologies;
        \item 
        \walds of the same topology $E$ form a \emph{grove}
        \begin{equation*}
            \G{E} =\big\{F=(E',\lambda')\in\F\colon E = E'\big\},
        \end{equation*}
        \item 
        for any two $u,v \in L$ with leaf partition $L_1,\ldots,L_K$, define 
        $$ E(u,v) := \{A\vert B: \exists\, 1\leq \alpha \leq K \mbox{ and } e \in \gE(u,v) \mbox{ that splits } L_\alpha\mbox{ into } A \mbox{ and }B\}\,,$$
        which also denotes set of edges between $u$ and $v$, that may be empty;
        \item 
        the edge length based matrix representation  $\psi$ from \Cref{eq:definition-phi-map-ell-wald-space} translates to the \emph{edge weight} based matrix representation $\phi$ defined by
        \begin{equation}\label{eq:definition-phi-map-lambda}
            \phi\colon\W\to\spd,\quad 
            F = (E,\lambda) \mapsto (\rho_{uv})_{u,v=1}^N 
            \coloneqq \bigg(\prod_{e\in E(u,v)} \big(1 - \lambda_e\big)\bigg)_{u,v=1}^{N},
        \end{equation}
        with the agreement 
        that in case of empty $E(u,v)$
        \begin{align}\label{eq:phi-map-agreement}
            \left. 
            \begin{array}{rl}
                &\rho_{uv} \coloneqq 1 \text{ whenever } u = v \text{ and }\\
                &\rho_{uv} \coloneqq 0 \text{ whenever } u\in L_\alpha \text{ and } v\in L_\beta\text{, }\alpha\neq\beta\text{, }\alpha,\beta\in\{1,\dots,K\}
            \end{array}
            \right\}\,;
        \end{align}
        here  $\lambda$ is computed from $\ell$ as defined in \Cref{eq:lambda-ell}.
    \end{enumerate}
\end{definition}

\begin{remark}\label{rmk:grove-rep}

        In light of \Cref{def:wald-space}, the \emph{\waldspace} is the topological space $\cW$ equipped with the unique topology such that the map $\phi\colon\cW\to\spd$ is a homeomorphism onto its image.  
        Thus, groves can be identified topologically with open unit cubes
                \begin{equation}\label{eq:grove-cube-identification}
            \G{E} \cong (0,1)^E\,
        \end{equation}
        and the \emph{\waldspace}  thus with the disjoint union 
    \begin{eqnarray}\label{eq:wald-cubes-identification}
    \W = \bigsqcup_{E \in \cE} \G{E} \cong \bigsqcup_{E \in \cE} (0,1)^E\,,
    \end{eqnarray}
    where we note that $\vert E\vert $ runs from $0$ (corresponding to $F_\infty$) to $2N-3$ (for fully resolved trees), as is easily seen upon induction on $N$.
        

        Furthermore, observe that

\begin{enumerate}
 \item 
\Cref{eq:lambda-ell}  links strictly monotonous edge weights with edge lengths so that the limits $\lambda_e \to 0,1$ correspond to the limits $\ell_e \to 0,+\infty$, respectively;

\item for any partition $A,B$ of $L_\alpha$ ($1\leq \alpha\leq K$, as above), we have that 
\begin{equation}\label{eq:splits-on-path-equiv}
    e = A\vert B\iff  e\in E(u,v) \mbox{ for all }u\in A, v\in B\,,
\end{equation}
where the implication to the right is a consequence of $E_\alpha$ being a tree topology and the reverse implication is a consequence of $A$ and $B$ being a partition of $L_\alpha$.
\end{enumerate}

\end{remark}

\begin{example}\label{exa:phylogenetic-forest-topology}
    Let $N = 6$.
    Consider
    \begin{equation*}
        E = \big\{1\vert 234, 3\vert 124, 4\vert 123, 12\vert 34, 5\vert 6\big\},
    \end{equation*}
    i.e.~the partition of labels is $L_1 = \{1,2,3,4\}$, $L_2 = \{5,6\}$,
    the corresponding graph is depicted in \Cref{fig:example-phylogenetic-forest-topology}.
    One can easily check that all edges that are splits of $L_1$ are compatible, likewise for all splits of $L_2$ (there is only one split, $5\vert 6$, in this case). 
    
    Moreover, observe that the unique path from $1$ to $4$ contains the edges $E(1, 4) = \big\{1\vert 234, 4\vert 123, 12\vert 34\big\}$, i.e. all splits separating $1$ and $4$.
    
    Indeed, for every connected pair of leaves, there is a split separating this pair, for instance for all $u,v\in L_1$ there is a split $e = A\vert B\in E$ such that $u\in A$ and $v\in B$. Removing the edge $1\vert 234$ from the subtree comprising the leaf set $L_1$ violates this condition:  If there is no split separating $1$ and $2$, which remain connected, then one vertex is labelled twice with $1$ and $2$.  \cite[e.g. Section 3.1.]{semple_phylogenetics_2003} allow such trees, we, however, exclude them.

\end{example}

\begin{figure}[ht]
    \centering
    \includegraphics[width=0.4\linewidth]{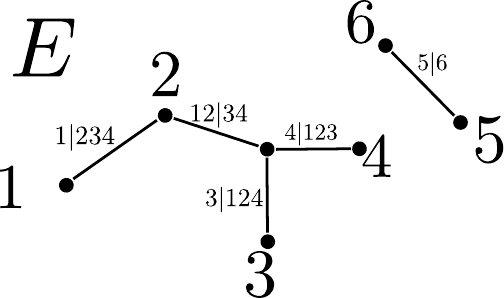}
    \caption{The topology $E$ as defined in \Cref{exa:phylogenetic-forest-topology} with the corresponding splits annotated to the edges.}
    \label{fig:example-phylogenetic-forest-topology}
\end{figure}

\section{Topology and Stratification of Wald Space}
\label{scn:topology-of-waldspace}

\subsection{Embedding}\label{scn:inequalities}
Recall from \Cref{thm:phi-well-defined-injective-into-spd} that $\psi:\W \to \cP$ from \Cref{eq:definition-phi-map-ell} is injective and so is the equivalent $\phi:\W \to \cP$ from \Cref{eq:definition-phi-map-lambda}. Its image is characterized by algebraic equalities and inequalities, as shown by the the following theorem. 
Further exploration will yield that the topology of \waldspace is that of a stratified union of disjoint open unit cubes, each corresponding to a grove from  \Cref{def:split-based-advanced}.

\begin{theorem}\label{thm:phi-spd-repr}
    A matrix $P=(\rho_{uv})_{u,v=1}^N\in\spd$ is the $\phi$-image of a \wald $F \in \F$  if and only if all of the following conditions are satisfied for arbitrary $u,v,s,t\in L$: 
    \begin{enumerate}[leftmargin=0.85cm, start=1, label={\textnormal{(R\arabic*)}}]\itemsep0.3em
      \item \label{item:diagonal-is-1}
      $\rho_{uu} = 1$,
      \item \label{item:4pt-condition} two of the following three are equal and smaller than (or equal to) the third
      \begin{equation*}
        \rho_{uv}\rho_{st},\quad \rho_{us}\rho_{vt},\quad \rho_{ut}\rho_{vs}\,,
      \end{equation*}
      \item \label{item:entries-geq-0}
      $\rho_{uv} \geq 0$.
    \end{enumerate}
    Furthermore, the \wald $F \in \F$ is then uniquely determined.
\end{theorem}

Before proving \Cref{thm:phi-spd-repr}, we elaborate on the above algebraic conditions.

\begin{remark}\label{rmk:phi-spd-repr}
    \begin{enumerate}
    \item 
    Condition \ref{item:4pt-condition} above is called the \emph{four-point-condition}. 
    In its \emph{non-strict} version, all three products are equal and this indicates some degeneracy, namely that some internal vertices have degree four or higher.
    The four-point-condition is equivalent to (e.g.~\cite{buneman_note_1974} or \cite[p.147]{semple_phylogenetics_2003})
    \begin{equation}\label{eq:four-point-condition-equiv-inequality}
        \rho_{uv}\rho_{st} \geq \min\big\{\rho_{us}\rho_{vt},\;\rho_{ut}\rho_{vs}\big\}
    \end{equation}
    and implies (e.g.~setting $s = t$ in \ref{item:4pt-condition} and exploiting \ref{item:diagonal-is-1}) 
    \begin{enumerate}[leftmargin=0.85cm, start=4, label={\textnormal{(R\arabic*)}}]\itemsep0.35em
      \item \label{item:triangle}
      $\rho_{uv} \geq \rho_{us}\rho_{sv}$ for all $u,v,s\in L$.
    \end{enumerate}  
    Notably \ref{item:diagonal-is-1} and \ref{item:4pt-condition} imply, in conjunction with $P\in \spd$ that
    \begin{enumerate}[leftmargin=0.85cm, start=5, label={\textnormal{(R\arabic*)}}]\itemsep0.35em
      \item \label{item:R5}
      $\rho_{uv} < 1$ for all $u\neq v$,
    \end{enumerate}
    for otherwise, if $\rho_{uv} = 1$ for some $u\neq v$, Condition \ref{item:triangle} implied for any $s\in L$ that
    $$\rho_{us}\geq \rho_{uv}\rho_{vs} = \rho_{vs} \quad\text{ and }\quad \rho_{vs}\geq \rho_{uv}\rho_{us} = \rho_{us},$$
    so $\rho_{us}=\rho_{vs}$ and hence, $P$ would be singular, a contradiction to $P \in \spd$.
    \item
    Observe that $\phi(F) = (\exp(-d_{uv}))_{u,v=1}^N$, where the $d_{uv}$ are the finite or infinite distances
    \begin{equation*}
        d_{uv} \coloneqq  \sum_{e \in E(u,v)} \ell_e = - \log \rho_{uv},
    \end{equation*}
    between leaves $u,v\in L$, and, with \Cref{def:split-based-advanced} 5., this translates to $d_{uu} = 0$ and $d_{uv} = \infty$ whenever $u$ and $v$ are in different components.
    In the literature, $(d_{uv})_{u,v=1}^N$ is also called \emph{tree metric} (e.g.~\cite[Chapter~7]{semple_phylogenetics_2003}) or distance matrix (e.g.~\cite[Chapter~11]{felsenstein_inferring_2003}).
    Indeed, it conveys a metric on $L$ as Condition \ref{item:triangle} encodes the \emph{triangle inequality} (for any $u,v,s\in L$)
    \begin{equation*}
        d_{uv} \leq d_{us} + d_{vs}.
    \end{equation*}
    \item In particular, the unit $N\times N$ matrix $I = (\delta_{uv})_{u,v\in L} \in \cP$ is the $\phi$-image of the complete disconnected \wald $F_\infty\in \W$ with topology $E_\infty = \emptyset$ in which each leaf comprises one of the $K=N$ single element trees.
    \item For a given $P\in \cP$ satisfying conditions \ref{item:diagonal-is-1}, \ref{item:4pt-condition} and \ref{item:entries-geq-0} there are neighbour joining algorithms in \cite[Scn~7.3]{semple_phylogenetics_2003}, determining its split  $E\in \cE$.
    \end{enumerate}
\end{remark}

\begin{proof}[Proof of \Cref{thm:phi-spd-repr}.]
    ``$\Longrightarrow$''. 
    Let $F\in\F$ and $(\rho_{uv})_{u,v=1}^N = \phi(F)$. \ref{item:diagonal-is-1} and \ref{item:entries-geq-0} hold by definition. Further, applying \citet[Theorem~7.2.6]{semple_phylogenetics_2003} to each connected component asserts  \ref{item:4pt-condition} for all $u,v,s,t\in L_\alpha$ for all $\alpha=1,\dots,K$.
    Furthermore, since $\rho_{uv}=0$ whenever $u\neq v$ are in different components, for any $s\in L\setminus\{u,v\}$, $\rho_{us} = 0$ or $\rho_{vs} = 0$, so \ref{item:4pt-condition} holds true in general.

    ``$\Longleftarrow$''. 
    Let $P = (\rho_{uv})_{u,v=1}^N\in\spd$ satisfy \ref{item:diagonal-is-1}, \ref{item:4pt-condition} and \ref{item:entries-geq-0} (and thus by \Cref{rmk:phi-spd-repr} also \ref{item:triangle} and \ref{item:R5}).
    The equivalence relation on $L$, defined by $u\sim v \iff \rho_{uv}\neq 0$ partitions $L$ into $L_1,\dots,L_K$ for some $K\in\{1,\dots,N\}$.
    For each $\alpha=1,\dots,K$, apply \citet[Theorem~7.2.6]{semple_phylogenetics_2003} to each tree metric $(d_{uv})_{u,v\in L_\alpha}$ (defined in \Cref{rmk:phi-spd-repr} 2.) to obtain a unique corresponding tree, say $[\gV_\alpha, \gE_\alpha, \ell^{(\alpha)}]$, where, in contrast to our definition, \cite{semple_phylogenetics_2003} allow leaves on top of each other, in their language, vertices \emph{labelled} more than once.
    The union of trees gives a forest $\gF = [\gV, \gE, \ell]$ with label set $L$ with $\gV = \bigcup_{\alpha} \gV_\alpha$, $\gE = \bigcup_{\alpha} \gE_\alpha$, satisfying $\psi(\gF) = P$.
    Suppose now a vertex was labelled more than once, say, with distinct leaf labels $u,v\in L_\alpha$, i.e. $u\neq v$, for some $\alpha=1,\dots,K$. Then, $u$ and $v$ have zero distance $d_{uv}$, hence $\rho_{uv}=1$, yielding a contradiction to \ref{item:R5} (i.e. $P \not \in \cP$ as argued in \Cref{rmk:phi-spd-repr} 1.)
    Thus $\gF\in\F$ and with \Cref{def:split-based-advanced}, we obtain $F\in\F$ with $\phi(F) = \psi(\gF)=P$.
\end{proof}

Since $\phi(\F)$ is defined by algebraic equalities and nonstrict inequalities, we have the following corrollary at once. 

\begin{corollary}\label{cor:forests-closed-in-spd}
    $\phi(\F)\subseteq\spd$ is a closed subset of $\spd$.
\end{corollary}

\begin{example}[$\F$ for $N = 3$]
\label{exa:forests-n3-in-spd}
For $N = 3$, all matrices $P = \phi(F)$ with $F\in\F$ are given by (using \Cref{thm:phi-spd-repr})
\begin{equation*}
    \begin{pmatrix}
        1 & \rho_{12} & \rho_{13}\\
        \rho_{12} & 1 & \rho_{23}\\
        \rho_{13} & \rho_{23} & 1
    \end{pmatrix}
    \quad\text{satisfying the triangle inequalities}\quad \left\{
    \begin{array}{c}
    \rho_{12}\geq\rho_{13}\rho_{23},\\
    \rho_{13}\geq\rho_{12}\rho_{23},\\
    \rho_{23}\geq\rho_{12}\rho_{13},
    \end{array}\right.
\end{equation*}
and $0 \leq \rho_{12},\rho_{13},\rho_{23} < 1$.
This set in coordinates $\rho_{12},\rho_{13},\rho_{23} $ is depicted in \Cref{fig:forests-n3-tetrahedron}, where the two-dimensional surfaces correspond to the non-linear boundaries resulting from the triangle inequalities.
Note that the regions, where at least one coordinate is one, are not included in $\phi(\F)$, as the corresponding matrix is no longer strictly positive definite.
\end{example}

\begin{figure}[ht]
  \centering
    \begin{minipage}{0.99\linewidth}
        \includegraphics[width=\linewidth]{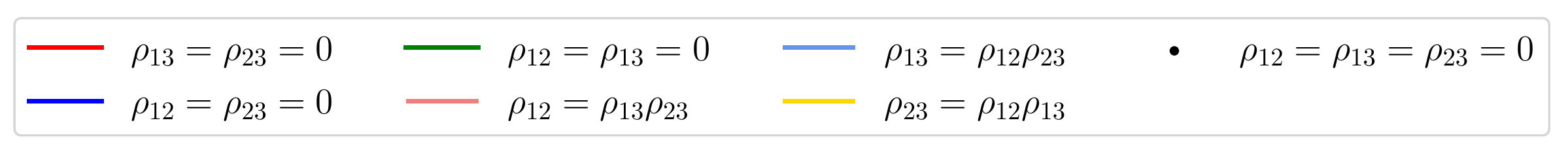}
    \end{minipage}
    \begin{minipage}{0.99\linewidth}
      \centering
      \begin{minipage}{0.32\linewidth}
        \includegraphics[width=\linewidth]{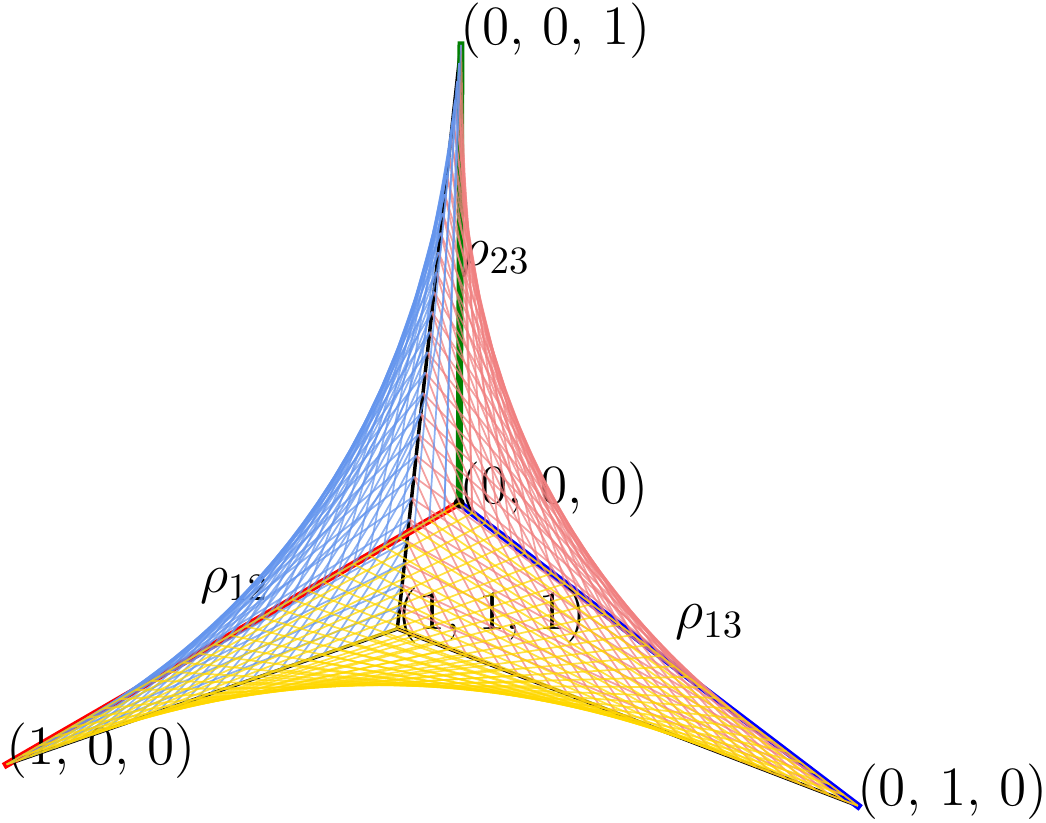}
      \end{minipage}
      \begin{minipage}{0.32\linewidth}
        \includegraphics[width=0.85\linewidth]{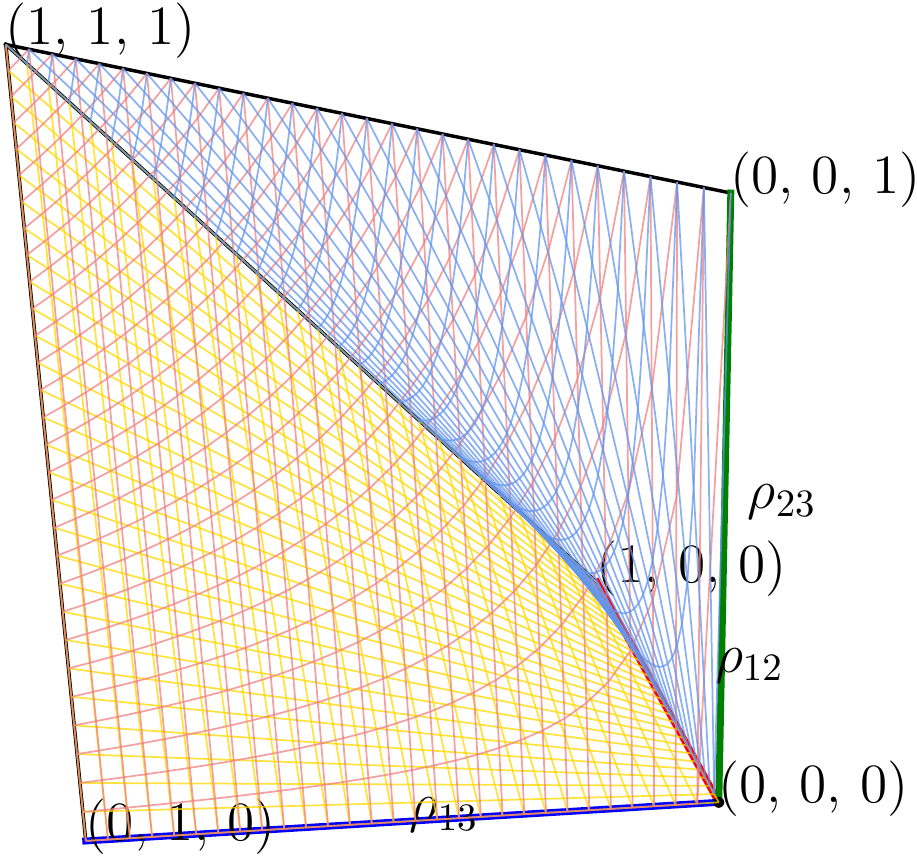}
      \end{minipage}
      \begin{minipage}{0.32\linewidth}
        \includegraphics[width=\linewidth]{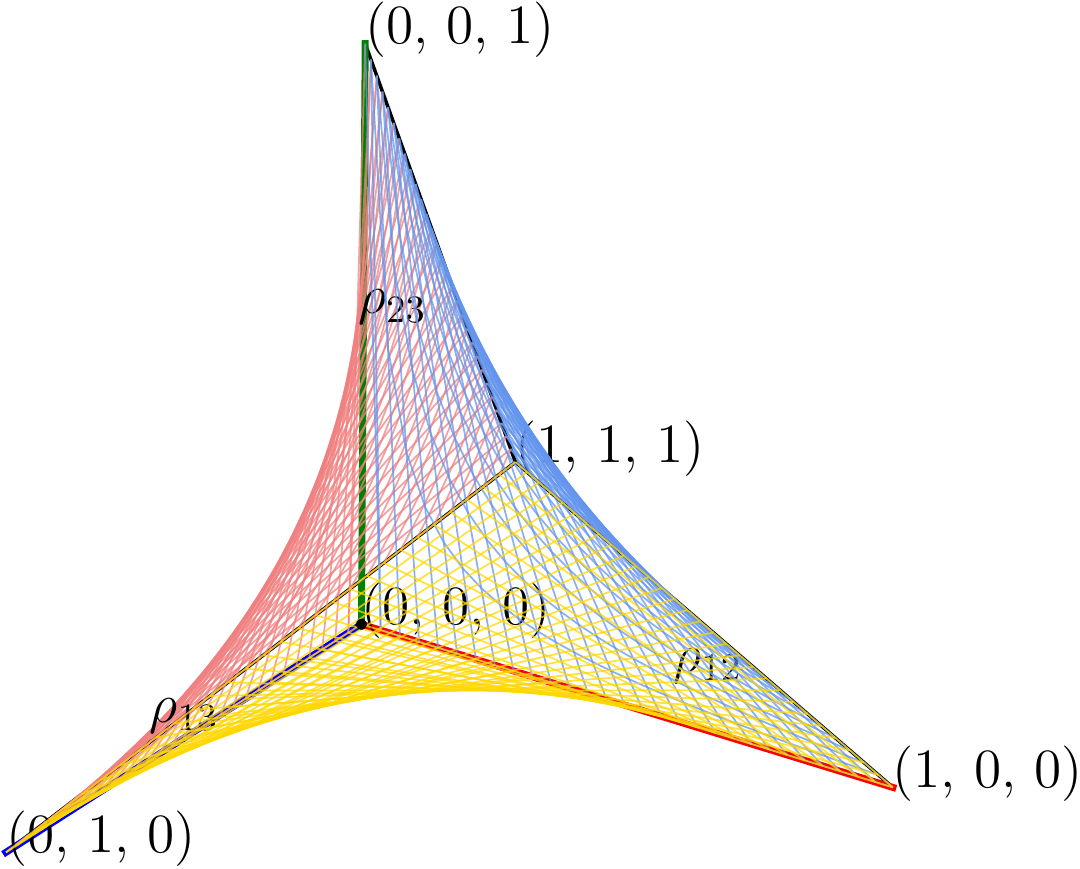}
      \end{minipage}
    \end{minipage}
  \caption{
  $\phi(\F)$ for $N = 3$ embedded in $\spd$, where only the off-diagonal entries on the boundary are depicted.
  Note that the geometry of the \waldspace is not Euclidean and thus this depiction may be deceiving (as it is a non-isometric embedding into $\RR^3$), e.g.~the regions where one coordinate equals 1 are infinitely far away.}
  \label{fig:forests-n3-tetrahedron}
\end{figure}

\begin{corollary}\label{cor:forests-are-contractible}
    Conveyed by the homeomorphism $\phi$, $\W$ is star shaped as a subset $\RR^{N\times N}$ with respect to $F_\infty$ and hence contractible.
\end{corollary}

\begin{proof} Let $F\in \F$ with $\phi(F) = P = (\rho_{uv})_{u,v\in L}$ satisfying \ref{item:diagonal-is-1} - \ref{item:entries-geq-0} by Theorem \ref{thm:phi-spd-repr}. 
Recalling from \Cref{rmk:phi-spd-repr}, 3., that $\phi(F_\infty)=I$, consider
$$ (\rho_{uv}^{(x)})_{u,v\in L} =  P^{(x)} = x\,I + (1-x) P,$$
and observe that for all $x\in [0,1]$, $P^{(x)} \in \spd$, $\rho^{(x)}_{uu} = 1 = \rho_{uu}$ for all $u\in L$ and $\rho^{(x)}_{uv} = (1-x) \rho_{uv} \geq 0$ for all $u,v\in L$ with $u\neq v$, i.e. $P^{(x)}$ satisfies \ref{item:diagonal-is-1} and \ref{item:entries-geq-0} for all $x\in[0,1]$. 
Moreover,  to see that $P^{(x)}$ satisfies \Cref{eq:four-point-condition-equiv-inequality} for all $x\in (0,1)$ for all $u,v,s,t\in L$, assume w.l.o.g that
\begin{align}\label{eqn:4pt-condition-star-shaped-1} \rho_{uv}\rho_{st} &= \rho_{us}\rho_{vt}\leq \rho_{ut}\rho_{vs}\,.
\end{align}
If all four  $u,v,s,t$ are pairwise distinct then
\begin{align}\label{eqn:4pt-condition-star-shaped-2} \rho^{(x)}_{uv}\rho^{(x)}_{st} &= \rho^{(x)}_{us}\rho^{(x)}_{vt}\leq \rho^{(x)}_{ut}\rho^{(x)}_{vs}\,,
\end{align}
as well. If only one pair is equal, there are two typical cases. If $u=v$, say, we obtain a different but valid four point condition
$$ \rho^{(x)}_{uv}\rho^{(x)}_{st} ~\geq~ \rho^{(x)}_{us}\rho^{(x)}_{vt} ~=~ \rho^{(x)}_{ut}\rho^{(x)}_{vs}\,,$$
where the inequality is strict in case of $\rho_{st} >0$ due to $1 - x > (1 - x)^2$. If $u=t$, say, then we obtain \Cref{eqn:4pt-condition-star-shaped-2} where the inequality is strict if $\rho_{vs}>0$. If exactly two pairs are the same, then, with the above setup only $u=t$ and $v=s$ is possible and both \Cref{eqn:4pt-condition-star-shaped-1}  and \Cref{eqn:4pt-condition-star-shaped-2} are  strict. In case of three equal indices, one different, or the same, \Cref{eqn:4pt-condition-star-shaped-2} holds again.
Therefore, $P^{(x)}$ satisfies \ref{item:4pt-condition} for all $x\in[0,1]$, and by Theorem \ref{thm:phi-spd-repr} the entire continuous path $x\mapsto P^{(x)}, [0,1] \to \spd$ corresponds to a path $F^{(x)}\coloneqq  \phi^{-1}(P^{(x)}) \in \F$, connecting $F= F^{(0)}$ with $F_\infty = F^{(1)}$ as asserted.
\end{proof}

Showing contractibility of the edge-product space, Moulton and Steel contract to the same forest (cf. \cite[Proposition~5.1]{moulton_peeling_2004}), employing a different proof, however.

\begin{remark}\label{rmk:4pt}
We make the following observations about the proof of \Cref{cor:forests-are-contractible}.
\begin{enumerate}
        \item 
    All of the \walds $\phi^{-1}(P^{(x)})$, for $0 \leq x < 1$  in the proof share the same partition of leaves into connected tree  components, due to $\rho_{uv}\neq0\iff(1-x)\rho_{uv}\neq0$ for all $x\in[0,1)$ for all $u,v\in L$.
   
    \item 
    For $0<x<1$,  $P^{(x)}$ satisfies unchanged, strict or nonstrict four-point conditions \ref{item:4pt-condition}, that may be different, though, from those of $P^{(0)} = \phi(F)$. 
    \item All triangle inequalites \ref{item:triangle} involving initial nonzero $\rho_{uv}$ are strict, however, for $0< x < 1$, so that for $\phi^{-1}(P^{(x)})$ none of the leaves have degree 2.
    For example, starting with the \wald consisting of a chain of three vertices with $N = 3$ (so each vertex is labelled and the middle is of degree two), it is immediately transformed into a fully resolved tree (and stays one for all $x\in(0,1)$).
    \item 
    The point $F_\infty$ can be viewed as a vantage point of $\W$ which is then a bounded part of a cone where every 
    \begin{equation*}
    B_a = \Big\{F\in\W \;\Big\vert\; \phi(F) = (\rho_{uv})_{u,v=1}^N,~a = 1 - \prod_{\substack{u,v=1\\u<v}}^{N} \big(1 - \rho_{uv}\big)\Big\}.
    \end{equation*} 
    is a \emph{slice} of \emph{level} $a\in[0,1)$. Then for every $F\in B_a $, there is $r_F >1$ such that
    $$ F = \phi^{-1}\left((1-x) \phi(F_\infty) +x \phi(F) \right) \in \W$$
    for all $0\leq x < r_F$ and $\phi(F)$ is singular for $x=r_F$.
        For $N = 3$, the set $B_a$ for several $a\in(0,1]$ embedded into $\cP$ is depicted in Fig.~\ref{fig:cone-waldspace-slices}. 

\end{enumerate}
\end{remark}
  
\begin{figure}[ht]
    \centering
    \begin{minipage}{0.9\linewidth}
        \includegraphics[width=\linewidth]{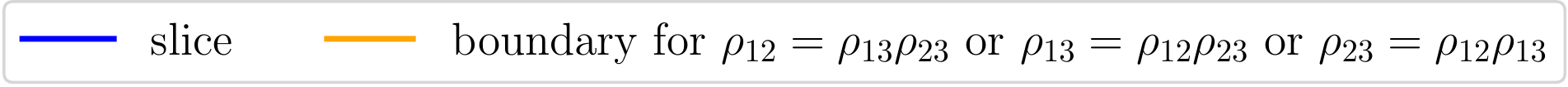}
        \vspace{0.1cm}
    \end{minipage}
    \begin{minipage}{0.99\linewidth}
      \centering
      \begin{minipage}{0.32\linewidth}
        \includegraphics[width=\linewidth]{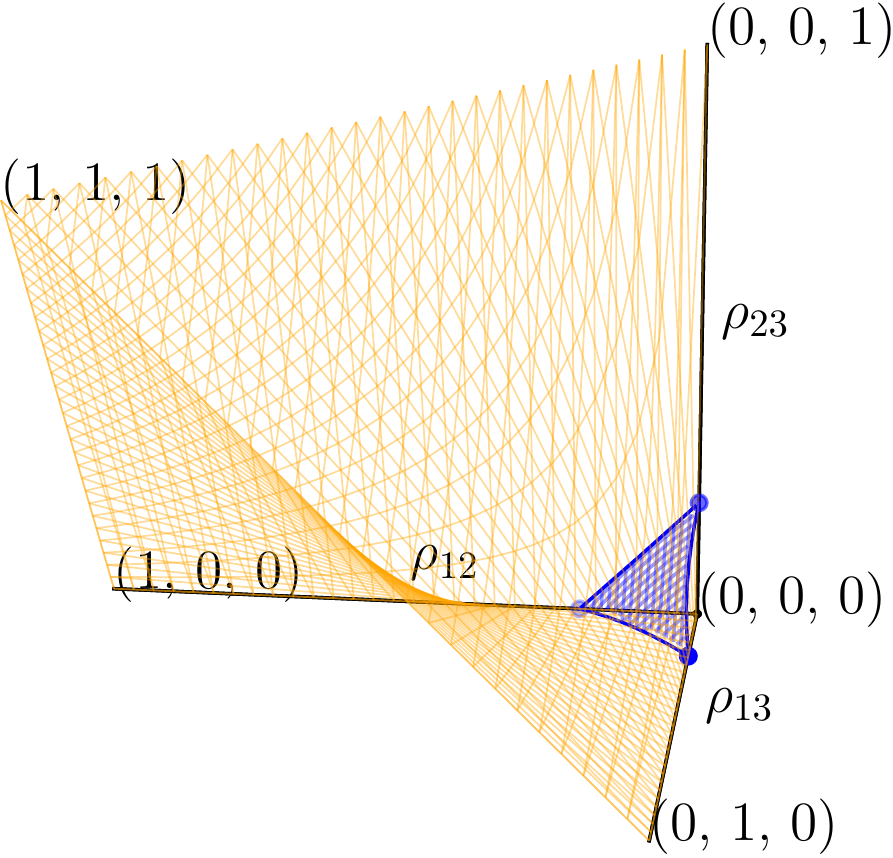}
      \end{minipage}
      \begin{minipage}{0.32\linewidth}
        \includegraphics[width=\linewidth]{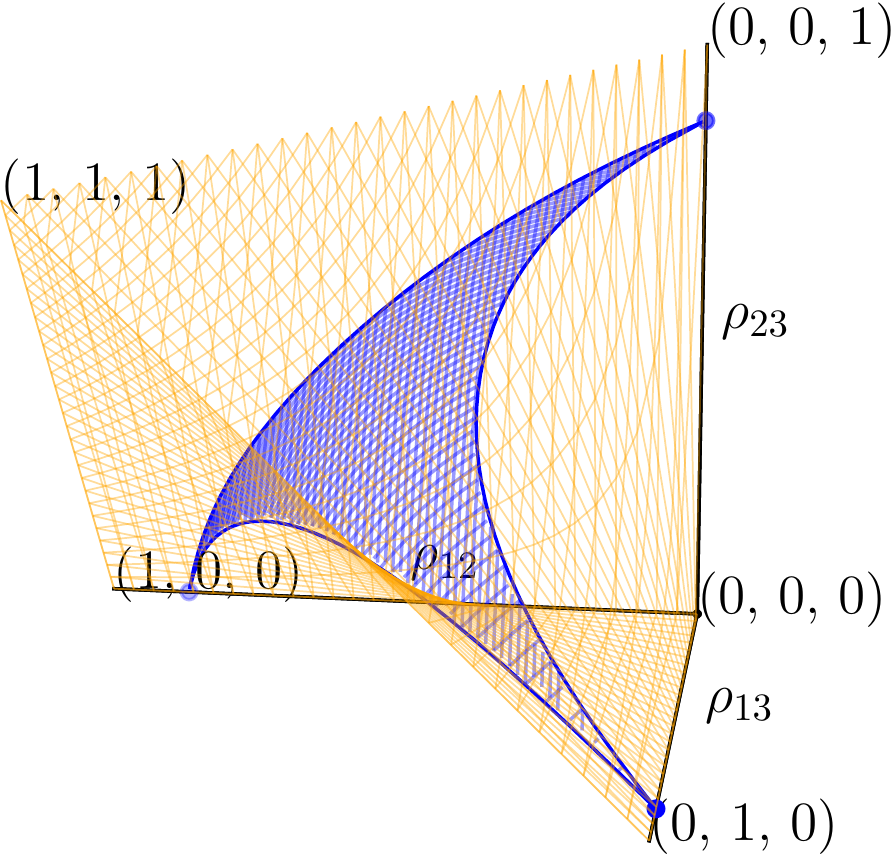}
      \end{minipage}
      \begin{minipage}{0.32\linewidth}
        \includegraphics[width=\linewidth]{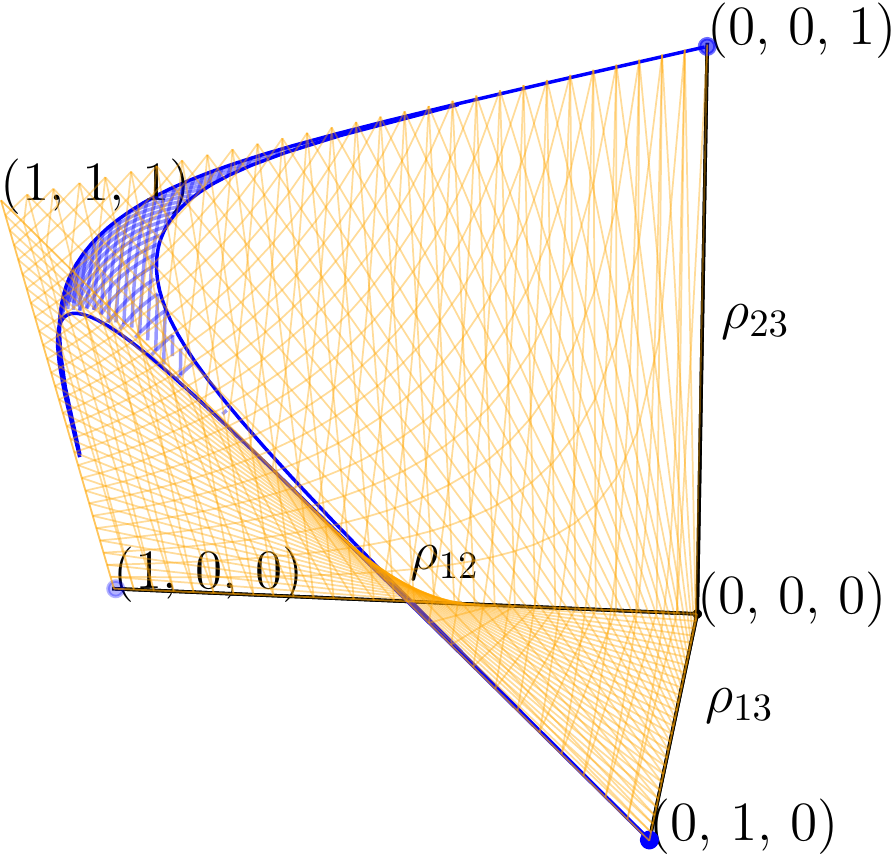}
      \end{minipage}
    \end{minipage}
    \caption{Depicting the off-diagonal matrix entries of
      $\phi(\F)$ embedded in $\spd$ for $N = 3$ (orange boundary) in a 3-dimensional coordinate system (cf.~Example~\ref{exa:forests-n3-in-spd}) and the 2-dimensional images $\phi(B_a)$ (purple) of the slices $B_a$ for $a=0.2, 0.87, 0.997$ (from left to right).
    }
    \label{fig:cone-waldspace-slices}
\end{figure}

We next consider the restriction of the map $\phi$ to each grove $\G{E}$ explicitly in terms of edge weights. 

\begin{definition}\label{rmk:grove-cube-identification}
        With the agreement (\ref{eq:phi-map-agreement}) in case of empty $E(u,v)$, we denote the restriction of $\phi\colon\F\to\spd$ from  \Cref{def:split-based-advanced} to a grove $\G{E}$ by 
        \begin{equation}\label{eq:phi-map-restricted-to-grove}
            \phi_E\colon(0,1)^E\to\spd,\qquad \lambda \mapsto (\rho_{uv})_{u,v\in  L} = \Bigg(\prod_{e\in E(u,v)} \big(1 - \lambda_e\big)\Bigg)_{u,v=1}^{N};
        \end{equation}
          its continuation onto all of $\RR^E$ is denoted by
    \begin{equation}\label{eq:phi-map-extended-analytic}
        \bar\phi_E\colon\RR^E\to\sym,\qquad \lambda \mapsto (\rho_{uv})_{u,v\in  L} = \Bigg(\prod_{e\in E(u,v)} \big(1 - \lambda_e\big)\Bigg)_{u,v=1}^{N},
    \end{equation}
    \end{definition}

    \begin{remark}\label{rmk:phi-analytic-cont}
         The continuation $\bar\phi_E$ is multivariate real analytic on all of $\RR^E$.
     \end{remark}

The following theorem characterizes each grove.
\begin{theorem}\label{lem:grove-topology-eucl-phi-is-embedding}
    \begin{enumerate}
    \item For $F=(E,\lambda) \in \cW$ with $\phi(F) = (\rho_{uv})_{u,v\in L}$ we have
    $$ \lambda_e = 1 - \max_{\substack{u,v\in A\\s,t\in B}} \sqrt{\frac{\rho_{ut}\rho_{vs}}{\rho_{uv}\rho_{st}}}\,,\mbox{ for all } e=A\vert B \in E\,.$$
    \item The derivative of $\phi_E$ has full rank $\vert E\vert$ throughout $(0,1)^E$. 
    \item The map $\phi_E\colon(0,1)^E\cong\G{E}\to\spd$ is a smooth embedding.
\end{enumerate}
\end{theorem}

\begin{proof}
For the first assertion consider $e=A\vert B$, where $A\cup B= L_\alpha$, for some  $1\leq \alpha \leq K$ and where $L_1,\ldots,L_K$ is the leaf partition induced by $E$. Then the matrix entries $d_{uv} \coloneqq -\log\,\rho_{uv}$ ($u,v\in L_\alpha$) define a metric on $L_\alpha$, as noted in \Cref{rmk:phi-spd-repr}. For such a metric, 
 \cite[Lemma~8]{buneman_recovery_1971} asserts that one can assign a tree $(\gV_\alpha,\gE_\alpha,\ell^{\alpha})$ where
\begin{equation}\label{eq:length-min-char}
    \ell^{\alpha}_e = \min_{\substack{u,v\in A\\s, t\in B}}\frac{1}{2}\big(d_{ut} + d_{vs} - d_{uv} - d_{st}\big),
\end{equation}
which is uniquely determined by \cite[Theorem~2]{buneman_recovery_1971}. Due to our uniqueness results from \Cref{thm:definitions-of-phylo-forest-equivalent} and \Cref{thm:phi-spd-repr}, due to \Cref{eq:lambda-ell}, $\lambda_e = 1 - \exp(-\ell^{\alpha}_e)$ and hence, using $\rho_{uv} = \exp(-d_{uv})$, the asserted equation follows at once from \Cref{eq:length-min-char}.

For the second assertion, let $e\in E$ and suppose that $F=(E,\lambda)$ decomposes into $K$ subtrees inducing the leaf partition $L_1,\ldots, L_K$. Using \Cref{eq:phi-map-restricted-to-grove}, if either $u,v \in L$ are in different subtrees or $u=v$, then
$$ \bigg(\frac{\dd\phi_E}{\dd\lambda_e}(\lambda)\bigg)_{uv}
        = 0\,.
$$
Else, if $u,v\in L_\alpha$ for some $1\leq \alpha \leq K$, then $\rho_{uv} >0$ and with  the Kronecker delta $\delta$,
    \begin{equation}\label{eq:partial-derivative-of-embedding}
        \bigg(\frac{\dd\phi_E}{\dd\lambda_e}(\lambda)\bigg)_{uv}
        =
        -\delta_{e\in E(u,v)}\prod_{\substack{\tilde{e}\in E(u,v)\\\tilde{e}\neq e}} \big(1 - \lambda_{\tilde{e}}\big)
        = - \frac{\rho_{uv}}{1 - \lambda_e} \delta_{e\in E(u,v)}\,,
    \end{equation}
    Thus, for every $x \in \RR^E$, we have 
    \begin{equation*}
        \Big((\dr\phi_E)_\lambda(x)\Big)_{uv} = -\rho_{uv}\,\sum_{e\in E} \frac{x_e}{1 - \lambda_e} \delta_{e\in E(u,v)}\,,
    \end{equation*}
    so that    $\Big((\dr\phi_E)_\lambda(x)\Big)_{uv}=0$ implies
    \begin{equation}
        \label{eq:proof-full-rank-phi}
         0 =\sum_{e\in E(u,v)} \frac{x_e}{1 - \lambda_e} \eqqcolon h_{uv}.
    \end{equation}
    We now view each of the $\ell'_e := \frac{x_e}{1 - \lambda_e}$, $e\in E$  as a real valued ``length'' of $e$. With a representative $(\gV,\gE)$ of $E$ with leaf set partition $L_1,\ldots,L_K$, for every $e\in E$ there are  $v_1,v_2\in \gV_\alpha$ with suitable $1 \leq \alpha \leq K$ such that $e$ corresponds to $\{v_1,v_2\}\in \gE_\alpha$. In particular, since $(\gV_\alpha,\gE_\alpha)$ is a tree, there are  $u,v,s,t \in L_\alpha$ (not necessarily all of them distinct), such that 
        \begin{equation*}
        \ell'_e = \frac{1}{2}\big(h_{uv} + h_{st} - h_{ut} - h_{vs}\big).
    \end{equation*}
    If the r.h.s. is zero due to \Cref{eq:proof-full-rank-phi}, then $x_e =0$, yielding that $(\dr\phi_E)_\lambda$ has full rank, as asserted.
    
    The third assertion follows directly from 1.~and 2., i.e.~$\phi_E$ is bijectively smooth onto its image and its differential is injective.
\end{proof}

In the following, we are concerned with $\bar\phi_E(\lambda)$ if $\lambda\in(0,1)^E$ approaches the boundary.
The next result characterises exactly under which conditions $\bar\phi_E(\lambda)$ stays in the image $\phi(\W)$ of \waldspace under $\phi$. 

\begin{lemma}\label{lem:boundary-of-cube-when-spd}
  Let $F\in\F$ with topology $[F] = E$ and let $\lambda^*\in\dd([0,1]^E)$ with $\bar\phi_E(\lambda^*) = (\rho_{uv}^*)_{u,v=1}^N$.
  Then
  \begin{equation*}
      \bar\phi_E(\lambda^*)\in\phi(\F) \iff \bar\phi_E(\lambda^*) \in \spd \iff \rho_{uv}^* < 1 \text{ for all } u,v\in L \text{ with } u\neq v.
  \end{equation*}
\end{lemma}

\begin{proof}
    The first equivalence follows from that \Cref{eq:boundary-of-grove} is well-defined. 
    We prove the second equivalence.\\
    ``$\Rightarrow$'': Follows from \Cref{rmk:phi-spd-repr}, Condition (R5).\\
    ``$\Leftarrow$'': Analogously to the proof of \Cref{thm:phi-spd-repr}, ``$\Leftarrow$'', we find a phylogenetic forest in the sense of \cite[Chapter 2.8]{semple_phylogenetics_2003}, whose tree metric coincides with the one obtained from $\bar\phi_E(\lambda^*)$, but there might be multiply labelled vertices. 
    However, this is impossible due to $\rho_{uv}^* < 1$ for any $u\neq v$, which is equivalent to a distance greater than zero between $u$ and $v$.
    Therefore, there exists a phylogenetic forest $F'\in\F$ with $\phi(F') = \bar\phi_E(\lambda^*)$, and thus by \Cref{thm:phi-spd-repr}, $\bar\phi_E(\lambda^*)\in\spd$.
\end{proof}

The previous result immediately shows which matrices in $\spd$ form the boundary of a \grove.

\begin{corollary}
    Let $E$ be a \waldtop. Then the {boundary of the \grove} $\G{E}$ in $\W$ is given by
    \begin{equation}\label{eq:boundary-of-grove}
        \dd \G{E} = \Big\{\phi^{-1}\big(\bar\phi_E(\lambda^*)\big) \colon \lambda^*\in \dd([0,1]^E),\;\bar\phi_E(\lambda^*)\in\spd \Big\}\,.
    \end{equation}
\end{corollary}

The following result gives a first glimpse on how different groves are connected through the convergence of \walds.
\begin{theorem}\label{thm:wald-convergence-general}
    Let $\W \ni (E_n,\lambda^{(n)}) = F_n \to F' = (E',\lambda') \in \W$. 
    Then there is a subsequence $n_k$, $k\in \NN$, and a common topology $E$ such that $E_{n_k} = E$ for all $k\in \NN$. 
    Furthermore
    \begin{enumerate}
        \item $\lambda^{(n_k)}$ has a cluster point $\lambda^* \in [0,1]^E$,
        \item and $\phi(F') = \bar \phi_E(\lambda^*)$ for every of such cluster point $\lambda^* \in [0,1]^E$,
        \item and $F'\in\dd\G{E}$ whenever $E\neq E'$.
    \end{enumerate}
\end{theorem}

\begin{proof}
    For the first assertion, noting that there are only finitely many \waldtops, there needs to exist a subsequence $F_{n_k}$ of $F_n$ with $E_{n_k} = E$ for some topology $E$ for all $k\in\NN$, and thus, since $F_{n_k}\in\G{E}\cong(0,1)^E$, there exists $\lambda^{(n_k)}\in(0,1)^E$ with $\phi_E(\lambda^{(n_k)}) = \phi(F_{n_k})$ for all $k\in\NN$.
    
    For 1., by Bolzano-Weierstra\ss, there needs to exist a cluster point $\lambda^*\in[0,1]^E$ of $\lambda^{(n_k)}$.
    
    For 2., for any cluster point $\lambda^*\in[0,1]^E$, from the continuity of $\bar\phi_E$, $\bar\phi_E(\lambda^*)$ is a cluster point of $(\phi(F_n))_{n\in\NN}$ and by $F_n\to F'$ we find $\phi(F_n)\to\phi(F')$ and thus $\bar\phi_E(\lambda^*) = \phi(F')$.
    
    For 3., let $\lambda^*\in[0,1]^E$ be a cluster point. 
    If $\lambda^*\in(0,1)^E$ then $F'\in\G{E}$ and $E=E'$, a contradiction. Thus $\lambda^*\in\dd([0,1]^E)$, and due to $\bar\phi_E(\lambda^*) = \phi(F') \in\spd$, the assertion follows.
\end{proof}

The following example teaches that when $F_n \to F$, $\lambda^{(n)}$ can have distinct cluster points.

\begin{example}\label{exa:phylogenetic-forest-sequences-weak-conv}
    Let $N = 3$, set $e_i\coloneqq u\vert (L\setminus\{u\})$, $u\in L = \{1,2,3\}$ and $E = \{e_1,e_2,e_3\}$.
    Define the sequence of \walds  $F_n \coloneqq (E, \lambda^{(n)})$, $n\in\NN$, using a sequence $\eps_n\in (0,\frac{1}{4})$ with $\eps_n\to 0$ as $n\to\infty$, via
    \begin{align*}
        &\lambda^{(2n-1)}_{e_1} \coloneqq \frac{1}{2} - \eps_n, & &\lambda^{(2n-1)}_{e_2} \coloneqq\eps_n, & &\lambda^{(2n-1)}_{e_3}\coloneqq 1 - \eps_n,\\
        &\lambda^{(2n)}_{e_1} \coloneqq \eps_n, & &\lambda^{(2n)}_{e_2} \coloneqq \frac{1}{2} - \eps_n, & &\lambda^{(2n)}_{e_3}\coloneqq 1 - \eps_n.
    \end{align*}
    The corresponding forests are depicted in \Cref{fig:example-phylogenetic-forest-sequences-weak-conv}.
    Clearly, the sequence $\lambda^{(n)}$ ($n\in\NN$) has two distinct cluster points
    $ (1/2,0,1), (0, 1/2,1) \in (0,1)^3$. We observe, however, that
    \begin{align*}
        \phi(F_{2n-1}) = 
        \begin{pmatrix}
            1 & (1-\eps_n)(\frac{1}{2} + \eps_n) & (\frac{1}{2} + \eps_n)\eps_n\\
            (1-\eps_n)(\frac{1}{2} + \eps_n) & 1 & (1 - \eps_n)\eps_n\\
            (\frac{1}{2} + \eps_n)\eps_n & (1 - \eps_n)\eps_n & 1
        \end{pmatrix}
        \overset{n\to\infty}{\longrightarrow}
        \begin{pmatrix}
            1 & \frac{1}{2} & 0\\
            \frac{1}{2} & 1 & 0\\
            0 & 0 & 1
        \end{pmatrix},
    \end{align*}
    and similarly, $\phi(F_{2n})$ converges to the same matrix as $n\to\infty$.
    Letting $e' = 1\vert2$ and defining $F' = (E',\lambda') = (\{e'\},\lambda')$ with $\lambda'_{e'} = \frac{1}{2}$ (i.e. label partitions $L'_1 = \{3\}$, $L'_2 = \{1,2\}$; cf.~\Cref{fig:example-phylogenetic-forest-sequences-weak-conv}), we have that $\phi(F_n)\to\phi(F')$, so $F_n\to F'$.
\end{example}

\begin{figure}[ht]
    \centering
    \includegraphics[width=0.6\linewidth]{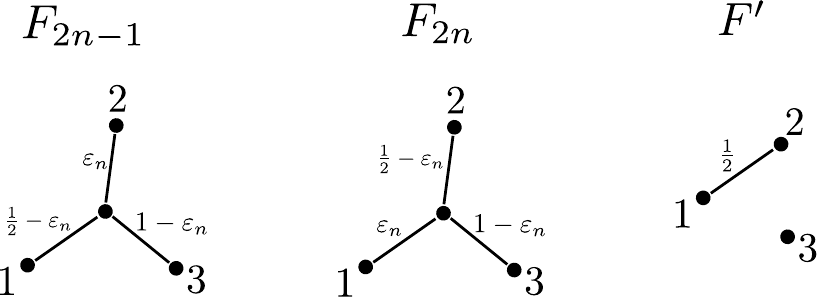}
    \caption{A sequence of \walds (left and middle) converging (right) but having different $\lambda$ cluster points as detailed in \Cref{exa:phylogenetic-forest-sequences-weak-conv}.}
    \label{fig:example-phylogenetic-forest-sequences-weak-conv}
\end{figure}

\Cref{thm:wald-convergence-general} shows that whenever a sequence of \walds $F_n\in\G{E}$ converges to a \wald $F'\in\W$ with topology $E'$ and $F'\notin\G{E}$, then $F'\in\dd\G{E}$. 
In the following section we make this relationship between $E'$ and $E$ more precise and unravel the boundary correspondences via a partial ordering on the \waldtops.

\subsection{At Grove's End}\label{scn:groves-end}

In light of \Cref{thm:wald-convergence-general}, we investigate how two \wald topologies $E = [F]$ and $E' = [F']$ are related to each other.

\begin{definition}
    Let $F\in\F$ be a wald with topology $E = [F]$.
    For an edge $e = A\vert B\in E$, we define the edge \emph{restricted} to some subset $L'\subseteq L$ by
    \begin{equation*}
        e\res{L'} \coloneqq (A\cap L')\vert(B\cap L')
    \end{equation*}
    if both of the sets above are non void, else, we say that the \emph{restriction does not exist}.
    In case of existence, we also say that $e\res{L'}$ is a \emph{valid split}.
\end{definition}

The following definition is from \cite{moulton_peeling_2004} and translated into the language of \walds and their topologies.

\begin{definition}\label{def:partial-ordering-on-forest-topologies}
    For two \walds $F,F'\in\F$ with topologies $E = [F],E' = [F']$, respectively, we say that 
    \begin{eqnarray}\label{eq:subforest}
    E'=[F']\leq [F] = E
    \end{eqnarray}
    if all of the following three properties hold:
    \begin{description}
        \item[Refinement] with  the partitions $L_1,\dots,L_K$ and $L'_1,\ldots,L_{K'}'$ of $L$ induced by $E'$ and $E$, respectively,  for every $1\leq \alpha'\leq K'$ there is $1\leq\alpha\leq K$ with $L'_{\alpha'} \subseteq L_\alpha$; 
        \item[Restriction] for every $1\leq \alpha'\leq K'$,
        \begin{equation*}
            E'_{\alpha'} \subseteq E\res{L'_{\alpha'}} \coloneqq \big\{\tilde{e}\colon\exists e \in E\st \tilde{e} \coloneqq e\res{L'_{\alpha'}}\text{ is a valid split}\big\}
        \end{equation*}
        where the r.h.s. is the set of splits $E$ \emph{restricted to $L'_{\alpha'}$}; 
        \item[Cut] 
        for every $1\leq \alpha'_1 \neq \alpha'_2\leq K'$ and $1\leq \alpha\leq K$ with $L'_{\alpha'_1}, L'_{\alpha'_2}\subset L_\alpha$, there is some 
        $$A\vert B \in E\mbox{ with }L'_{\alpha'_1}\subseteq A,~L'_{\alpha'_2}\subseteq B\,.$$
    \end{description}
    Further, we say $E' < E$ if $E\neq E' \leq E$. We also write $F' < F$ if $E' < E$.
\end{definition}

The restriction condition above corresponds to the definition of a tree \emph{displaying} another tree in \cite{moulton_peeling_2004}.
From \cite[Lemma~3.1]{moulton_peeling_2004}, it follows at once that the relation $\leq$ as defined in \Cref{eq:subforest} is a partial ordering.

\begin{example}\label{exa:partial-ordering-1}
    Let $N = 5$, so $L = \{1,\dots,5\}$.
    Define three \wald topologies
    \begin{align*}
        E &= \big\{1\vert2345, 12\vert345, 3\vert1245, 123\vert45, 1234\vert5, 1235\vert4\big\},\\
        E_1' &= \{2\vert3, 4\vert5\},\\
        E_2' &= \{2\vert5, 3\vert4\big\}.
    \end{align*}
    They are depicted in \Cref{fig:example-partial-ordering-1}.
    Then $E_1' < E$, as the refinement property holds, the restriction property, due to $E\res{\{2,3\}} = \{2\vert 3\}$, $E\res{\{4,5\}} = \{4\vert 5\}$, and the cut property, since the edge $1\vert 2345$ separates $\{1\}$ from $\{2,3\}$ and  $\{4,5\}$, and $123\vert 45$ separates $\{2,3\}$ from $\{4,5\}$.  Separating edges like $1\vert 2345$ cannot be restricted to any of the leaf sets  $\{1\}$, $\{2,3\}$ and  $\{4,5\}$, and if edges are restricted, they can only be restricted to one leaf set, e.g. $12\vert 345$ can be restricted only to $\{2,3\}$ and not to any of the others.
    
    In contrast, $E_2' \not\leq E$, although the refinement and restriction properties are satisfied, the cut property is not, since there is no edge $A\vert B = e\in E$ with $\{2,5\}\subseteq A$ and $\{3,4\}\subseteq B$.
\end{example}

\begin{figure}[ht]
    \centering
    \includegraphics[width=0.9\linewidth]{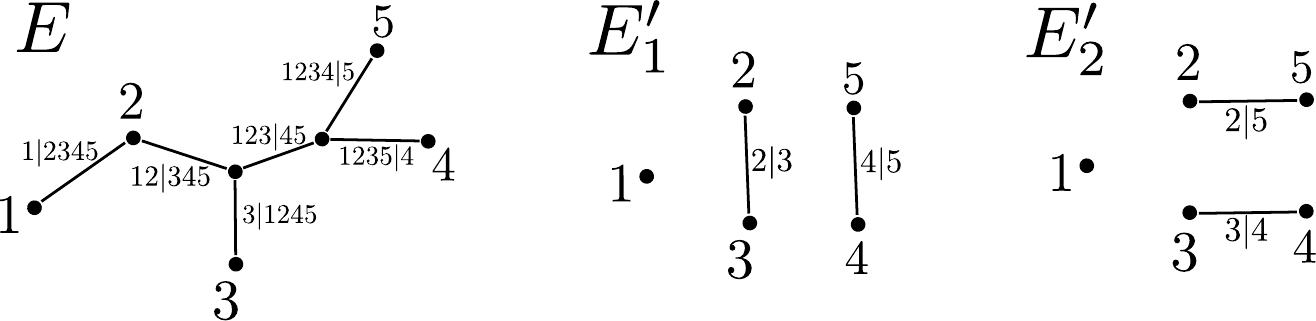}
    \caption{Wald topologies $E$, $E_1'$ and $E_2'$ from \Cref{exa:partial-ordering-1}, where $E_1' < E$ but $E_2'\not\leq E$.}
    \label{fig:example-partial-ordering-1}
\end{figure}

\begin{definition}\label{rmk:cut-disappear}
    Let $E,E'$ be \wald topologies with $E' \leq E$.
    \begin{enumerate}
        \item 
            For each edge $e' \in E'_{\alpha'}$, $1\leq\alpha'\leq K'$, denote the set of all \emph{corresponding} splits in $E$ by 
            \begin{equation*}
                R_{e'} \coloneqq \big\{e\in E\colon e\res{L'_{\alpha'}} = e'\big\}.
            \end{equation*}
        \item 
            Furthermore, denote the set of all \emph{disappearing} splits in $E$ with
            \begin{equation*}
                R_{dis} \coloneqq \big\{e\in E\colon \exists\alpha'\mbox{ s.t. } e\res{L'_{\alpha'}} \mbox{ is a valid split of } L'_{\alpha'}\mbox{, but } e\res{L'_{\alpha'}}\notin E'\big\}.
            \end{equation*}
        \item
            Denote the set of all \emph{cut} splits with 
            \begin{equation*}
                R_{cut} \coloneqq \big\{e\in E\colon \not\exists\alpha'\mbox{ s.t. } e\res{L'_{\alpha'}} \mbox{ is a valid split of } L'_{\alpha'} \big\}.
            \end{equation*}
    \end{enumerate}
\end{definition}

\begin{example}\label{exa:partial-ordering-2-3}
\begin{enumerate}
    \item 
        We revisit \Cref{exa:partial-ordering-1}, cf.~also \Cref{fig:example-partial-ordering-1}.
        Note that with respect to $E_1' < E$, we have, for instance, with $e' = 2\vert 3$ that $R_{e'} = \{12\vert 345, 3\vert 1245\}$, $R_{dis} = \emptyset$ and $R_{cut} = \{1\vert 2345, 123\vert 45\}$. By definition, none of the cut edges can be restricted.
    \item
        Let $N = 4$, so $L = \{1,2,3,4\}$.
        Define two wald topologies with 
        \begin{align*}
            E &= \big\{1\vert234, 2\vert 134, 3\vert124, 123\vert4, 12\vert34\big\},\\
            E' &= \big\{1\vert234, 2\vert 134, 3\vert124, 123\vert4\big\},
        \end{align*}
        where $E$ is a fully resolved tree with interior edge $12\vert 34$ and $E'$ is a \emph{star tree}, i.e.~four leaves attached to one interior vertex, cf.~\Cref{fig:example-partial-ordering-2}.
        Then $E' < E$ since $E'\subset E$, and the split $12\vert 34$ disappears, i.e.~$R_{dis} = \{12\vert 34\}$.
        Furthermore, $R_{cut} = \emptyset$ and $R_{e'} = \{e'\}$ for all $e'\in E'$.
    \item
        Let $N = 5$, so $L = \{1,2,3,4,5\}$.
        Define two wald topologies with 
        \begin{align*}
            E &= \big\{1\vert2345, 2\vert 1345, 3\vert 1245, 4\vert 1235, 5\vert 1234, 12\vert345, 123\vert 45\big\},\\
            E' &= \big\{1\vert245, 2\vert 145, 4\vert 125, 5\vert 124, 12\vert45\big\},
        \end{align*}
        where $E$ is a fully resolved tree with two \emph{cherries} containing $1,2$ and $4,5$, respectively, and $3$ attached as a leaf to an interior vertex, cf.~\Cref{fig:example-partial-ordering-3}.
        Furthermore, $E'$ has two connected components, a fully resolved tree with labels $1,2,4,5$ and isolated label $3$, cf.~\Cref{fig:example-partial-ordering-3}.
        Then $E' < E$ and $R_{12\vert 45} = \{12\vert 345, 123\vert 45\}$, $R_{cut} = \{3\vert 1245\}$ and $R_{dis} = \emptyset$.
\end{enumerate}
\end{example}

\begin{figure}[ht]
    \centering
    \includegraphics[width=0.7\linewidth]{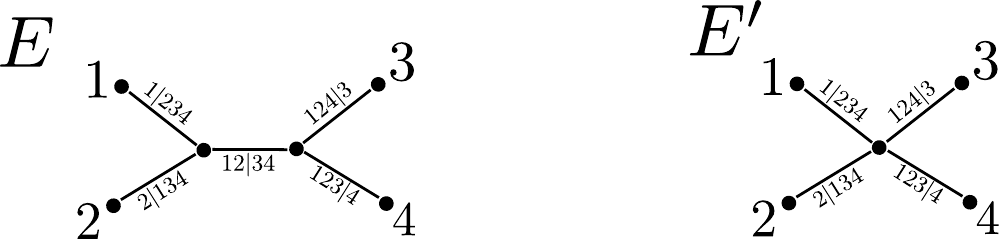}
    \caption{Wald topologies $E$ and $E'$ from \Cref{exa:partial-ordering-2-3} (2), where $E' < E$.}
    \label{fig:example-partial-ordering-2}
\end{figure}

\begin{figure}[ht]
    \centering
    \includegraphics[width=0.7\linewidth]{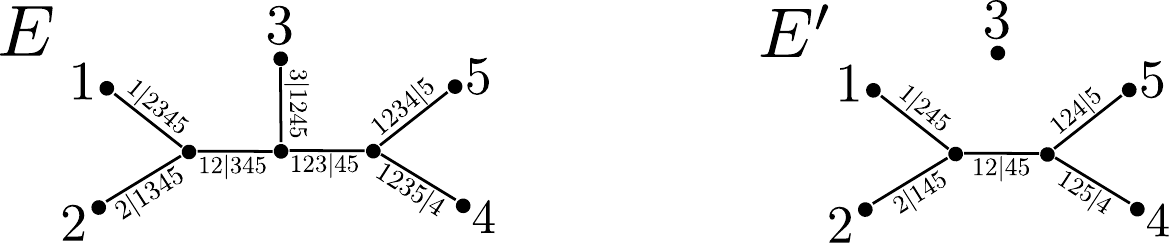}
    \caption{Wald topologies $E$ and $E'$ from \Cref{exa:partial-ordering-2-3} (3), where $E' < E$.}
    \label{fig:example-partial-ordering-3}
\end{figure}

\begin{lemma}\label{lem:groves-end-basics}
    Let $E' \leq E$ with label partitions $L_1,\ldots,L_K$ and $L'_1,\ldots,L'_{K'}$, respectively, and $u,v \in L$. 
    Then the following hold
    \begin{enumerate}[label=(\roman*)]
        \item If $K = K'$ then w.l.o.g.~$L'_\alpha = L_\alpha$ and $E'_\alpha\subseteq E_\alpha$ for all $\alpha=1,\dots,K$ and $R_{e'} = \{e'\}$ for all $e'\in E'$.
        \item $K < K' \iff R_{cut}\neq \emptyset$.
        \item If  $K=K'$ then $E'<E \iff R_{dis} \neq \emptyset$.
        \item $R_{e'}\neq\emptyset$ for all $e'\in E'$ and if $\,\exists e'\in E'_{\alpha'}$ with $\vert R_{e'}\vert > 1$  and $L'_{\alpha'} \subseteq L_\alpha$, then $L'_{\alpha'} \subsetneq L_\alpha$.
        \item $E = E' \iff (R_{dis}=\emptyset \text{ and } R_{cut}=\emptyset)$.
        \item $R_{e'}\cap R_{e''} = \emptyset$ for all $E'\ni e'\neq e'' \in E'$.
        \item The splits in $E\res{L'_{\alpha'}}$ are pairwise compatible.
        \item $e'\in E'(u,v)\iff R_{e'}\cap E(u,v) \neq\emptyset\iff R_{e'}\subseteq E(u,v)$.
        \item $R_{dis}, R_{cut}$ in conjunction with the $R_{e'}$ over all $e'\in E'$ give a pairwise disjoint union  of $E$, where $R_{dis}$ and $R_{cut}$ might be empty.
        \item Let $u,v\in L'_{\alpha'}$ for some $1\leq\alpha'\leq K'$.
        Then $R_{dis}\cap E(u,v)$ in conjunction with the $R_{e'}$ over all $e'\in E'(u, v)$  give a pairwise disjoint union  of $E(u,v)$, where $R_{dis}\cap E(u,v)$ might be empty.
        \item For any $L'_{\alpha'},L'_{\alpha''}\subset L_\alpha$ with $\alpha'\neq\alpha''$, there exists a split $A\vert B = e\in E$ with $L'_{\alpha'}\subseteq A$, $L'_{\alpha''}\subseteq B$ and $e\in R_{cut}$.
    \end{enumerate}
    Let $F,F'\in \cW$ with $\rho=\phi(F)$, $\rho'=\phi(F')$ and topologies $E$ and $E'$, respectively, with label partitions $L_1,\ldots,L_K$ and $L'_1,\ldots,L'_{K'}$, respectively. 
    Then the following hold
    \begin{enumerate}[label=(\roman*)]
        \setcounter{enumi}{11}
        \item If for all $u,v \in L$: $\rho_{uv} = 0 \implies \rho'_{uv} = 0$, then $L'_1,\dots,L'_{K'}$ is a refinement of $L_1,\dots,L_K$.
        \end{enumerate}
        Finally, we have the general result
        \begin{enumerate}[label=(\roman*)]
        \setcounter{enumi}{12}
\item For every \waldtop $E'$ with $\vert E'\vert < 2N-3$ there is a \waldtop $E$ with $\vert E\vert = \vert E'\vert +1 $ and $E'<E$.
    \end{enumerate}
\end{lemma}

\begin{proof} Since Assertion (ii), which, among others, implies Assertion (v),  follows from  Assertion (xi) we proceed in the following logical order.

\textit{(i)}:
        $K=K'$ implies w.l.o.g.~$L_\alpha = L'_\alpha$ for all $\alpha=1,\dots,K$ and therefore $e\res{L'_\alpha} = e\res{L_\alpha} = e$ are valid splits for all $e\in E_\alpha$ for all $\alpha=1,\dots,K$, so $E'_\alpha\subseteq E_\alpha$ as well as $R_{e'} = \{e'\}$ for all $e'\in E'$.

\textit{(iii)}:
        From \textit{(i)} w.l.o.g.~$L_\alpha=L'_\alpha$ and $E'_\alpha\subseteq E_\alpha$, $\alpha=1,\dots,K$.
        Thus $R_{dis} = \emptyset \iff(\text{for all } \alpha=1,\dots,K,~E'_\alpha = E_\alpha)\iff E' = E$.

\textit{(iv)}:
        By the restriction property of $E'\leq E$, each $e'\in E'_{\alpha'}$ is the restriction of some $e\in E_\alpha$, thus $e\in R_{e'} \neq\emptyset$.
        Assume that there exist $e_1,e_2\in R_{e'}$ with $e_1\neq e_2$. If $L'_{\alpha'} = L_\alpha$ was true, then $e_1 = e_1\res{L'_{\alpha'}} = e_2\res{L'_{\alpha'}} = e_2$, a contradiction.

\textit{(vi)}:
        Assume the contrary: let $A\vert B = e\in R_{e'}\cap R_{e''}$, where $e'\in L'_{\alpha'}\subset L_\alpha$ and $e''\in L'_{\alpha''}\subset L_\alpha$.\\
        If $\alpha' = \alpha''$, then $e' = e\res{L'_{\alpha'}} = e''$, a contradiction to $e'\neq e''$, so $\alpha'\neq\alpha''$.
        Since $e$ is in both $R_{e'}$ and $R_{e''}$, both restrictions to $L'_{\alpha'}$ and $L'_{\alpha''}$ exist and therefore
        \begin{equation*}
            A\cap L'_{\alpha'}\neq\emptyset,\quad B\cap L'_{\alpha'}\neq\emptyset,\quad A\cap L'_{\alpha''}\neq\emptyset,\quad B\cap L'_{\alpha''}\neq\emptyset.
        \end{equation*}
        Due to $E' \leq E$, by the cut property there exists $C\vert D=\tilde{e}\in E_\alpha$ separating $L'_{\alpha'}$ and $L'_{\alpha''}$, i.e. $L'_{\alpha'}\subseteq C$ and $L'_{\alpha''}\subseteq D$.
        But then $\tilde{e}, e\in E_\alpha$ cannot be compatible, a contradiction.

\textit{(vii)}:
        Let $L'_{\alpha'} \subset L_\alpha$ and $e',e''\in E\res{L'_{\alpha'}}$ such that $e' = e\res{L'_{\alpha'}}$ and $e'' = e^\circ\res{L'_{\alpha'}}$ with $e=A\vert B\in E$ and $e^\circ=A^\circ\vert B^\circ\in E$.
        Then $e,e^\circ\in E_\alpha$ for otherwise their restrictions to $L'_{\alpha'}$ would not be valid splits.
        Since $e$ and $e^\circ$ are compatible, w.l.o.g.~$A\cap A^\circ = \emptyset$.
        Consequently, $e' = (A\cap L'_{\alpha'})\vert(B\cap L'_{\alpha'})$ and $e'' = (A^\circ\cap L'_{\alpha'})\vert(B^\circ\cap L'_{\alpha'})$ are compatible as $(A\cap L'_{\alpha'})\cap (A^\circ\cap L'_{\alpha'}) = \emptyset$.

\textit{(viii)}:
        We show $e'\in E'(u,v)\implies R_{e'}\subseteq E(u,v)\implies R_{e'}\cap E(u,v)\neq\emptyset\implies e'\in E'(u,v)$.\\
        If $e'\in E'(u,v)$, then due to \textit{(iv)}, $R_{e'}\neq\emptyset$. Hence $e'\coloneqq e\vert_{L'_{\alpha'}} = (A\cap L'_{\alpha'})\vert(B\cap L'_{\alpha'})$ for some $e = A\vert B\in R_{e'}$, and thus $u\in A$, $v\in B$, or vice versa, i.e. $e\in E(u,v)$. 
        Since the choice $e\in R_{e'}$ was arbitrary, $R_{e'}\subseteq E(u,v)$.\newline
        If $e\in R_{e'}\cap E(u,v)$, $u,v\in L'_{\alpha'}$, then $e'= e\vert_{L'_{\alpha'}}$ and $e'\in E'(u,v)$ due to \Cref{eq:splits-on-path-equiv}.

\textit{(ix)}:
        By definition of $R_{dis}$ and $R_{cut}$, they are disjoint and furthermore have empty intersection with each $R_{e'}$, $e'\in E'$ and the latter are pair-wise disjoint due to \textit{(vi)}.

\textit{(x)}:
        By definition, $R_{cut}\cap E(u,v)$ for all $u,v\in L'_{\alpha'}$ (else $R_{cut}$ would contain valid splits).
        Then \textit{(ix)} in conjunction with \textit{(viii)} yields the assertion.

\textit{(xi)}:
        Without loss of generality, let $K = 1< K'$ and suppose that $\alpha'=1$, $\alpha''=2$.
        \\
        In the first step note that it suffices to find a split $e=A\vert B$ that separates $L'_{1}$ from $L'_{\alpha'}$ for all $2\leq \alpha'\leq K'$ for then, w.l.o.g. $L'_1 \subseteq A$, $L'_2,\ldots L'_{K'} \subset B$, which implies $L'_1 =A$, $L'_2\cup \ldots\cup L'_{K'} = B$, so that none of the $e\res{L'_1},\ldots e\res{L'_{K'}}$ is a valid split and in consequence $e \in R_{cut}$ as desired.\\
        In the second step we show the existence of such a $e$. In fact, to this end, it suffices to establish the following claim for all $3\leq J \leq K'$, invoke induction and separately show the assertion for $K'=2$.
        \begin{quote}Claim: 
            If $\exists$ split $f=C\vert D$ separating $L'_1$ from all of $L'_1,\ldots,L'_{J-1}$, i.e. w.l.o.g. $L'_1\subseteq C,~L'_2,\ldots,L'_{J-1}\subset D$, that has the property $C\cap L'_{J} \neq \emptyset \neq D\cap L'_{J}$ then $\forall$ compatible splits $e=A\vert B$  separating $L'_1$ from  $L'_{J}$ where, w.l.o.g. $L'_1\subseteq A, L'_{J}\subseteq B$ we have that $e$  separates $L'_1$ from all of $L'_1,\ldots,L'_{J}$, i.e. equivalently 
            $$ L'_{\alpha'}\subset B~\forall 2\leq \alpha'\leq J\,.$$
        \end{quote}
        Indeed, if $K'=2$ and $e=A\vert B$ separates $L'_1$ from $L'_2$ then, w.l.o.g., $A=L'_1$ and $B=L'_2$.\\
        In the third step we show the claim. To this end let $K'\geq 3$, $3\leq J \leq K'$, $f=C\vert D$ as in the claim's hypothesis and suppose that $e=A\vert B$ is an arbitrary compatible split with $L'_1 \subseteq A$, $L'_J \subseteq B$. Then
        $$ C\cap A \supseteq L'_1 \neq \emptyset,~ C \cap B \supseteq  C \cap L'_J \neq \emptyset,~D \cap B \supseteq  D \cap L'_J \neq \emptyset\,.$$
        By compatibility of splits we have thus $\emptyset = D\cap A\supseteq  L'_{\alpha'}\cap A$ for all $2\leq \alpha' \leq J$ by hypothesis, yielding
        $$ L'_{\alpha'}\subset B~\forall 2\leq \alpha'\leq J\,,$$
        thus establishing the claim.        
        
\textit{(ii)}: We show equivalently $K=K'\Leftrightarrow R_{cut}=\emptyset$. 
``$\Rightarrow$'': If $K=K'$, then by \textit{(i)} w.l.o.g.~$L'_\alpha = L_\alpha$ and in particular $e\res{L'_\alpha} = e\res{L_\alpha} = e$ are valid splits for all $e\in E_\alpha$, $\alpha=1,\dots,K$, so that $R_{cut} = \emptyset$.
``$\Leftarrow$'' follows at once from \textit{(xi)}. 

\textit{(v)}: ``$\Rightarrow$'': Trivial. ``$\Leftarrow$'': $R_{cut} = \emptyset \implies K = K'$ due to \textit{(ii)} and thus $R_{dis} = \emptyset\implies E = E'$ due to \textit{(iii)}.

\textit{(xii)}: 
        Let $1 \leq \alpha'\leq K'$ and $u\in L'_{\alpha'}$. Then, there is $1 \leq \alpha\leq K$ such that $u\in L_\alpha$. For any other $v\in L'_{\alpha'}$, $\rho'_{uv}>0$, so by assumption $\rho_{uv}>0$, thus $v\in L_\alpha$, yielding $L'_{\alpha'}\subseteq L_\alpha$.
        
\textit{(xiii)}: Suppose that $F'$ is  a \wald 
 with leaf partition $L'_1,\ldots, L'_{K'}$ and $\vert E'\vert < 2N-3$. 
 \\
 In case of $K'=1$ there is a vertex of degree $k\geq 4$, i.e. there is a partition $A_1,\ldots,A_k$ of $L=L'_1$ with splits
 $$ A_i\vert L\setminus A_i \in E',~1\leq i\leq k$$
 and all other splits in $E'$ are of form 
 $$ A'_i\vert L\setminus A'_i \in E',~1\leq i\leq k$$
 where $A'_i$ is a suitable subset of $A_i$.  Then one verifies at once that the new split $e:=A_1\cup A_2 \vert L\setminus (A_1\cup A_2)$ is compatible with all splits in $E'$ so that 
 $E := E' \cup \{e\}$ is a \waldtop with the desired properties $\vert E\vert = \vert E'\vert +1$ and $E'<E$. For the latter note that $R_{e'} = \{e'\}$ for all $e' \in E'$, $R_{cut} =\emptyset$ and $R_{dis}= \{e\}$.
 \\
In case of $K'\geq 2$ introduce the new split $f:=L'_1\vert L'_2$ and for every $e'_1=A\vert B \in E'_1$ let $e(e'_1): = A\vert B\cup L'_2$, so that $e(e'_1)\res{L'_1} =e'_1$. Similarly, for every $e'_2=C\vert D \in E'_2$ let $e(e'_2): = C\vert D\cup L'_1$, so that $e(e'_2)\res{L'_2} =e'_2$. Setting
$$ E := \{e(e'): e' \in E'_1\cup E'_2\} \cup \{f\}  \cup E'_3 \ldots \cup E'_{K'}$$
one verifies that all splits in $E$ are pairwise compatible. Hence $E$ is a \waldtop with  $\vert E\vert = \vert E'\vert +1$ and $E'<E$. Indeed, for the latter note that $R_{e'} = \{e(e')\}$ for all $e' \in E'_1\cup E'_2$, $R_{e'} = \{e'\}$ for all $e' \in E'_3\cup\ldots\cup E'_{K'}$, $R_{cut} = \{f\}$ and $R_{dis}=\emptyset$.
\end{proof}

In the following theorem, we characterize the boundaries of \groves via the partial ordering on \waldtops.

\begin{theorem}\label{thm:boundary-of-groves}
    For \wald topologies $E$ and $E'$, the following three statements are equivalent (with $\dd\G{E}$ as in \Cref{eq:boundary-of-grove}):
    \begin{enumerate}[label={(\roman*)}]\itemsep0.3em
      \item $E' < E$,
      \item $\G{E'} \subset \dd\G{E}$,
      \item $\G{E'}\cap \dd\G{E} \neq\emptyset$.
    \end{enumerate}
\end{theorem}

\begin{proof}
    Let $E$ have label partition $L_1,\dots,L_K$.
    
    ``$(i)\implies(ii)$''.
    Assume that $F' = (E',\lambda')\in\G{E'}$ with partition $L_1',\dots,L_{K'}'$.
    Using Lemma~\ref{lem:groves-end-basics}, (ix), set
        \begin{equation*}
            \lambda^*_e \coloneqq  \left\{\begin{array}{cl}
                 0 & e\in R_{dis} \\
                 1 & e \in R_{cut}\\
                  1 - (1-\lambda'_{e'})^{1/\vert R_{e'}\vert }\quad & e\in R_{e'}, e' \in E' 
            \end{array}\right.
        \end{equation*}
    to obtain $\lambda^* \in \partial([0,1]^E)$ since $R_{cut}\cup R_{dis} \neq \emptyset$ due to $E' < E$ by Lemma~\ref{lem:groves-end-basics}, (v). 
    By injectivity of $\phi$, it suffices to show $(*)$:
    \begin{equation*}
        \bar\phi_E(\lambda^*) \eqqcolon (\rho_{uv}^*)_{u,v=1}^N   \overset{(*)}{=}  (\rho_{uv}')_{u,v=1}^N \coloneqq \phi(F')\,.
    \end{equation*}
    First, observe by Agreement (\ref{eq:phi-map-agreement}) that for all $u \in L$,
    \begin{equation*}
        \rho^*_{uu} = 1 = \rho'_{uu} \,.
    \end{equation*}
    Next, again from Agreement (\ref{eq:phi-map-agreement}), for all $u,v\in L$ with $u\neq v$ that are not connected in $F'$, say $u\in L'_{\alpha'_1}$, $v\in L'_{\alpha_2'}$ for some $\alpha'_1,\alpha'_2\in\{1,\dots,K'\}$, we have $\rho_{uv}' = 0$. 
    If $u$ and $v$ are also not connected in $E$, then $\rho_{uv}^* = 0 = \rho_{uv}'$.
    Assume now that $u$ and $v$ are connected in $E$. 
    Then, by \Cref{lem:groves-end-basics}, (xi), there exists an edge $A\vert B = e\in R_{cut}$ with $u\in A$ and $v\in B$, and due to $\lambda^*_e = 1$ by construction, $\rho^*_{uv} = 0 = \rho'_{uv}$.
    
    Finally, for all $u,v \in L$ that are connected in $F'$, we have, due to construction and Lemma \ref{lem:groves-end-basics}, (x),
    \begin{eqnarray*}
    \rho^*_{uv} &=&  \prod_{e\in E(u,v)} (1-\lambda^*_e)\\
    &=& \bigg(\prod_{e\in R_{dis}\cap E(u,v)}(1 - \lambda^*_e) \bigg)\bigg(\prod_{e'\in E'(u,v)} \prod_{e \in R_{e'}}(1-\lambda'_{e'})^{1/\vert R_{e'}\vert }\bigg)\\
    &=& \prod_{e'\in E'(u,v)} (1-\lambda'_{e'}) ~=~\rho'_{uv}\,.
    \end{eqnarray*}
    Thus, we have shown $\phi(F') = \bar\phi_E(\lambda^*)$. As $F'=(E',\lambda')$ was arbitrary, we have shown $\G{E'} \subset \dd\G{E}$ where equality cannot be due to $\lambda^* \in \partial ([0,1]^E)$.
    
    ``$(ii)\implies(iii)$'' is trivial.
    
    ``$(iii)\implies(i)$''.
    Let $F'=(E',\lambda')\in\G{E'}\cap\dd\G{E}$, i.e.~there exists $\lambda^*\in\partial([0,1]^E)$ with $\bar\phi_E(\lambda^*)=\phi(F')\in\spd$.
    In the following, we will construct $F^\circ = (E^\circ,\lambda^\circ)$ with $\lambda^\circ \in (0,1)^{E^\circ}$ and show that 
    \begin{description}
        \item[Claim I:] $E^\circ < E$, and
        \item[Claim II:] $\phi(F^\circ) = \phi(F')$.
    \end{description}
    As Claim II implies $F^\circ = F'$ and $E^\circ = E'$, in conjunction with Claim I we then obtain the assertion $E' < E$.

    In order to see Claim I, let $\bar\phi_E(\lambda^*) = (\rho_{uv}^*)_{u,v=1}^N$. Denote the connectivity classes of $L$, where $u,v \in L$ are connected if and only if $\rho^*_{uv} > 0$, by
    \begin{equation*}
        L^\circ_{1},\ldots, L^\circ_{K^\circ},
    \end{equation*}
    with $1\leq K^\circ\leq N$.
    Since $\rho^*_{uv} > 0$ implies $\rho_{uv} > 0$, we have that $L^\circ_{1},\ldots, L^\circ_{K^\circ}$ is a \emph{refinement} of $L_1,\dots,L_K$ by \Cref{lem:groves-end-basics}, (xii).
    
    Define $E^\circ$ by setting for each $1\leq \alpha^\circ \leq K^\circ$ (where, say, $L^\circ_{\alpha^\circ}\subset L_\alpha$ for some $1\leq\alpha\leq K$)
    \begin{equation}\label{eq:groves-end-proof-restrict}
        E^\circ_{\alpha^\circ} \coloneqq \Big\{e\res{L^\circ_{\alpha^\circ}}\colon \exists e\in E \st e\res{L^\circ_{\alpha^\circ}} \text{ is a valid split and } \lambda^*_e \neq 0 \Big\},
    \end{equation}
    and $E^\circ \coloneqq \bigcup_{\alpha^\circ}E^\circ_{\alpha^\circ}$.
    By \Cref{lem:groves-end-basics} (vii), each $E^\circ_{\alpha^\circ}$ comprises compatible splits only 
    so that $E^\circ$ satisfies 
    the \emph{restriction property} from \Cref{def:partial-ordering-on-forest-topologies}.

    Verifying the cut property, suppose there exist $1\leq\alpha_1^\circ\neq\alpha_2^\circ\leq K^\circ$ and $1\leq \alpha\leq K$ such that $L^\circ_{\alpha_1^\circ},L^\circ_{\alpha_2^\circ}\subset L_\alpha$.
    Hence by construction 
    \begin{equation}\label{eq:groves-end-proof-cut}
    \rho_{us}^* = 0,~ \rho_{uv}^* > 0\mbox{ and } \rho_{st}^* > 0\mbox{  for all }u,v\in L^\circ_{\alpha_1}\mbox{ and }s,t\in L^\circ_{\alpha_2}\,.    
    \end{equation}
    Let now $u\in L^\circ_{\alpha_1^\circ}$ and $s\in L^\circ_{\alpha_2^\circ}$, then by definition of $\bar\phi_E$, $\rho_{us}^* =\prod_{e\in E(u,s)}(1 - \lambda_e^*)= 0$, so there must exist $e =A\vert B\in E(u,s)$ with $\lambda_e^* = 1$.
    This implies $L^\circ_{\alpha_1^\circ}\subseteq A$ and $L^\circ_{\alpha_2^\circ} \subseteq B$, for otherwise, if $A\not \ni v\in L^\circ_{\alpha_1^\circ}$, say,
    then $v\in B$ and hence $e \in E(u,v)$ due to \Cref{eq:splits-on-path-equiv} and hence $\rho^*_{uv}=0$, due to $\lambda^*_e=1$, a contradiction to \Cref{eq:groves-end-proof-cut}.
    Thus the \emph{cut property} holds.
    
    Having verified all of the properties from \Cref{def:partial-ordering-on-forest-topologies}, we have shown $E^\circ \leq E$, and we can use the notation introduced in \Cref{rmk:cut-disappear} and Lemma \ref{lem:groves-end-basics} is applicable for $E^\circ \leq E$.
    Since $\lambda^*$ is on the boundary, there must be some $e\in E$ with either $\lambda^*_e = 1 > \lambda_e > 0$ or all $\lambda^*_e < 1$ and there is $\lambda^*_e= 0 <\lambda_e$. In the first case, $e\in R_{cut}$, in the second case $e \in R_{dis}$, so that in both cases $E^\circ \neq E$ by \Cref{lem:groves-end-basics}, (v), yielding $E^\circ < E$, which was Claim I.
    
    In order to see Claim II we define suitable edge weights $\lambda^\circ$.
    Let $1\leq\alpha^\circ\leq K^\circ$ be arbitrary and let $1\leq\alpha\leq K$ be such that $L^\circ_{\alpha^\circ}\subseteq L_\alpha$.
    For each $e^\circ\in E^\circ_{\alpha^\circ}$, define
    \begin{equation}\label{eq:groves-end-weights}
        \lambda^\circ_{e^\circ} \coloneqq 1- \prod_{e\in R_{e^\circ}} (1-\lambda^*_e)\,.
    \end{equation}
    Indeed, $\lambda^\circ_{e^\circ}\in(0,1)$, since by \Cref{lem:groves-end-basics} (ix), none of the $e\in R_{e^\circ}
    $ lie in $R_{cut}$, we have  $\lambda^*_e< 1$, and, since at least for one $e\in  R_{e^\circ}$, we have $\lambda^*_e > 0$ by \Cref{{eq:groves-end-proof-restrict}}.
    Thus $F^\circ \coloneqq (E^\circ,\lambda^\circ)$ is a well defined \wald. 
    
    We now show the final part of Claim II, namely that $\phi(F') = \phi(F^\circ)$. 
    Recall that $\phi(F') = \bar\phi_E(\lambda^*) = (\rho_{uv}^*)_{u,v=1}^N$ and let $\phi(F^\circ) = (\rho^\circ_{uv})_{u,v=1}^N$.
    By Agreement (\ref{eq:phi-map-agreement}), for all $u\in L$ we have $\rho^*_{uu} = 1 = \rho^\circ_{uu}$ and by definition of the connectivity classes $L^\circ_1,\dots,L^\circ_{K^\circ}$ we have $\rho^*_{uv} = 0$ if and only if $\rho^\circ_{uv}=0$ for all $u,v\in L$.
    
    For all other $u,v \in L$, we may assume that $u,v \in L^\circ_{\alpha^\circ}$ with $L^\circ_{\alpha^\circ}\subseteq L_\alpha$ for some $1\leq\alpha^\circ\leq K^\circ$ and $1\leq\alpha\leq K$.
    By \Cref{lem:groves-end-basics} (viii) and (ix), the sets $R_{dis}\cap E(u,v)$ in conjunction with $R_{e^\circ}$ for all $e^\circ\in E^\circ(u,v)$ form a partition of $E(u,v)$.
    For the first set we have
    \begin{equation}\label{eq:groves-end-dis}
        e\in R_{dis}\cap E(u,v) \Rightarrow \lambda^*_e=0\,.
    \end{equation}
    Indeed, if $e\in R_{dis}\cap E(u,v)$ then the restriction $e^\circ \coloneqq e\res{L^\circ_{\alpha^\circ}}$ is a valid split as it splits $L^\circ_{\alpha^\circ}$ into two non empty sets. But as $e \in R_{dis}$ this split does not exist in $E^\circ$ which, taking into account \Cref{eq:groves-end-proof-restrict}, 
    is only possible for $\lambda^*_e = 0$.

    In consequence, we have (the first and the last equality are the definitions, respectively, the second uses that $R_{dis}\cap E(u,v)$ and $R_{e^\circ}$, $e^\circ\in E^\circ(u,v)$ partition $E(u,v)$ and the third uses for the first factor  (\ref{eq:groves-end-dis})  and (\ref{eq:groves-end-weights}) for the second factor)
    \begin{align*}
        \rho^*_{uv} 
        &=
        \prod_{e\in E(u,v)} (1-\lambda^*_e)\\
        &=\bigg(\underbrace{\prod_{e\in R_{dis}\cap E(u,v)} (1 - \lambda^*_e)}_{=1}\bigg)\bigg(\prod_{e^\circ\in E^\circ(u,v)} \underbrace{\prod_{e\in R_{e^\circ}} (1-\lambda^*_e)}_{= 1 - \lambda^\circ_{e^\circ}}\bigg)\\
        &= \prod_{e^\circ\in E^\circ(u,v)} (1-\lambda^\circ_{e^\circ}) \\
        &= \rho^\circ_{uv},
    \end{align*}
    completing the proof.
\end{proof}

From the above theorem and its proof, we collect at once the following key relationships.

\begin{corollary}\label{cor:closure-of-grove-consists-of-other-groves}
    Let $F\in\F$ with topology $E$. Then
    \begin{equation*}
      \dd\G{E} = \bigsqcup_{E' < E} \G{E'}.
    \end{equation*}
   Further for $F' \in \F$ with topology $E'<E$,  $\phi_{E'}(\lambda') = \phi(F') = \bar \phi_{E}(\lambda^*)$ for $\lambda' \in (0,1)^{E'}$ and $\lambda^*\in\dd([0,1]^E)$,
    the following hold:
    \begin{enumerate}[start=1]
        \item 
        for each $1\leq \alpha'\leq K'$, we have that 
        \begin{equation*} 
            E'_{\alpha'} = \Big\{e\res{L'_{\alpha'}} \colon e\in E \st e\res{L'_{\alpha'}} \text{ is a valid split and } \lambda^*_e \neq 0 \Big\},
        \end{equation*}
        \item 
        for any $e'\in E'$,
        \begin{equation*}
            \lambda'_{e'} = 1 - \bigcup_{e\in R_{e'}} (1- \lambda_e^*)\,.
        \end{equation*}
    \end{enumerate}
\end{corollary}

\begin{example}\label{exa:boundary-of-grove-n3}
    Let $N = 3$, so $L = \{1,2,3\}$ and let 
    \begin{equation*}
        E = \{e_1 = 1\vert 23,~e_2 = 2\vert 13,~e_3 = 3\vert 12\}.
    \end{equation*}
    Abbreviating $\lambda_{e_i}=\lambda_i$ for $i=1,2,3$ we have $\G{E} \cong (0,1)^3$, cf. \Cref{eq:grove-cube-identification}, and the map $\bar\phi_E\colon[0,1]^3\to\sym$ from \Cref{eq:phi-map-extended-analytic} has the form
    \begin{equation*}
        \bar\phi_E(\lambda) = 
        \begin{pmatrix}
            1 & \rho_{12} & \rho_{13} \\
            \rho_{12} & 1 & \rho_{23} \\
            \rho_{13} & \rho_{23} & 1
        \end{pmatrix}
        = 
        \begin{pmatrix}
            1 & (1 - \lambda_1)(1 - \lambda_2) & (1 - \lambda_1)(1 - \lambda_3) \\
            (1 - \lambda_1)(1 - \lambda_2) & 1 & (1 - \lambda_2)(1 - \lambda_3) \\
            (1 - \lambda_1)(1 - \lambda_3) & (1 - \lambda_2)(1 - \lambda_3) & 1
        \end{pmatrix}\hspace{-4pt}\,.
    \end{equation*}
    One can easily see from \Cref{lem:boundary-of-cube-when-spd} that $\bar\phi_E(\lambda) \in\spd$ if and only if at most one coordinate of $\lambda$ is zero, otherwise there would exist $u,v\in L$ with $u\neq v$ such that $\rho_{uv} = 1$.
    This means that the origin and the intersections of the coordinate axes with the cube $[0,1]^3$ are not mapped into $\phi(\F)$ and thus not part of the boundary of $\G{E}$.
    The cube and the corresponding \walds $F'\in\F$ of the grove's boundaries (i.e. there exists $\lambda\in\dd([0,1]^E)$ with $\bar\phi_E(\lambda) = \phi(F')$) are depicted in \Cref{fig:example-boundary-of-grove-n3}.
    
    Note that for the boundaries where at least one $\lambda$ coordinate is one, infinitely many coordinates give the same \wald: let $F' = (\{2\vert 3\}, \lambda')$ with $\lambda'_{2\vert 3} = 0.8$, then all coordinates $\lambda^*=(1,\lambda_2^*,\lambda_3^*)\in\dd([0,1]^E)$ that satisfy $1 - (1 - \lambda_2^*)(1 - \lambda_3^*) = 0.8$ give $\bar\phi_E(\lambda^*) = \phi(F')$.
    This is also illustrated in \Cref{fig:example-boundary-of-grove-n3} (right panel), where several arrows point to the coordinates on curves that correspond the same \wald.
    This means that a two-dimensional boundary of the cube collapses into a one-dimensional \grove.
    
    If at least two coordinates of $\lambda^*$ are equal to 1, then the corresponding phylogenetic forest will be the forest consisting of three isolated vertices, and in this case, four points as well as the three segments where two coordinates are $1$ and one is strictly between zero and one on the boundary of the cube collapse to only one point in $\F$, marked red in \Cref{fig:example-boundary-of-grove-n3} (right panel).
\end{example}

\begin{figure}
    \centering
    \begin{minipage}{0.49\linewidth}
    \includegraphics[width=1\linewidth]{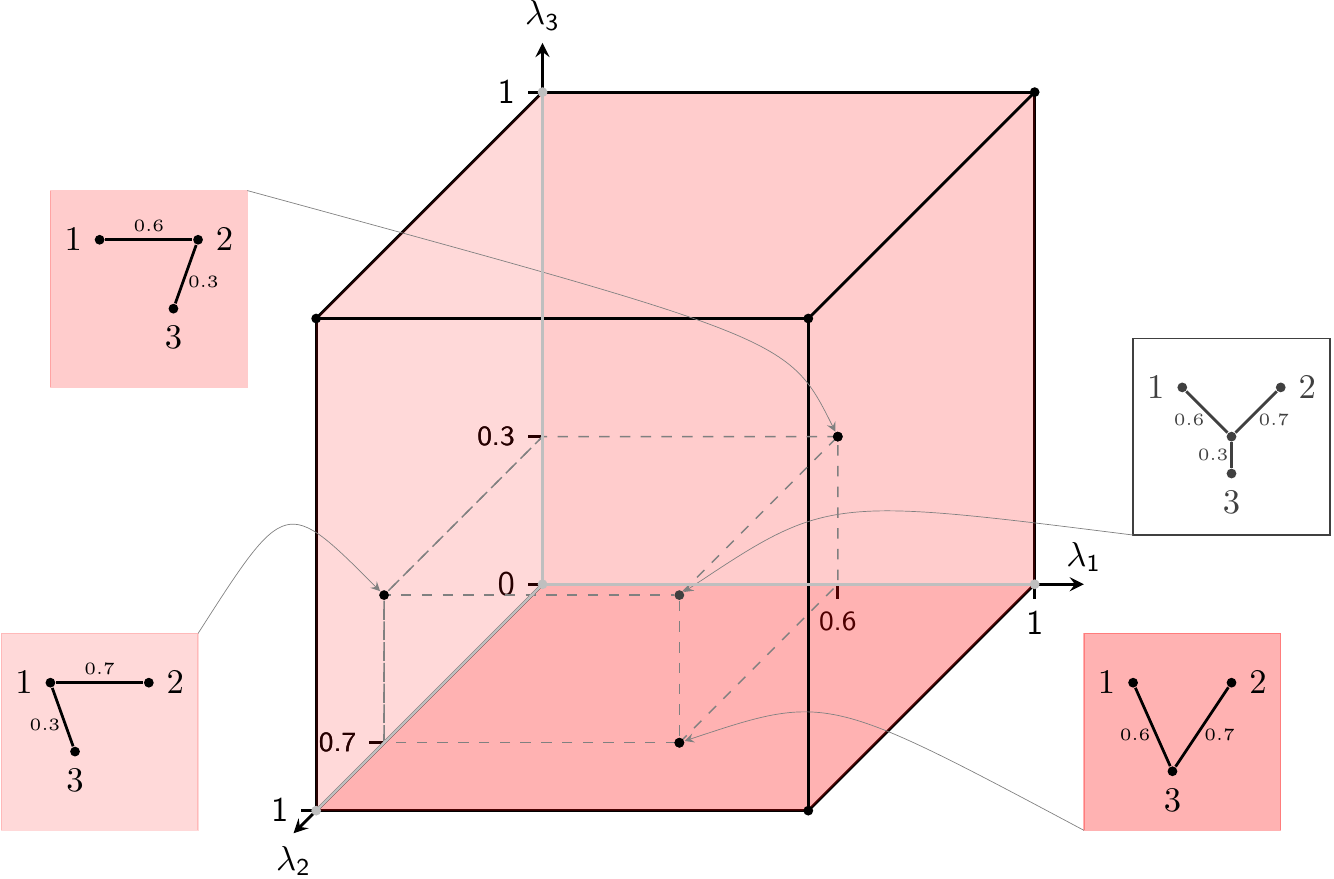}
    \end{minipage}
    \begin{minipage}{0.49\linewidth}
    \includegraphics[width=1\linewidth]{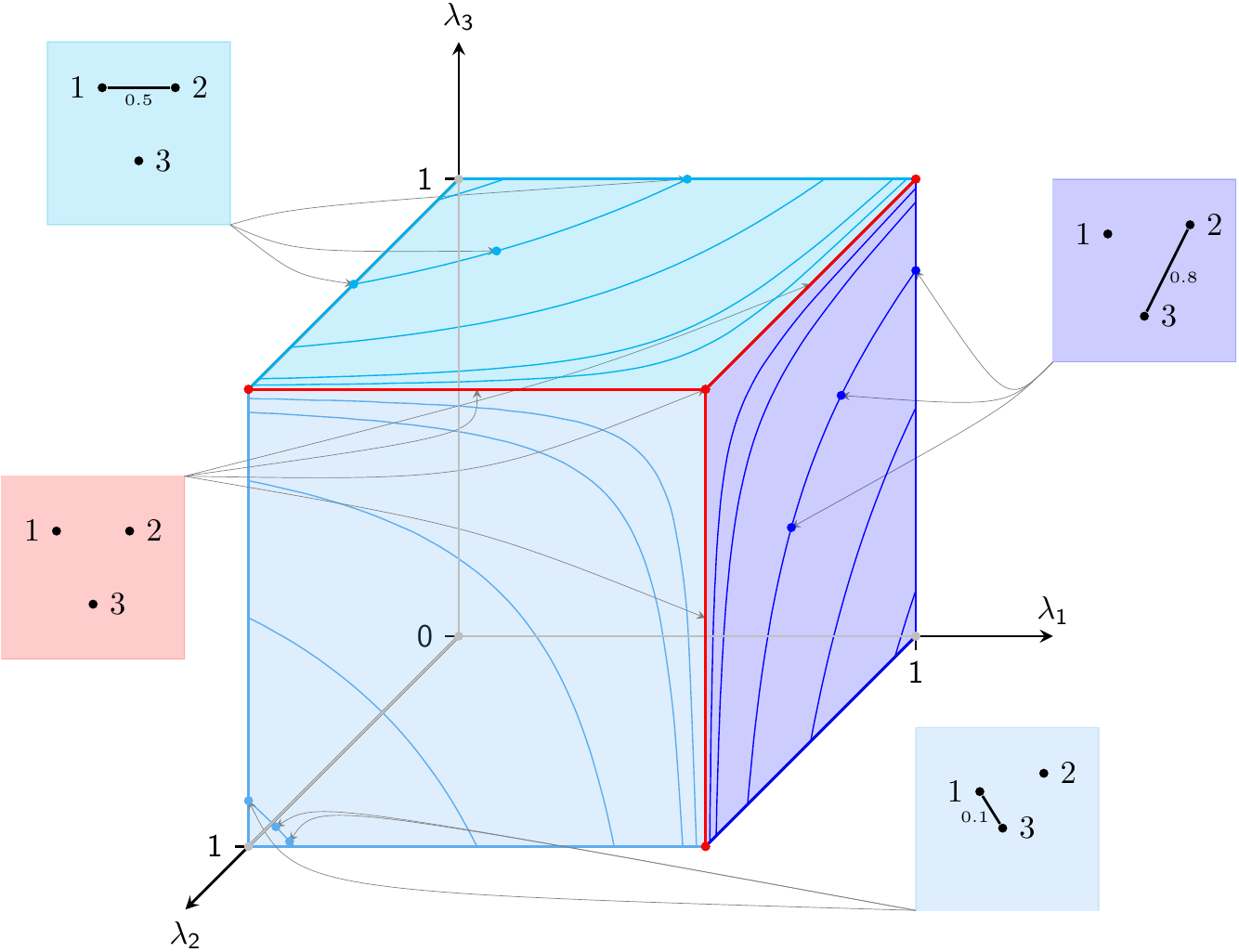}
    \end{minipage}
    \caption{Depicting the grove $\G{E}\cong(0,1)^3$ of a fully resolved tree with $N=3$ leaves, and its boundary $\dd\G{E}$, as discussed in \Cref{exa:boundary-of-grove-n3}. Left: $\G{E}$ and its two-dimensional ``boundary at zero'' (coordinate axes are excluded). Right: the ``boundary at one'' comprising the one-dimensional component (points on same blue curves  represent a single wald) and zero-dimensional component (points on the red spider) represent $F_\infty$. 
}
    \label{fig:example-boundary-of-grove-n3}
\end{figure}

\begin{corollary}\label{cor:converging-sequence-implies-boundary}
    Let $F,F'\in\F$ with topologies $E,E'$, respectively, and let $(F_n)_{n\in\NN}\subset\G{E}\subset\F$ be a sequence. If $F_n\to F'$, then $E'\leq E$.
\end{corollary}

\begin{proof}
    Let $\lambda^{(n)}\in(0,1)^E\cong\G{E}$ such that $\phi_E(\lambda^{(n)}) = \phi(F_n)$ for all $n\in\NN$.
    With the same argument as in the proof of \Cref{thm:wald-convergence-general}, there exists at least one subsequence $(\lambda^{(n_k)})_{k\in\NN}$ such that $\lambda^{(n_k)}\to\lambda^*\in[0,1]^E$ with $\bar\phi_E(\lambda^*) = \phi(F') \in \spd$, so either $F'\in\G{E}$, then $E' = E$, or $F'\in\dd\G{E}$ (by definition of $\dd\G{E}$ from \Cref{eq:boundary-of-grove}), then by \Cref{thm:boundary-of-groves} it follows that $E' < E$, so in general $E'\leq E$.
\end{proof}

\subsection{Whitney Stratification of \waldspace}\label{scn:whitney}

Recall from \Cref{scn:notation} the differentiable manifold of strictly positive definite matrices $\spd$, and that the tangent space $T_P\spd$ at $P\in\spd$ is isomorphic to the vector space of symmetric matrices $\sym$.
In order to study convergence of linear subspaces of $\sym$, we recall the \emph{Grassmannian manifold} of $k$-dimensional linear subspaces in $\RR^m$, $0\leq k \leq m$, see e.g. \cite[Chapter 7]{lee_introduction_2018}.

Every $k$-dimensional linear subspace $\mc{V}$
of $\RR^m$ is the span of the columns $v_1,\ldots,v_k$ of a matrix $(v_1,\ldots,v_k) = V \in S(m,k)$, the \emph{column space}, 
$$\mc{V} = \Span\{v_1,\ldots,v_k\} = \col(V)$$
where $$ S(m,k) = \{V \in \RR^{m\times k}: \rank(V) = k\}$$
is the \emph{Stiefel manifold} of maximal rank $(m\times k)$-matrices equipped with the smooth manifold structure inherited from embedding in the Euclidean $\RR^{m\times k}$. Since $\col(V) = \col(VG)$ for every $G\in S(k,k)$ and $V\in S(k,m)$, the space
$$\big\{\mc{V}\subset\RR^m\colon \mc{V}\text{ linear subspace}, \dim(\mc{V}) = k\big\}$$
can be identified with the \emph{Grassmannian}
$$ G(m, k) \coloneqq S(m,k)/S(k, k)\,.$$
As every orbit $\{VG: G \in S(k,k)\}$ of $V\in S(m,k)$ is closed in $S(m,k)$ and  since for every $V\in S(m,k)$ its \emph{isotropy group} $\{G\in S(k,k): VG= V\}$ contains the unit matrix only, the quotient carries a canonical smooth manifold structure. 
\begin{definition}\label{def:conv-Grassmannian}
    With the above notation, 
a sequence of $k$-dimensional linear subspaces $\mc{V}_n$ , $n\in \NN$, of $\RR^m$, $1\leq k < m$, \emph{converges in the Grassmannian $G(m,k)$} to a  $k$-dimensional linear subspace $\mc{V}$ if there are $V_n,V \in S(m,k)$ and $G_n \in S(k,k)$ such that 
$$ \col(V_n) = \mc{V}_n \mbox{ for all }n\in\NN,~\col(V)=\mc{V}\mbox{ and } \|V_nG_n - V\| \to 0\mbox{ as }n\to\infty\,.$$
\end{definition}

\begin{remark}\label{rmk:Grassmann}
     1) Note that none of the cluster points of $G_n$ or $G_n/\|G_n\|$ can be singular, hence they are all in $S(k,k)$
     
     2) There may be, however, a sequence $V_n \in S(m,k)$ and $V\in S(m,k)$, $W\in \RR^{m\times k}\setminus S(m,k)$ with
$$\col(V_n) = \mc{V}_n \to \mc{V} = \col(V)$$
in the Grassmanian $G(m,k)$ but 
$$\|V_n-W\| \to 0$$
in $\RR^{m\times k}$. Nevertheless we have the following relationship.
\end{remark}

\begin{lemma}
    \label{lem:limits-of-subspaces}
    Let $V_n \in S(m,k)$ and assume that the two limits below exist.  
    Then
\begin{eqnarray*} 
\col\left(\lim_{n\to \infty}  V_n\right) &\subseteq& \lim_{n\to \infty} \col(V_n)\,.
\end{eqnarray*}
\end{lemma}    

\begin{proof}
Let $v\in \RR^m$ with $v \perp \lim_{n\to \infty} \col(V_n)$. Then the assertion follows, once we show  $v \perp W$ with $W=\lim_{n\to \infty} V_n$.

By hypothesis, for every $\epsilon >0$ there are $N\in \NN$ and $G_n \in S(k,k)$ such that
$$ \vert v^T V_n G_n\vert < \epsilon~\forall n > N\,.$$
Let us first assume that there is a subsequence $n_k$ with $\|G_{n_k}\| > 1$. Then
$$ \vert v^T V_{n_k} R_{n_k}\vert < \frac{\epsilon}{\|G_{n_k}\|} < \epsilon~\forall n_k > N\,,$$
where $R_{n_k} = \frac{G_{n_k}}{\|G_{n_k}\|} \in S(k,k)$ is of unit norm, hence it has a cluster point $R$ satisfying $ \vert v^T W R\vert \leq  \epsilon$. As $\epsilon >0$  was arbitrary, we have $v^TWR =0$. Since $R\in S(k,k)$ by \Cref{rmk:Grassmann} we have thus $v^TW =0$ as asserted.

If there is no such subsequence, w.l.o.g. we may assume $\|G_n\| \leq 1$ for all $n\geq N$. Again, $G_n$ has a cluster point $R$ and thus $ \vert v^T W R\vert \leq  \epsilon$ which implies, as above, $v^TWR =0$. Since $R\in S(k,k)$ by \Cref{rmk:Grassmann} we have $v^TW =0$ as asserted.
 
\end{proof}

In the following, recall the definition of a Whitney stratified space of type (A) and (B), respectively, taken from the wording of \citet[Section~10.6]{huckemann_statistical_2020}.

\begin{definition}\label{def:whitney-stratification}
    A \emph{stratified space} $\mc{S}$ of dimension $m$ embedded in a Euclidean space (possibly of higher dimension $M \geq m$) is a direct sum \begin{equation*}
        \mc{S} = \bigsqcup_{i=1}^k S_i
    \end{equation*}
    such that $0\leq d_1 < \ldots < d_k = m$, each $S_{i}$ is a $d_i$-dimensional manifold and $S_{i}\cap S_{j} = \emptyset$ for $i\neq j$ and if $S_{i}\cap\overline{S_{j}}\neq\emptyset$ then $S_{i}\subset\overline{S_{j}}$.

    A stratified space $\mc{S}$ is \emph{Whitney stratified of type (A)}, 
    \begin{itemize}
        \item[(A)] if for a sequence $q_1,q_2,\dots\in S_j$ that converges to some point $p\in S_i$, such that the sequence of tangent spaces $T_{q_n}S_j$ converges  in the Grassmannian $G(M,d_j)$ to a $d_j$-dimensional linear space $T$ as $n\to\infty$, then $T_p S_{i}\subseteq T$, where all the linear spaces are seen as subspaces of $\RR^M$.
    \end{itemize}
    Moreover, a stratified space $\mc{S}$ is a \emph{Whitney stratified space of type (B)},
    \begin{itemize}
        \item[(B)] if for sequences $p_1,p_2,\dots\in S_{i}$ and $q_1,q_2,\dots\in S_{j}$ which converge to the same point $p\in S_{i}$ such that the sequence of secant lines $c_n$ between $p_n$ and $q_n$ converges to a line $c$ as $n\to\infty$ (in the Grassmannian $G(M, 1)$), and such that the sequence of tangent planes $T_{q_n}S_{j}$ converges to a $d_j$-dimensional plane $T$ as $n\to\infty$ (in the Grassmannian $G(M, d_j)$), then $c\subset T$.
    \end{itemize}
\end{definition}

\begin{theorem}\label{thm:whitney}
    Wald space with the smooth structure on every \grove $\G{E}$ conveyed by $\phi_E$ from (\ref{eq:phi-map-restricted-to-grove}), 
    is a Whitney stratified space of type (A).
\end{theorem}

\begin{proof}
    First, we show that $\W$ is a stratified space.
    In conjunction with Remark \ref{rmk:grove-rep}, the manifolds $S_i$ of dimension $d_i=i$ are the unions over disjoint \groves of $\W$ of equal dimenison $i=0,\dots,2N-3 =m$, counting the number of edges, each diffeomorphic to an $i$-dimensional open unit cube,
    \begin{equation*}
        S_i = \bigsqcup_{\vert E\vert = i} \G{E}\,.
    \end{equation*}
    If $S_{i}\cap\overline{S_{j}}\neq\emptyset$ for some  $0\leq i\neq j\leq m$ then there are \waldtops $E,E'$  with $j=\vert E\vert$, $i=\vert E' \vert$ and $\G{E'} \cap \in \overline{\G{E}}\neq\emptyset$, implying  $\G{E'} \subset \overline{\G{E}}$ by Theorem \ref{thm:boundary-of-groves}. In particular, then $i<j$. Further, if $\widetilde{E}'$ with $i=\vert E' \vert$  is any other \waldtop, induction on \Cref{lem:groves-end-basics} (xiii) shows that it can be extended  to a \waldtop $\widetilde{E}$ with  $j=\vert E\vert$ such that $\widetilde{E}'<\widetilde{E}$ and hence $\G{\widetilde{E}'} \subset \overline{\G{\widetilde{E}'}}$ by Theorem \ref{thm:boundary-of-groves}. Thus, we have shown that $S_{i}\subset\overline{S_{j}}$,  as required.
    
    In order to show Whitney condition (A), it suffices to assume $i\neq j$. Let $F_1,F_2,\dots\in S_j$ be a sequence of \walds that converges to some \wald $F' = (E',\lambda') \in S_i$, so $i<j$.  Since $S_j$ is a disjoint union of finitely many \groves, w.l.o.g. we may assume that $F_1,F_2,\dots\in\G{E}$ for some \waldtop $E> E'$ with $\vert E\vert = j$. Hence, under the hypothesis that $T_{F_n}\G{E} \cong T_{\Phi(F_n)} \Phi_{E}(\G{E}) \subset\sym$ converges in the Grassmannian $G(\dim(\sym),j)$, to a $j$-dimensional linear space $T\subset\sym$ as $n\to\infty$, we need to show that
    \begin{equation}\label{eq:proof-thm-Whithney-A-1}
    T_{F'}\G{E'} \cong T_{\Phi(F')} \Phi_{E'}(\G{E'}) \subseteq T\,.
    \end{equation}
    With the analytic continuation $\bar\phi_E$ of $ \phi_E$, see Remark \ref{rmk:phi-analytic-cont}, a cluster point $\lambda^* \in [0,1]^E$ of $\lambda^{(n)} = \phi_E^{-1}(F_n) \in (0,1)^E$, $F' = \phi^{-1}\circ \bar\phi_E (\lambda^*)$, see Theorem \ref{thm:wald-convergence-general}, and the unit standard basis $\partial/\partial \lambda_e$, $e\in E$ of $\G{E} \cong (0,1)^E$ we have thus
    $$ T_{\Phi(F_n)} \Phi_{E}(\G{E}) = \vspan\bigg\{\frac{\dd\bar\phi_E}{\dd\lambda_e}\big(\lambda^{(n)}\big)\colon e\in E\bigg\}\,$$
    and, due to \Cref{lem:limits-of-subspaces}, 
    \begin{align*}
        \vspan\bigg\{\frac{\dd\bar\phi_E}{\dd\lambda_e}\big(\lambda^*\big)\colon e\in E\bigg\} &= \vspan\bigg\{\lim_{n\to\infty}\frac{\dd\bar\phi_E}{\dd\lambda_e}\big(\lambda^{(n)}\big)\colon e\in E\bigg\}\\
        &\subseteq\lim_{n\to\infty}\vspan\bigg\{\frac{\dd\bar\phi_E}{\dd\lambda_e}\big(\lambda^{(n)}\big)\colon e\in E\bigg\} = T.
    \end{align*}
    Since likewise
    $$ T_{\phi(F')}\phi(\G{E'}) = \vspan\bigg\{\frac{\dd\phi_{E'}}{\dd\lambda'_{e'}}(\lambda')\colon e'\in E'\bigg\}$$
    showing assertion (\ref{eq:proof-thm-Whithney-A-1}) is equivalent to showing 
    \begin{equation*}
     \vspan\bigg\{\frac{\dd\phi_{E'}}{\dd\lambda'_{e'}}(\lambda')\colon e'\in E'\bigg\} \subseteq \vspan\bigg\{\frac{\dd\bar\phi_E}{\dd\lambda_e}\big(\lambda^*\big)\colon e\in E\bigg\}\,.
    \end{equation*}
    To see this, it suffices to show that for each $e'\in E'$, there exists a constant $c>0$ and an edge $e\in E$ such that
    \begin{equation}\label{eq:proof-thm-Whithney-A-2}
        \frac{\dd\phi_{E'}}{\dd\lambda'_{e'}}(\lambda') = c\,\frac{\dd\bar\phi_E}{\dd\lambda_e}(\lambda^*).
    \end{equation}
    In the following we show (\ref{eq:proof-thm-Whithney-A-2}).
    
    Recalling  for $u,v\in L$  
    $$
        \Big(\bar\phi_E(\lambda^*)\Big)_{uv} 
        = \prod_{e\in E(u,v)} \big(1 - \lambda^*_e\big)\,,\quad
        \Big(\phi_{E'}(\lambda')\Big)_{uv} 
        = \prod_{e'\in E'(u,v)} \big(1 - \lambda'_{e'}\big)$$
     from \Cref{rmk:grove-cube-identification}, obtain their derivatives
    \begin{align}\label{eq:proof-thm-Whithney-A-3}
        \bigg(\frac{\dd\bar\phi_E}{\dd\lambda_e}(\lambda^*)\bigg)_{uv} 
        &= -\one_{e\in E(u,v)} \prod_{\substack{\tilde{e}\in E(u,v)\\\tilde{e}\neq e}} \big(1 - \lambda^*_{\tilde{e}}\big),\\\label{eq:proof-thm-Whithney-A-4}
        \bigg(\frac{\dd\phi_{E'}}{\dd\lambda'_{e'}}(\lambda')\bigg)_{uv} 
        &= -\one_{e'\in E'(u,v)} \prod_{\substack{\tilde{e}'\in E'(u,v)\\\tilde{e}'\neq e'}} \big(1 - \lambda'_{\tilde{e}'}\big).
    \end{align}
    Recall from \Cref{cor:closure-of-grove-consists-of-other-groves} the two relationships between $F'$ and $\bar\phi_E(\lambda^*)$:
    \begin{equation*}
        E'_{\alpha'} = \big\{e'\colon e' = e\res{L'_{\alpha'}} \text{ is a valid split of } L'_{\alpha'}, e\in E \text{ and } \lambda^*_e \neq 0 \big\}
    \end{equation*}
    as well as 
    for each $e'\in E'$, 
    \begin{equation}
        \label{eq:formula-for-edge-lengths-wrt-boundary}
        \lambda'_{e'} = 1- \prod_{e\in R_{e'}} (1-\lambda^*_e)\neq 0.
    \end{equation}
    Consequently, for any $e'\in E'_{\alpha'}$ there exists $e\in R_{e'}$ with $\lambda^*_e\neq 0$.
    
    Now, let $u,v\in L$ be arbitrary and for every $e'\in E'$, we consider $e$ as above.
    \begin{enumerate}
        \item 
            Case $e\notin E(u,v)$. Then $(\bar \phi_E(\lambda^*))_{uv}$ is constant as $\lambda_e$ varies and since
            $R_{e'} \ni e \not\in E(u,v)$, i.e. $R_{e'}\not\subseteq E(u,v)$, we have $e'\not \in E'(u,v)$ by \Cref{lem:groves-end-basics} (viii) so that likewise $(\phi_{E'}(\lambda'))_{uv}$ is constant as $\lambda'_{e'}$ varies, yielding
            \begin{equation*}
            \bigg(\frac{\dd\phi_{E'}}{\dd\lambda'_{e'}}(\lambda')\bigg)_{uv} = 0 = \bigg(\frac{\dd\bar\phi_E}{\dd\lambda_e}(\lambda^*)\bigg)_{uv}.
            \end{equation*}
            Thus for $c$ in \eqref{eq:proof-thm-Whithney-A-2} any positive constant can be chosen.
        \item 
            Case $e=A\vert B\in E(u,v)$. W.l.o.g. assume that $u \in A$ and $v\in B$. Then there are two subcases:
            \begin{enumerate}
                \item 
                    $e'\notin E'(u,v)$. 
                    On the one hand, as above this implies $\big(\frac{\dd\phi_{E'}}{\dd\lambda'_{e'}}(\lambda')\big)_{uv} =0$, on the other hand, as $e' = A\cap L'_{\alpha'}\vert B\cap L'_{\alpha'}\notin E'(u,v)$,  either $u\notin L'_{\alpha'}$ or $v\notin L'_{\alpha'}$, implying 
                    \begin{equation*}
                        0 = \big(\phi_{E'}(\lambda')\big)_{uv} = \big(\bar\phi_E(\lambda^*)\big)_{uv} =\prod_{\tilde{e}\in E(u,v)} \big(1 - \lambda^*_{\tilde{e}}\big)\,.
                    \end{equation*}
                    Thus $\lambda^*_{\tilde{e}} = 1$ for some $\tilde{e}\in E(u,v)$ with $\tilde{e}\neq e$ (recall that $\lambda^*_e <1$ for otherwise $\lambda'_{e'} =1$ by \eqref{eq:formula-for-edge-lengths-wrt-boundary}), which implies in conjunction with \eqref{eq:proof-thm-Whithney-A-3} that
                    \begin{equation*}
                        \bigg(\frac{\dd\bar\phi_E}{\dd\lambda_e}(\lambda^*)\bigg)_{uv} = 0 = \bigg(\frac{\dd\phi_{E'}}{\dd\lambda'_{e'}}(\lambda')\bigg)_{uv}.
                    \end{equation*}
                    Again,  for $c$ in \eqref{eq:proof-thm-Whithney-A-2}) any positive constant can be chosen.
                \item 
                    $e'\in E'(u,v)$.
                    Then \Cref{lem:groves-end-basics} (viii) yields $R_{e'} \subseteq E(u,v)$ and we have, invoking \eqref{eq:proof-thm-Whithney-A-4} as well as \eqref{eq:formula-for-edge-lengths-wrt-boundary}, that 
                    \begin{eqnarray} \nonumber
                        \bigg(\frac{\dd\phi_{E'}}{\dd\lambda'_{e'}}(\lambda')\bigg)_{uv} 
                        ~=~ -\prod_{\substack{\tilde{e}'\in E'(u,v)\\\tilde{e}'\neq e'}} \big(1 - \lambda'_{\tilde{e}'}\big)
                        &=& -\prod_{\substack{\tilde{e}'\in E'(u,v)\\\tilde{e}'\neq e'}} \prod_{\tilde{e} \in R_{\tilde{e}'}} \big(1 - \lambda^*_{\tilde{e}}\big)\
                        \\ \label{eq:proof-thm-Whithney-A-6}
                        &=& -\prod_{\substack{\tilde{e}\in E(u,v)\\\tilde{e}\notin R_{e'}}} \big(1 - \lambda^*_{\tilde{e}}\big)\,,
                    \end{eqnarray}
                    where the last equality follows from observing that $\tilde{e}\notin R_{e'} \Leftrightarrow \exists \tilde{e}' \in E', \tilde{e}'\neq e'\st \tilde{e}\in R_{\tilde{e}'}$, due to  \Cref{lem:groves-end-basics} (vi).
                    Furthermore, again by \eqref{eq:proof-thm-Whithney-A-3}, recalling from above that $R_{e'}  \subseteq E(u,v)$ and \eqref{eq:proof-thm-Whithney-A-6},
                    \begin{align*}
                        \bigg(\frac{\dd\bar\phi_E}{\dd\lambda_e}(\lambda^*)\bigg)_{uv}
                        &= - \prod_{\substack{\tilde{e}\in E(u,v)\\\tilde{e}\neq e}} \big(1 - \lambda^*_{\tilde{e}}\big)\\
                        &= -\bigg(\prod_{\substack{\tilde{e}\in R_{e'}\\\tilde{e}\neq e}} \big(1 - \lambda^*_{\tilde{e}}\big)\bigg)
                        \bigg(\prod_{\substack{\tilde{e}\in E(u,v)\\\tilde{e}\notin R_{e'}}} \big(1 - \lambda^*_{\tilde{e}}\big)\bigg)\\
                        &= \bigg(\prod_{\substack{\tilde{e}\in R_{e'}\\\tilde{e}\neq e}} \big(1 - \lambda^*_{\tilde{e}}\big)\bigg) \bigg(\frac{\dd\phi_{E'}}{\dd\lambda'_{e'}}(\lambda')\bigg)_{uv}\,.
                    \end{align*}
                    Thus $$c = \prod_{\substack{\tilde{e}\in R_{e'}\\\tilde{e}\neq e}} \big(1 - \lambda^*_{\tilde{e}}\big)$$ satisfies \eqref{eq:proof-thm-Whithney-A-2} as it does not depend on $u$ and $v$ and is non-zero by \Cref{eq:formula-for-edge-lengths-wrt-boundary}.
            \end{enumerate}
    \end{enumerate}
        Having thus shown \eqref{eq:proof-thm-Whithney-A-2}, as detailed above we have established \eqref{eq:proof-thm-Whithney-A-1} thus verifying Whitney condition (A).
\end{proof}

Whitney condition (B) is a conjecture.

\section{Information Geometry for Wald Space}\label{scn:geometry}

In  \cite{garba_information_2021} we equipped the space of phylogenetic forests with a metric induced from the metric of the \emph{Fisher-information} Riemannian metric $g$ on $\spd$ (see \Cref{scn:notation}), where the latter induces the metric $d_\spd$  on $\spd$. In this section we show,  first that this induced metric is compatible with the stratification structure of $\W$, and second that this turns $\W$ into a geodesic Riemann stratified space. 
\subsection{Induced Intrinsic Metric}

In \cite{garba_information_2021} we introduced a metric on $\cW$ induced from the geodesic distance metric $d_\spd$ of $\cP$ introduced in Section \ref{scn:notation}. Recalling also the definition of path length $L_\spd$ from Section \ref{scn:notation}, for two \walds $F,F'\in \W$, set
    \begin{equation*}
      d_{\W}(F,F') := \inf_{\substack{\gamma\colon[0,1]\to\W\\\phi\circ\gamma \text{ continuous in }\spd,\\\gamma(0)=F, \gamma(1)=F'}} L_{\spd}(\phi\circ\gamma)\,.
    \end{equation*}
This metric defines the \emph{induced intrinsic metric} topology on $\W$. While in general this topology may be finer than the one conveyed by making an embedding a homeomorphism, as the following example teaches, this is not the case for \waldspace.

\begin{example}
  \label{exa:intrinsic-topology-finer}
  Consider 
  $$ \cM := (\{1\}\times [0,1]) \quad\cup \bigcup_{y \in  \{1/n : n \in \NN\}\cup \{0\}} 
  [-1,1) \times \{y\} \,,$$
  an infinite union of half open intervals in $\RR^2$ connected vertically on the right. In the trace topology where the canonical embedding $\iota : \cM \hookrightarrow \RR^2$ is a homeomorphism, the sequence $q_n = (0,1/n)$ converges to $q=(0,0)$. For the induced intrinsic metric
  \begin{equation*}
    d_{\cM}(x, y) = \inf_{\substack{\gamma\colon[0,1]\to \cM\\\gamma \text{  continuous in } \RR^2,\\\gamma(0)=x, \gamma(1)=y}} L_{\RR^2}(\gamma)\,,
  \end{equation*}
  with the Euclidean length $L_{\RR^2}$, we have, however, $d_{W}(q_n, q) \geq  2$ for all $n\in \NN$.
  
\end{example}

\begin{theorem}
    The topology of $\W$ obtained from making $\phi$ a homeomorphism agrees with the topology induced from the induced intrinsic metric $d_\W$. In particular $d_\W$ turns $\W$ into a metric space.
\end{theorem}

\begin{proof}
    By definition we have that $d_\W \geq d_\spd$, which implies that sequences that converge with respect to $d_\W$ also converge with respect to $d_\spd$.
    
    For the converse, assume that $\W\ni F_n \to F' \in \W$ w.r.t. $d_\spd$, as $n\to \infty$. Since there are only finitely many groves in $\W$ it suffices to show that $d_\W(F_n,F)\to 0$ for $F_n \in \G{E}$ and $F \in \overline{\G{E}}$ with a common grove $\G{E}$. Hence, we assume that $\bar\phi^{-1}_E\circ \phi(F_n) = \lambda_n \in  (0,1)^E$ and $\bar\phi^{-1}_E\circ\phi(F') = \lambda' \in [0,1]^E$ with $\lambda_n \to \lambda'$, due to Theorem \ref{thm:wald-convergence-general}. Then, with   $\delta(t) = t \lambda' + (1-t) \lambda_n $,
    $$\gamma: [0,1] \to \spd, t\mapsto \phi^{-1}\circ\bar\phi_E\circ\delta(t))=: \gamma(t) $$ 
    is a path in $\W$ connecting $\gamma(0) = F$ with $\gamma(1) = F_n$. For $k \in \NN$ and $j=1,\ldots,k$ we note that
    \begin{eqnarray*}
    \bar\phi_E \circ \delta_n\left(\frac{j}{k}\right) &=& \bar\phi_E \circ \delta_n\left(\frac{j-1}{k}\right) + (D\bar\phi_E) \circ \delta_n\left(\frac{j-1}{k}\right) \cdot \frac{\lambda'-\lambda_n}{k} + o\left(\frac{\|\lambda'-\lambda_n\|}{k}\right)
    \end{eqnarray*}
    where both terms 
    $$\bar\phi_E \circ \delta_n\left(\frac{j-1}{k}\right), (D\bar\phi_E) \circ \delta_n\left(\frac{j-1}{k}\right)  $$ 
    are bounded, also uniformly $n\in \NN$, due to Remark \ref{rmk:phi-analytic-cont}. In consequence, in conjunction with Section \ref{scn:notation},
    \begin{eqnarray*}
    \lefteqn{
    d_\W(F_n,F)}\\ &\leq& \lim_{k\to \infty} \sum_{j=1}^k d_\spd \left(\bar\phi_E \circ
    \delta\left(\frac{j-1}{k}\right), \bar\phi_E \circ \delta\left(\frac{j}{k}\right)\right)\\
    &=& \lim_{k\to \infty} \sum_{j=1}^k\left\| \log\left(\sqrt{\bar\phi_E \circ
    \delta\left(\frac{j-1}{k}\right)}^{-1}\bar\phi_E \circ
    \delta_n\left(\frac{j}{k}\right)\sqrt{\bar\phi_E \circ
    \delta_n\left(\frac{j-1}{k}\right)}^{-1} \right)\right\|\\
    &\leq& C \|\lambda'-\lambda_n\|
    \end{eqnarray*}
    with a constant $C>0$ independent of $n$. Letting $n\to \infty$ thus  yields the assertion.
    
\end{proof}

\subsection{Geodesic Space and Riemann Stratification}
Having established the equivalence between the stratification topology and that of the Fisher information metric, we longer distinguish between them.

\begin{theorem}\label{thm:waldspace-is-geodesic-metric-space}
    The \waldspace equipped with the information geometry is a geodesic metric space, i.e.~every two points in $(\W, d_\W)$ are connected by a minimising geodesic.
\end{theorem}
  
\begin{proof}
    By \cite[p.325]{lang_fundamentals_1999}, $(\spd, g)$ is geodesically complete as a Riemannian manifold and thus by the Hopf-Rinow Theorem for Riemannian manifolds (among others, \cite[p.224]{lang_fundamentals_1999}), it follows that $(\spd, d_\spd)$ is complete and locally compact.
    By \Cref{cor:forests-closed-in-spd}, $\phi(\W)$ is a closed subset of the complete and locally compact metric space $\spd$ and so $(\phi(\W), d_\spd)$ itself is, and so is $(\W,d_\spd)$. 
    By \cite[Theorem~5.1]{garba_information_2021}, any two \walds in are connected by a continuous path of finite length in $(\W,d_\spd)$, which is complete, and thus applying \cite[Corollary on p.123]{hu_local_1978} yields that $(\W,d_\W)$ is complete.
    Applying the Hopf-Rinow Theorem for metric spaces \cite[p.35]{bridson_metric_1999} to $(\W,d_\W)$, the assertion holds.
\end{proof}

Following \citet[Section 10.6]{huckemann_statistical_2020}, extend the notion of a Whitney stratified space in \Cref{def:whitney-stratification} to the notion of a Riemann stratified space.
\begin{definition}
    A \emph{Riemann stratified space} is a Whitney stratified space $\mc{S}$ of type (A) such that each stratum $S_{i}$ is a $d_i$-dimensional Riemannian manifold with Riemannian metric $g^{i}$, respectively, if whenever a sequence $q_1,q_2,\dots\in S_{j}$ which converges to a point $p\in S_{i}$ (where, assume again that the sequence of tangent planes $T_{q_n} S_{j}$ converges to some $d_j$-dimensional plane $T$ as $n\to\infty$), then the Riemannian metric $g_{q_n}^{j}$ converges to some two form $g_p^*\colon T\otimes T\to\RR$ with $g_p^{i}\equiv g_p^*\res{T_pS_i \otimes T_pS_{i}}$.
\end{definition}

\begin{theorem}
    The \waldspace $\W$ equipped with the information geometry is a Riemann stratified space.
\end{theorem}

\begin{proof}
    As we impose the Riemannian metric $g$ from $\spd$ onto all of $\phi(\W)\subset\spd$, the assertion follows immediately.
\end{proof}
\begin{example}[Geometry of \waldspace for $N = 2$]\label{exa:waldspace-2-geometry-analytical}
    For $N = 2$, $L = \{1,2\}$, there is one edge $e = 1\vert 2$, and two different topologies, namely 
    \begin{equation*}
        E = \{1\vert2\}\qquad\text{ and }\qquad E' = \emptyset.
    \end{equation*}
    The corresponding \groves are then $\phi(\G{E'}) = \{I\}$, where $I$ is the $2\times 2$ unit matrix, and $\G{E}\cong(0,1)$ such that $\W = \G{E} \sqcup \G{E'} \cong (0,1]$. Using $\lambda\in(0,1)$ for the only edge $e = 1\vert 2$ we have
    \begin{align*}
        \phi(\G{E}) = \bigg\{
        \phi_{E}(\lambda) =
        \begin{pmatrix}
            1 & 1 - \lambda \\
            1 - \lambda & 1
        \end{pmatrix}
        \colon \lambda\in(0,1) 
        \bigg\},
    \end{align*}
    Thus, with the definition of $d_\spd$ in Section \ref{scn:notation}, the distance between two phylogenetic forests $F_1=\bar\phi_{E}(\lambda_1)$, $F_2=\bar\phi_{E}(\lambda_2)$ with $\lambda_1,\lambda_2\in(0,1]$ can be calculated as
    \begin{align*}
      d_{\W}\big(F_1,F_2\big) 
      = \Bigg\vert &\ln\Bigg(\frac{1 - \lambda_2 + \frac{1}{\sqrt{2}}p(\lambda_2)}{1 - \lambda_1 + \frac{1}{\sqrt{2}}p(\lambda_1)}\Bigg)\\
      &+ \frac{1}{2\sqrt{2}}\ln\Bigg(\frac{\big(p(\lambda_1) + (1 - \lambda_1)\big)^2 - 1}
      {\big(p(\lambda_1) - (1 - \lambda_1)\big)^2 - 1}\,\cdot\,\frac{\big(p(\lambda_2) - (1 - \lambda_2)\big)^2 - 1}
      {\big(p(\lambda_2) + (1 - \lambda_2)\big)^2 - 1}\Bigg)\Bigg\vert ,
    \end{align*}
    where $p(x) = \sqrt{2}\sqrt{(1 - x)^2 + 1}$ for $x\in[0,1]$. Fig.~\ref{fig:waldspace-2-distance-to-forest-analytical} (right) depicts  the distance as a function of $\lambda_1,\lambda_2$.
    We obtain the distance to the disconnected forest $F_\infty = \phi^{-1}\bar\phi_E(1)$:
    \begin{align*}
      d_{\W}\big(F_1, F_\infty\big)  = \bigg\vert \frac{1}{2\sqrt{2}}\ln\bigg(\frac{(p(\lambda_1) + (1 - \lambda_1))^2 - 1}{(p(\lambda_1) - (1 - \lambda_1))^2 - 1}\bigg) - \ln\Big(1 - \lambda_1 + \frac{1}{\sqrt{2}}p(\lambda_1)\Big)\bigg\vert .
    \end{align*}
    This distance is depicted (as a function in $\lambda_1$) in Fig.~\ref{fig:waldspace-2-distance-to-forest-analytical} (left).
\end{example}

\begin{figure}[ht]
    \centering
    \begin{minipage}{0.39\linewidth}
      \fbox{
        \includegraphics[width=1.0\linewidth, trim=0.5cm 0.0cm 1.0cm 0.0cm, clip]{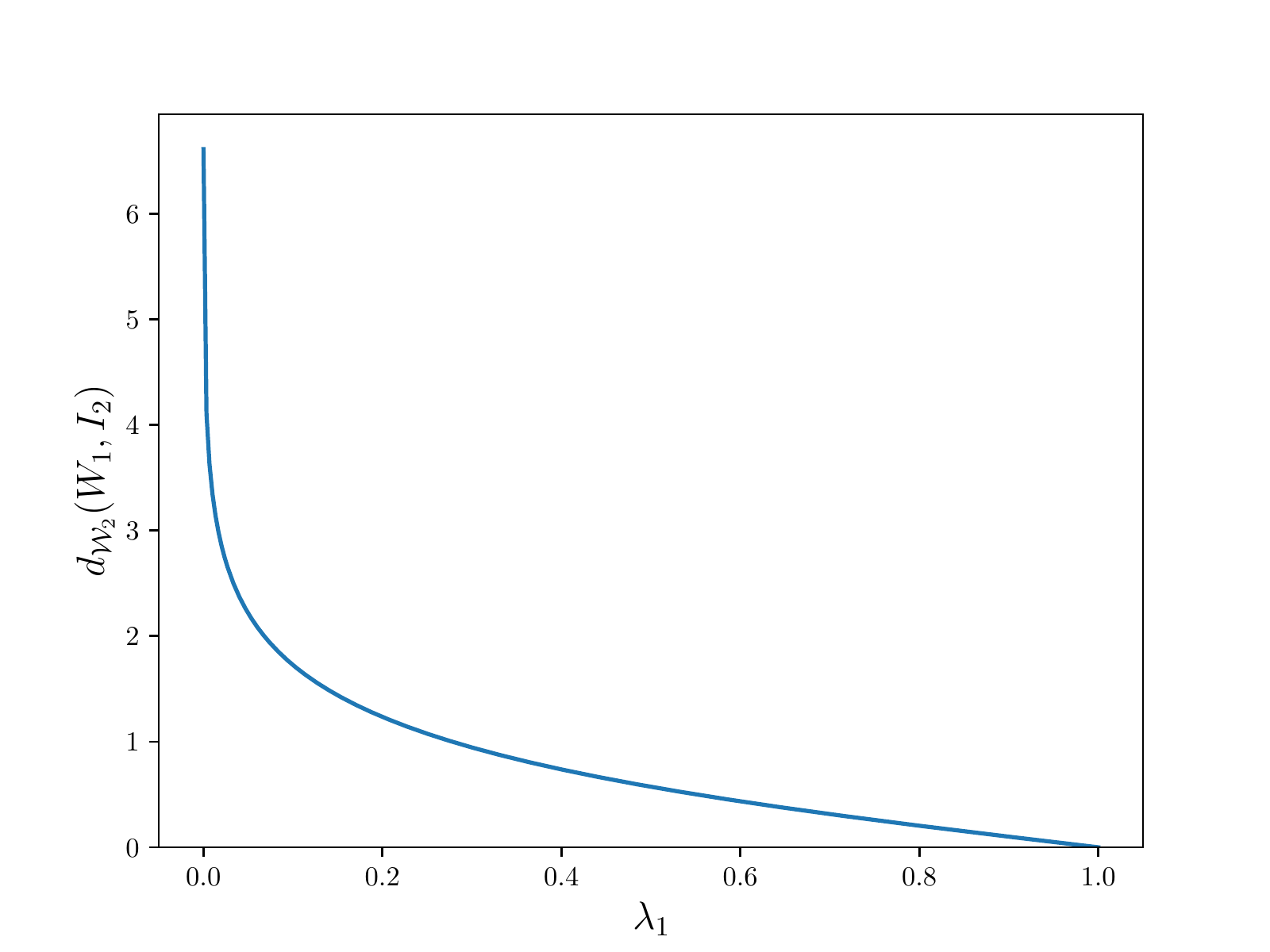}
      }
    \end{minipage}\hfil
    \begin{minipage}{0.39\linewidth}
      \fbox{
        \includegraphics[width=1.0\linewidth, trim=0.5cm 0.0cm 1.0cm 0.0cm, clip]{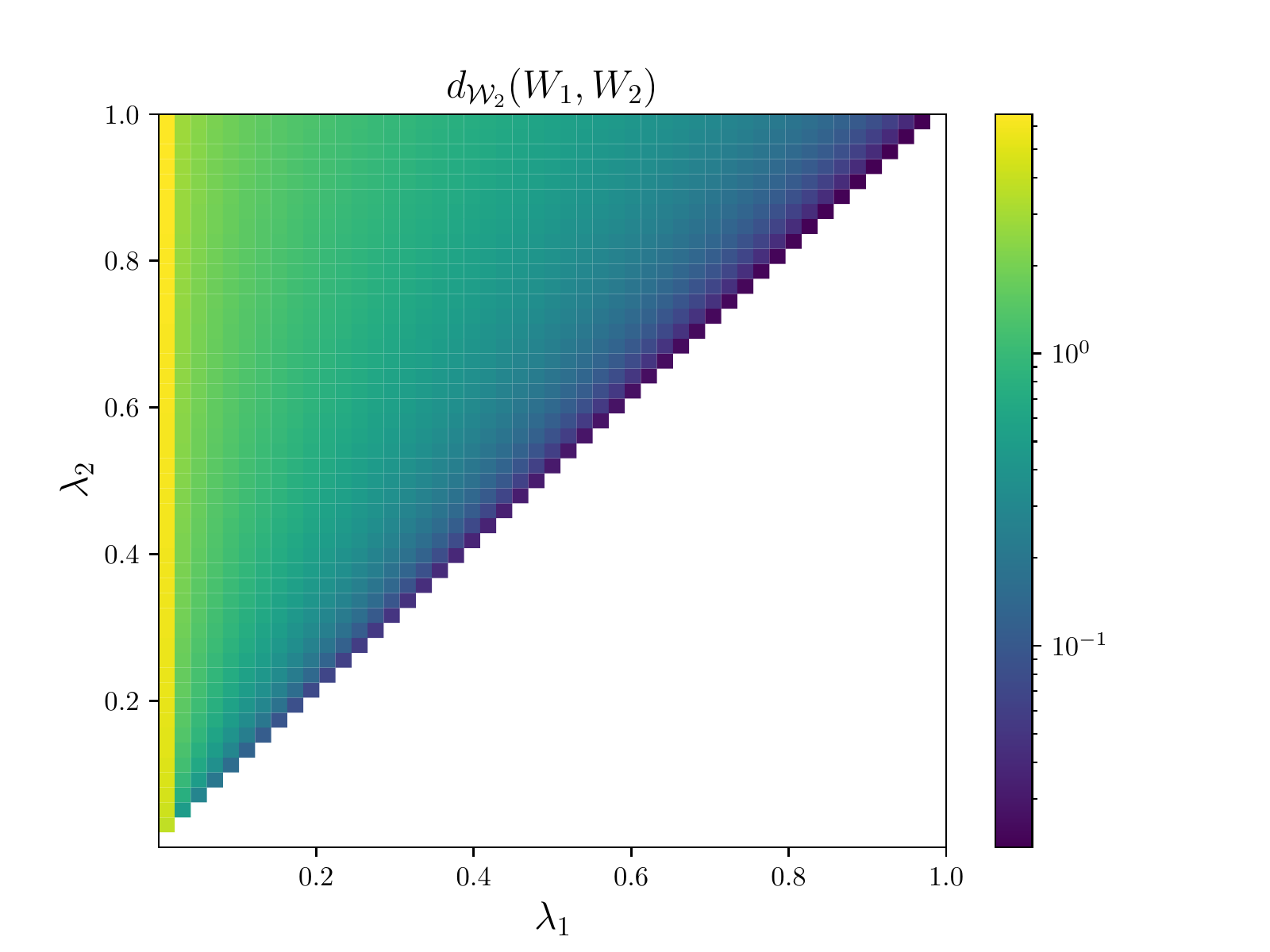}
      }
    \end{minipage}%
    \caption{Depicting the distance $d_{\W}$ of an arbitrary \wald, to the disconnected forest $F_\infty$ as a function of $\lambda_1$ (left) and between any two other \walds in $\W$, as a function of $(\lambda_1,\lambda_2)$ (right) as detailed in \Cref{exa:waldspace-2-geometry-analytical}.}
    \label{fig:waldspace-2-distance-to-forest-analytical}
\end{figure}

\section{Numerical Exploration of Wald Space}\label{scn:numerical}
In this section we propose a new algorithm to approximate geodesics between two fully resolved trees $F_1$ and $F_2$, that is a mixture of the \emph{successive projection} algorithm and the \emph{extrinsic path straightening} algorithm from \cite{lueg_wald_2021}. Using this algorithms allows to explore curvature and so-called \emph{stickiness} of Fr\'echet means.

\subsection{Approximating Geodesics in Wald Space}
From the ambient geometry of $\spd$, recalling the notation from Section \ref{scn:notation}, we employ the globally defined Riemannian exponential $\Exp$ and logarithm $\Log$ at $P\in\spd$, with $Q\in\spd$, $X\in T_P\spd$, as well as points on the unique (if $P\neq Q$) geodesic $\gamma_{P,Q}$ in $\spd$ comprising $P$ and $Q$. Further, $\exp$ and $\log$ denote the matrix exponential and logarithm, respectively:
\begin{align*}
    \Exp_P& \colon T_P\spd\to\spd,& &X \mapsto 
    \sqrt{P} \exp\big(\sqrt{P}^{-1} X \sqrt{P}^{-1}\big) \sqrt{P}, \\
    \Log_P& \colon\spd\to T_P\spd,& &Q \mapsto 
    \sqrt{P} \log\big(\sqrt{P}^{-1} Q \sqrt{P}^{-1}\big) \sqrt{P}, \\
    \gamma_{P,Q}^{(\spd)}& \colon[0,1]\to\spd,& &t\mapsto 
    \Exp_P\big(t\,\Log_P(Q)\big),
\end{align*}
Furthermore, for a forest $F\in\W$ with topology $E$, denote the orthogonal projection from the tangent space $T_P\spd$ at $P = \phi(F)$ onto the tangent space of the sub-manifold $T_P\phi_E(\G{E})$, as a subspace of $T_P\spd$, with
\begin{equation*}
    \pi_P \colon T_P\spd\to T_P\phi_E(\G{E}).
\end{equation*}
This projection is computed using an orthonormal basis of $T_P\phi_E(\G{E})$ obtained from applying Gram-Schmidt to the basis
\begin{equation*}
    \bigg\{\frac{\dd\phi_{E}}{\dd\lambda_e}(\lambda) \colon e\in E\bigg\}.
\end{equation*}
 of $T_P\phi_E(\G{E})$. Finally, we make use of the projection 
\begin{equation*}
    \pi\colon \spd\to\W,\qquad P\mapsto \argmin_{F\in\W} d_{\spd}\big(P,\phi(F)\big),
\end{equation*}
where $\pi$ is well-defined for $P\in\spd$ close enough to $\phi(\W)$.
The following algorithm is similar to the \emph{extrinsic path straightening} algorithm from \cite{lueg_wald_2021}, which has been inspired by \cite{schmidt_shape_2006}. 
It starts with generating a discrete curve using the \emph{successive projection} algorithm from \cite{lueg_wald_2021} and then iteratively straightening it and adding more points in between the points of the discrete curve.
To keep notation simple, we omit $\phi$ and identify a forest $F\in\W$ with its matrix representation $\phi(F)$.

\begin{definition}[Geodesic Approximation Algorithm]\label{def:geodesic-approximation-algorithm}
Let $5\leq n_0\in\NN$ be the odd number of points in the \emph{initial path}, $I \in\NN$ the number of \emph{extensions iterations} and $J\in\NN$ the number of \emph{straightening iterations}  of the path.
\begin{description}
    \item[Input] $F,F'\in\W$ 
    \item[Initial path]  Set $F_1\coloneqq F$, $F_{n_0}\coloneqq F'$, then, for $i = 2,\dots,(n_0-3)/2$, compute
            \begin{align*}
                F_i &= \pi\bigg(\gamma_{F_{i-1},F_{n_0 - i + 2}}\Big(\frac{1}{n_0 - 2i + 3}\Big)\bigg),\\
                F_{n_0 - i +1} &= \pi\bigg(\gamma_{F_{n_0 - i+2},F_{i-1}}\Big(\frac{1}{n_0 - 2i + 3}\Big)\bigg),\\
            \end{align*}
            and, with $F_{(n_0-1)/2} \coloneqq \pi(\gamma_{F_{(n_0-3)/2},F_{(n_0+1)/2}}(0.5))$, set the \emph{current discrete path} to
            \begin{equation*}
                \Gamma \coloneqq \big(F_1,\dots,F_{(n_0-3)/2},F_{(n_0-1)/2}, F_{(n_0+1)/2}, \dots, F_{n_0}\big).
            \end{equation*}
                   \item[Iteratively extend and straighten:] Do $I$ times:
            \begin{description}
                \item[Extend]  With the current discrete path $\Gamma = \big(F_0,\dots,F_{n}\big)$, for $i = 0,\dots,n-1$ compute $G_i \coloneqq \pi\big(\gamma_{F_{i},F_{i+1}}(0.5)\big)$ and define the new current discrete path
                    \begin{equation*}
                        \Gamma \coloneqq \big(F_0,G_0,F_1,G_1,\dots,F_{n-1},G_{n-1},F_n\big).
                    \end{equation*}
                \end{description}
                Do $J$ times:
             \begin{description}
               \item[Straighten] With the current discrete path $\Gamma = \big(F_0,\dots,F_n\big)$,
                for $i = 2,\dots,n-1$, compute $X_i = \frac{1}{2}\big(\Log_{F_i}(F_{i-1}) + \big(\Log_{F_i}(F_{i+1})\big)$, update $F_i\coloneqq \pi\big(\Exp_{F_i}^{(\spd)}(X_i)\big)$, and define the new current discrete path
                    \begin{equation*}
                        \Gamma \coloneqq \big(F_0,F_1,\dots,F_n\big).
                    \end{equation*}
                \end{description}
    \item[Return] The current discrete path  $\Gamma$, which is a discrete approximation of the geodesic between $F$ and $F'$ with $2^I(n_0 - 1) + 1$ points.
\end{description}

\end{definition}

While \Cref{thm:waldspace-is-geodesic-metric-space} guarantees the existence of a shortest path between any $F,F'\in\W$, it may not be unique, and it is not certain whether the path found by the algorithm is near a shortest path or represents just a local approximation.

To better assess the quality of the approximation  $\Gamma = (F_0,\ldots,F_n)$ found by the algorithm, \cite{rumpf2015variational} propose considering its energy,
\begin{equation*}
    E(\Gamma) = \frac{1}{2}\sum_{i = 0}^{n-1} d_{\spd}(F_i, F_{i+1})^2\,,
\end{equation*}
yielding a means of comparison for discrete paths with equal number of points.


\begin{example}[Geodesics in \waldspace for $N = 3$]\label{exa:waldspace-3-geodesics-over-zero-boundary} 
Revisiting $\W$ from \Cref{exa:boundary-of-grove-n3} with unique top-dimensional grove $\G{E}\cong(0,1)^3$, 
we approximate a shortest path between the two phylogenetic forests $F_1,F_2\in\F$ with $\phi(F_1) = \phi_{E}(\lambda^{(1)})$ and $\phi(F_2) = \phi_{E}(\lambda^{(2)})$ using the algorithm from \Cref{def:geodesic-approximation-algorithm}, where
    \begin{equation*}
        \lambda^{(1)} = (0.1, 0.9, 0.07) \qquad \text{ and } \qquad \lambda^{(2)} = (0.3, 0.1, 0.9).
    \end{equation*}
    This path is depicted in \Cref{fig:waldspace-3-shortest-path-boundary}, as well as the BHV space geodesic (which is a straight line with respect to the $\ell$-parametrization from \Cref{def:wald-graph-representation}), first in the coordinates $\lambda\in(0,1)^3$ and second embedded into $\spd$ viewed as $\RR^3$, cf.~\Cref{fig:forests-n3-tetrahedron}.
    In contrast to the BHV geometry, the shortest path in the \waldspace geometry sojourns on the two-dimensional boundary, where the coordinate $\lambda_1$ is zero for some time. 
The end points $\lambda^{(1)},\lambda^{(2)},$ are trees that show a high level of disagreement over the location of taxon 1, but a similar divergence between taxon 2 and taxon 3. 
The section of the approximate geodesic with $\lambda_1=0$ represents trees on which the overall divergence between taxon 1 and the other two taxa is reduced.
In this way, the conflicting information in the end points is resolved by reducing the divergence (and hence increasing the correlation) between taxon 1 and the other two taxa, in comparison to the BHV geodesic which has $\lambda_1>0$ along its length. 
\end{example}

\begin{figure}[ht]
    \centering
    \begin{minipage}{0.39\linewidth}
      \fbox{
        \includegraphics[width=\linewidth, trim=1.cm 0cm 1.cm 0cm, clip]{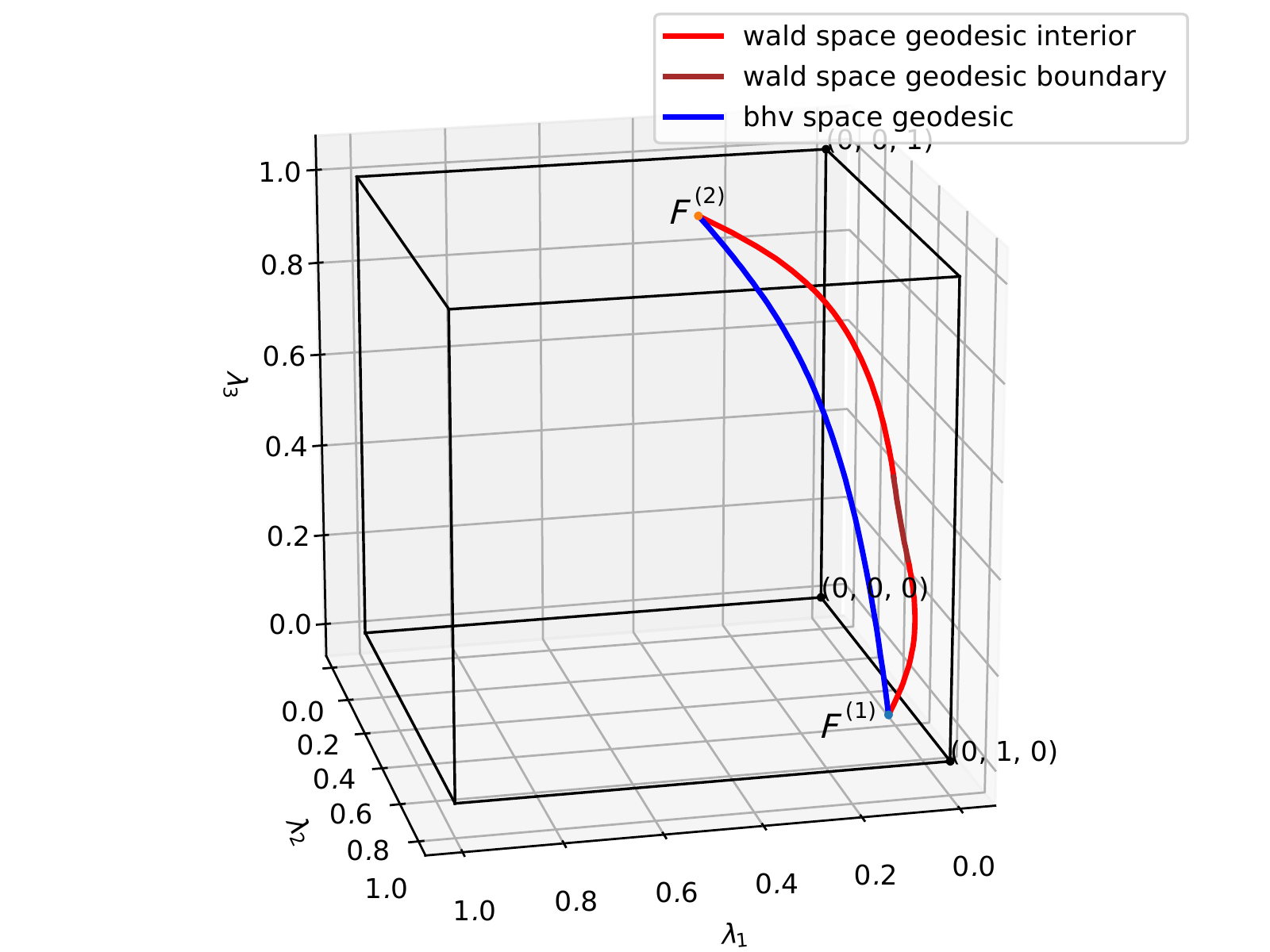}
      }
    \end{minipage}\hfil
    \begin{minipage}{0.39\linewidth}
      \fbox{
        \includegraphics[width=\linewidth, trim=1.cm 0cm 1.cm 0cm, clip]{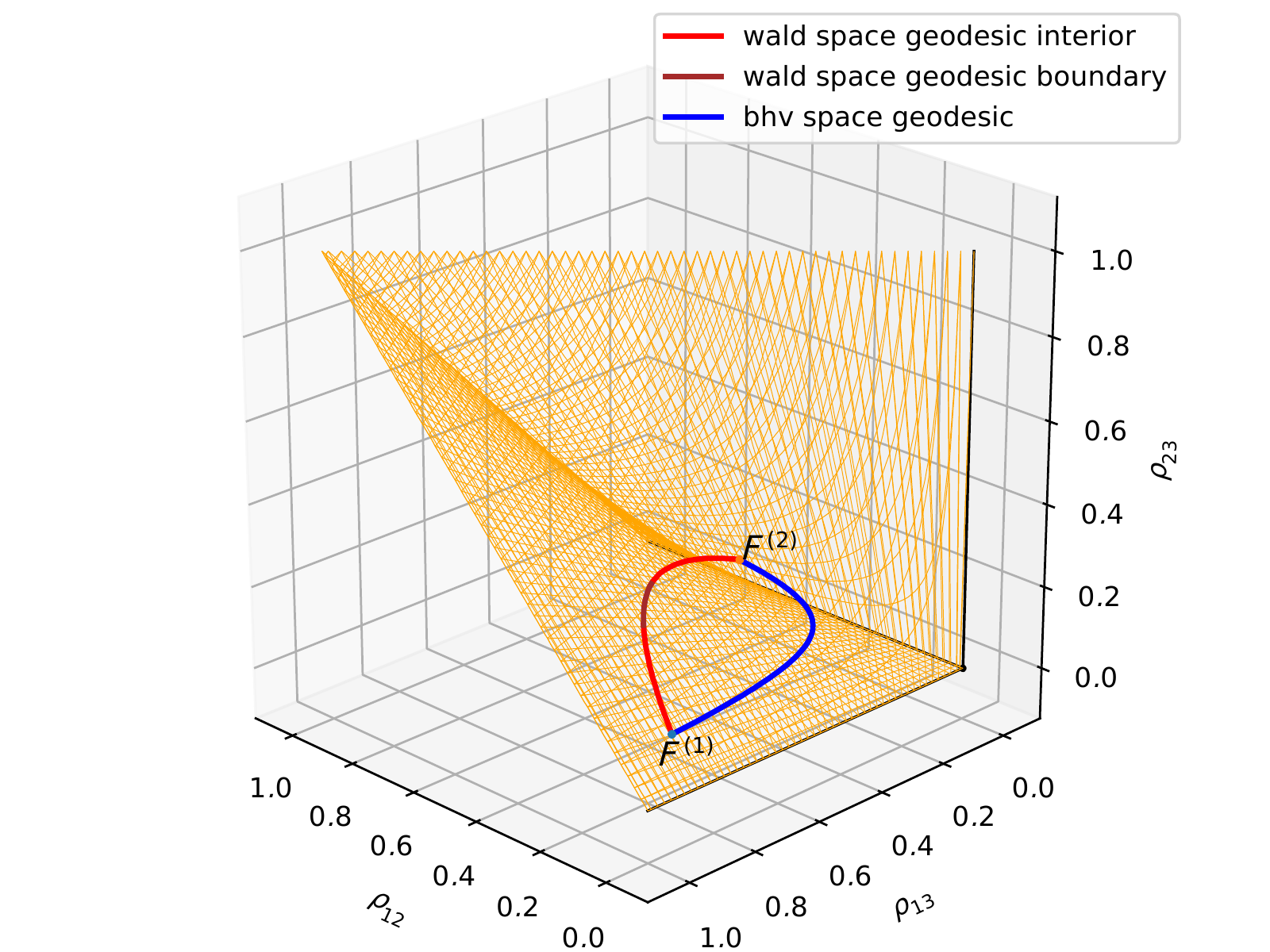}
      }
    \end{minipage}%
    \caption{
     The \waldspace geodesic (red) between fully resolved phylogenetic forests $F_1,F_2 \in \W$ ($N = 3$) sojourns on the boundary (brown). The image of the BHV space geodesic (blue)  remains in the grove as discussed in \Cref{exa:waldspace-3-geodesics-over-zero-boundary}. In $\lambda$-representation (left) and  embedded in $\spd$ viewed as $\RR^3$ (right, cf.~\Cref{fig:forests-n3-tetrahedron}).
    }
    \label{fig:waldspace-3-shortest-path-boundary}
\end{figure}

\subsection{Exploring Curvature of Wald Space} 
Since curvature computations involving higher order tensors are heavy on indices, we keep notation as simple as possible in the following by indexing splits in $E$ by  
 $$h,i,j,k,m,s,t\in E\,.$$
The concepts of transformation of metric tensors, Christoffel symbols and curvature employed in the following can be found in any standard text book on differential geometry, e.g. \cite{lang_fundamentals_1999, lee_introduction_2018}.

 Recall that the Riemannian structure  of \waldspace is inherited on each grove $\G{E} \cong (0,1)^E$ from the  information geometric Riemann structure of $\spd$ pulled back from $\phi_E\colon(0,1)^E \to \spd$. 
 In consequence, the Riemannian metric tensor $g_\lambda^{(\G{E})}$ of $\G{E}$, evaluated at $\lambda \in (0,1)^E$, is given by the Riemannian metric tensor $g_\lambda^{\spd}$ at $\phi_E(\lambda) = P$, where base vectors transform under the derivative of $\phi_E$: 
   \begin{equation*}
    g^{(\G{E})}_{\lambda}(x, y) = \sum_{i\in E}\sum_{j\in E} x_i y_j\; g^{(\spd)}_P\bigg(\frac{\dd\phi_{E}}{\dd\lambda_i}(\lambda),\,\frac{\dd\phi_{E}}{\dd\lambda_j}(\lambda)\bigg)
  \end{equation*}
  for $(x,y) \in T_\lambda\G{E}\times T_\lambda\G{E} \cong \RR^E\times \RR^E$ and 
  \begin{equation*}
    (\dr\phi_{E})_\lambda\big(T_\lambda\G{E}\big)
    = \vspan\bigg\{\frac{\dd\phi_{E}}{\dd\lambda_i}(\lambda)\colon i\in E\bigg\}
    \subseteq T_P\spd.
  \end{equation*}

  As usual $(g_{ij})_{i,j\in E}$ denotes the matrix of $g^{(\G{E})}_\lambda$ in standard coordinates and $(g^{ij})_{i,j\in E}$ its inverse.  This yields the \emph{Christoffel symbols} 
for $i,j,m\in E$,
  \begin{align*}
    \label{eq:christoffel_symbols_computed}
    \Gamma_{ij}^m
    = \frac{1}{2} \sum_{k\in E} \bigg(\frac{\dd g_{jk}}{\dd \lambda_i}
    + \frac{\dd g_{ki}}{\dd \lambda_j}
    - \frac{\dd g_{ij}}{\dd \lambda_k}\bigg) g^{km},
  \end{align*}
  which give the representation of the \emph{curvature tensor}
  \begin{align*}
    R_{ijks} = \sum_{t\in E} \bigg(
    \sum_{h\in E} \Gamma_{ik}^h\Gamma_{jh}^t 
    - \sum_{h\in E} \Gamma_{jk}^h\Gamma_{ih}^t 
    + \frac{\dd}{\dd\lambda_j} \Gamma_{ik}^t - \frac{\dd}{\dd\lambda_i} \Gamma_{jk}^t
    \bigg)g_{ts}
  \end{align*}
  in the coordinates $i,j,k,s\in E$.
  
  Introducing the notation ($P=\phi_E(\lambda)$)
  \begin{equation*}
    Q_i = 
    P^{-1}\,\frac{\dd\phi_E}{\dd\lambda_i}(\lambda)
    \qquad\text{and}\qquad
    Q_{ij} = P^{-1}\frac{\dd^2\phi_E}{\dd\lambda_i\dd\lambda_j}(\lambda)
  \end{equation*}
  and performing a longer calculation in coordinates $i,j \in E$, gives
  \begin{align*}
    R_{ijij}
    = \frac{1}{4} &\sum_{a,h\in E} g^{ah} 
    \Tr\Big[\big(2 Q_{ij} - Q_j Q_i - Q_i Q_j\big) Q_a\Big]
    \Tr\Big[\big(2 Q_{ij} - Q_j Q_i - Q_i Q_j\big) Q_h\Big]\nonumber\\
    - &\sum_{a,h\in E} g^{ah} 
    \Tr\Big[Q_i^2 Q_a\Big]
    \Tr\Big[Q_j^2 Q_h\Big]\nonumber\\
    - &\Tr\Big[\big(2 Q_{ij} - Q_j Q_i - Q_i Q_j\big)Q_{ij}\Big].
  \end{align*}
  Evaluating the sectional curvature tensor at a pair of tangent vectors $x,y\in T_\lambda\G{E}\cong\RR^E$ at $\lambda$ gives the \emph{sectional curvature} $K(x,y)$ at $\lambda$ of the local two-dimensional subspace spanned by geodesics with initial directions 
  generated by linear combinations of 
  $x$ and $y$. Abbreviating  
$\vert x\vert_\lambda^{(\G{E})} \coloneqq (g_\lambda^{(\G{E})}(x, x))^{1/2}$ we have 
  \begin{equation*}
    K(x, y) = \frac{\sum_{i\in E} \sum_{j\in E} x_i y_i R_{ijij}}{\vert x\vert_\lambda^{(\G{E})} \vert y\vert_\lambda^{(\G{E})} - g^{(\G{E})}_\lambda(x, y)}.
  \end{equation*}

  \begin{example}
    \label{exa:sectional-curvatures-examples}
    Again revisiting $\W$ from \Cref{exa:boundary-of-grove-n3} with $N=3$, we first consider \walds in the unique top-dimensional grove $\G{E}\cong(0,1)^3$, and then on its boundary.
    \begin{enumerate}
      \item 
        We compute minimum and maximum sectional curvatures at the \walds $F$ with $\lambda = (a,a,a)$, for $a\in(0,1)$, as displayed in  \Cref{fig:sectional-curvatures-diagonal}.
        Traversing along $0<a <1$ we find both positive and negative sectional curvatures and their extremes escape to positive and negative infinity as the vantage point, the isolated forest $F_\infty$, is approached, where all dimensions collapse.
      \item 
        In order to assess \emph{Alexandrov curvature} that measures ``fatness/slimless'' of geodesic triangles (hence it does not require a Riemannian structure, a geodesic space suffices, see \cite{sturm_probability_2003}) we compute several geodesic triangles and their respective angle sums within $\W$.
        The corners of the triangles are \walds $F_1,F_2,F_3\in\W$, where $\{i,j,k\} = \{1,2,3\}$, $E_i := \{j\vert k\}$ and $\lambda_{j\vert k} \in (0,1]$.         \Cref{fig:waldspace-3-geodesic-triangles-angle-sums} depicts the geodesic triangles (left panel) non isometrically embedded in $\RR^3$ representing the off-diagonals in $\spd$, as well as their respective angle sums (right panel). 
        When the two connected leaves approach one another ($\lambda_e\approx 0$) triangles become infinitely thin, but near $F_\infty$ ($\lambda_e\approx 1$) the triangles become Euclidean. 
    \end{enumerate}

  \end{example}
  
  \begin{figure}[ht]
    \centering
    \includegraphics[width=0.9\linewidth, trim=0cm 5.2cm 0cm 0.5cm, clip]{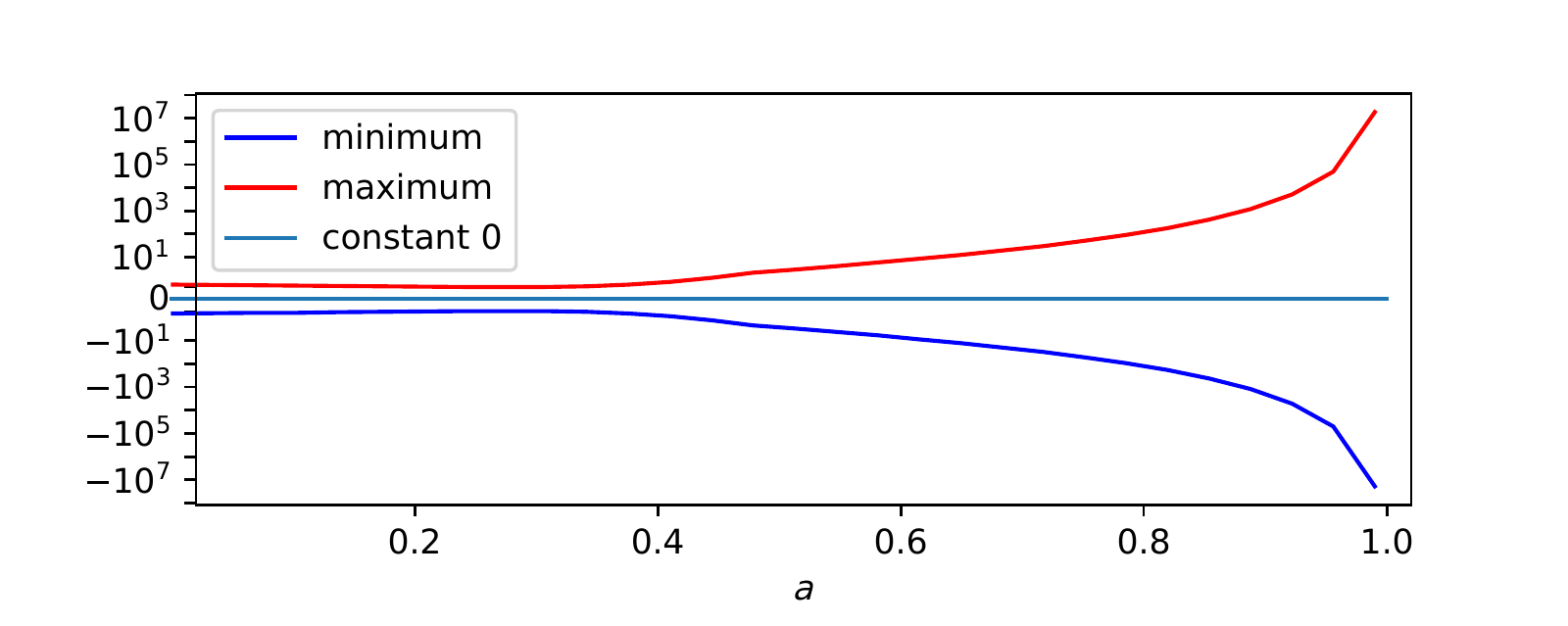}
    \caption{Minimum and maximum sectional curvatures along $0< a < 1$ of  \waldspace ($N=3$) at $F \in \W$ with $\lambda=(a,a,a)$, as described in \Cref{exa:sectional-curvatures-examples}.
    }
    \label{fig:sectional-curvatures-diagonal}
  \end{figure}

  \begin{figure}[ht] 
    \centering
    \begin{minipage}{0.435\linewidth}
        \includegraphics[width=\linewidth, trim=1.cm 0cm 1.cm 0cm, clip]{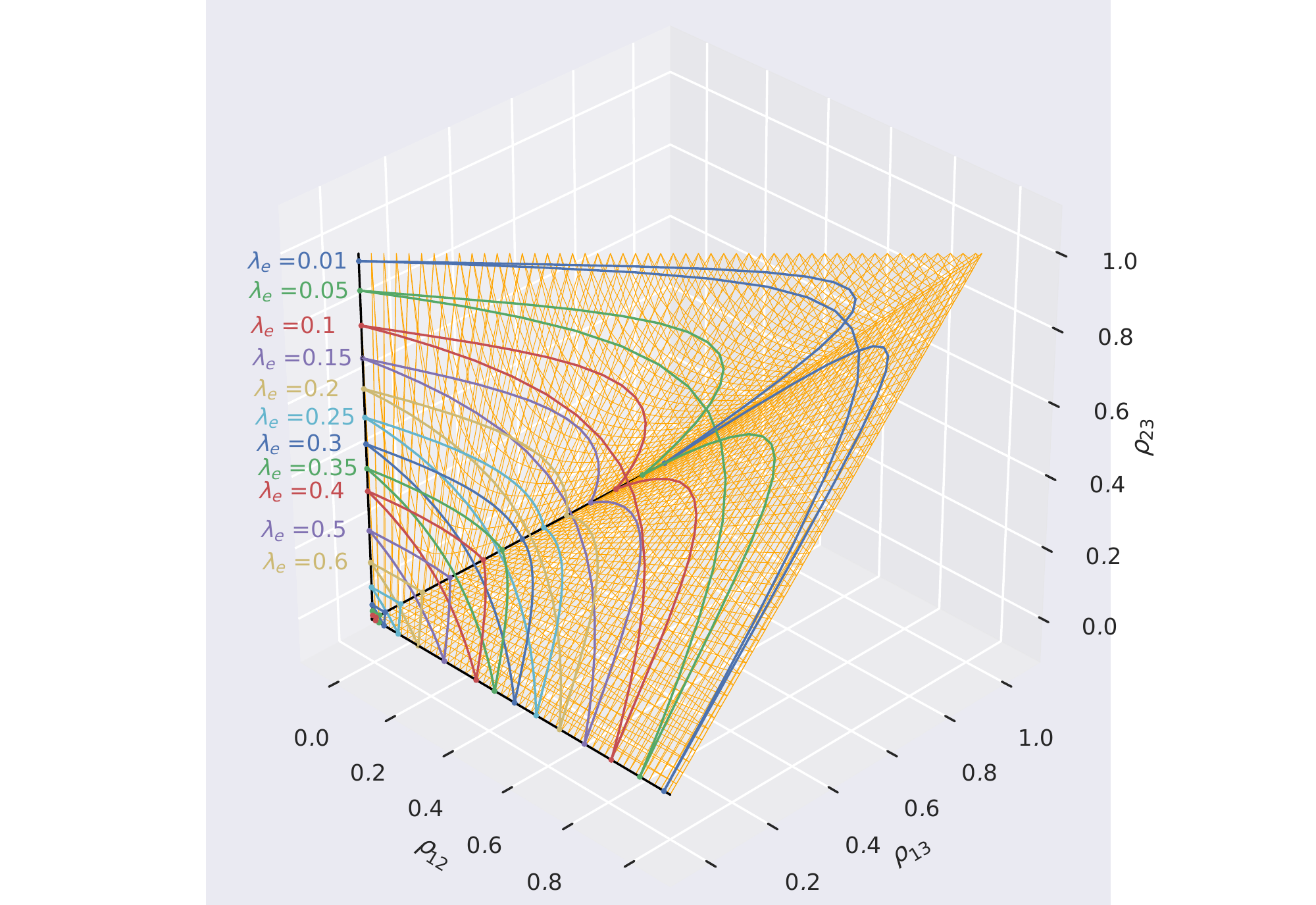}
    \end{minipage}\hfil
    \begin{minipage}{0.435\linewidth}
        \includegraphics[width=\linewidth, trim=1.cm 0cm 1.cm 0cm, clip]{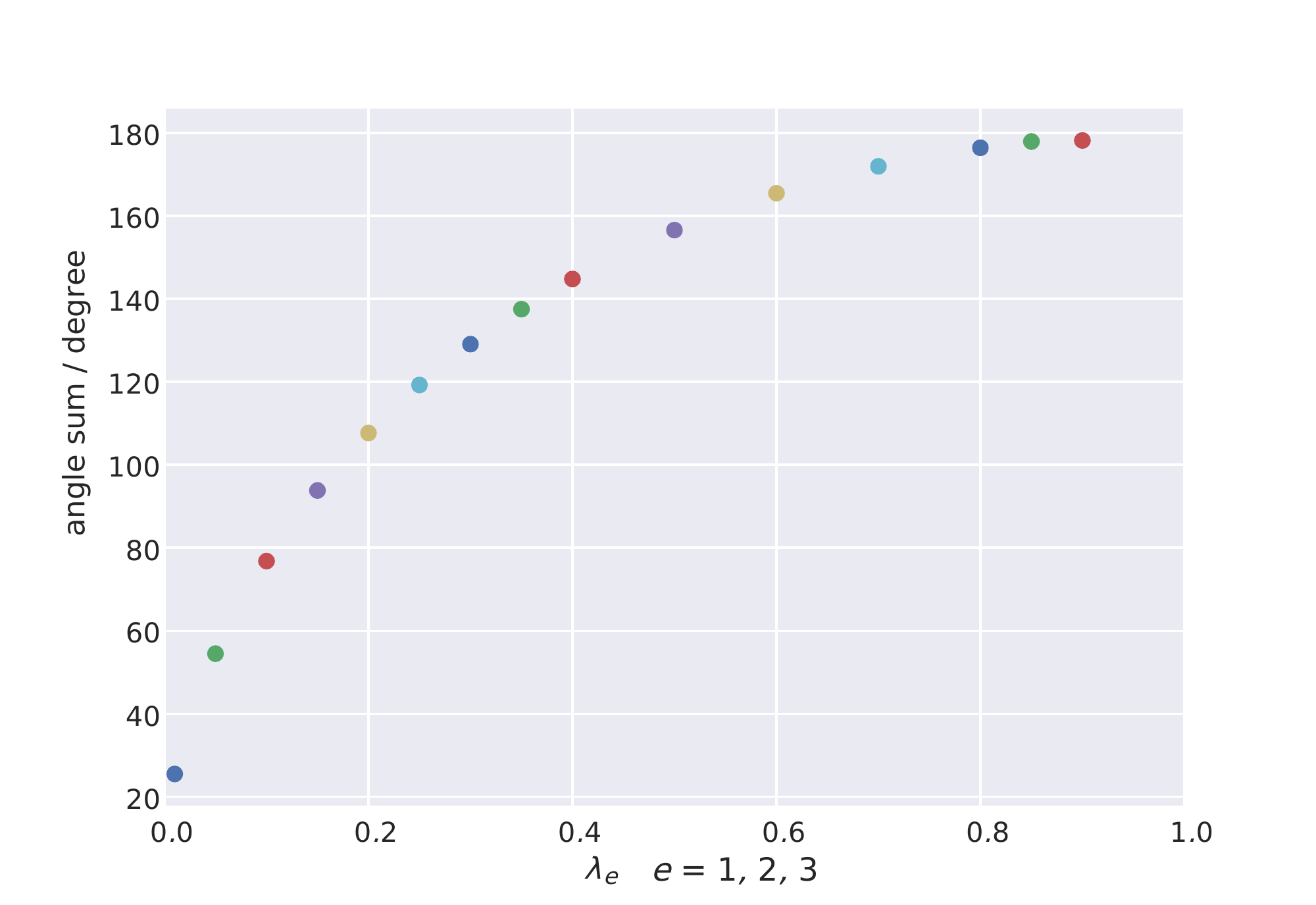}
    \end{minipage}%
    \caption{Displaying sums of angles in degrees  (right) of geodesic triangles spanned by three \walds for $N=3$ with one disconnected leaf and edge weight $0<\lambda_e<1$ between the other two leaves as discussed in \Cref{exa:sectional-curvatures-examples}. Embedding $\W$ in $\spd$ viewed (non isometrically) as $\RR^3$, the geodesic triangles are visualized on the left, where the origin corresponds to $\lambda_e=1$.
    }
    \label{fig:waldspace-3-geodesic-triangles-angle-sums}
  \end{figure}

\begin{conjecture}
        This example hints towards a general situation:
        \begin{itemize}
            \item[(i)] 
            \Waldspace groves feature positive and negative sectional curvatures alike, both of which become unbounded when approaching the vantage point $F_\infty$. 
            \item[(ii)] When approaching the  infinitely far away boundary of $\spd$ from within $\W$, some Alexandrov curvatures tend to negative infinity.
        \end{itemize}
\end{conjecture}

 \subsection{Exploring Stickiness in Wald Space} 

Statistical applications in tree space often require the concept of a mean or average tree. 
Since the expectation of a random variable taking values in a non-Euclidean metric space $(M,d)$ is not well-defined, \cite{F48} proposed to resort instead to a minimizer of expected squared distance to a random element $X$ in $M$,
$$p^* \in \argmin_{p\in M}\EE[d(p,X)^2]\,$$
called a \emph{barycenter} or \emph{Fr\'echet mean}. 
In a Euclidean geometry, if existent, the Fr\'echet mean is unique and identical to the expected value of $X$. 
Given a sample $X_1,\ldots,X_n \iid X$, measurable selections from the set
$$ \argmin_{p\in M} \mc{F}(p),\quad \mc{F}(p):=\frac{1}{n}\sum_{j=1}^n d(p,X_j)^2$$
are called \emph{empirical Fr\'echet means} and their asymptotic fluctuations allow for nonparametric statistics. Usually,  $\mc{F}$ is called the \emph{empirical Fr\'echet function}.

Recently, it has been discovered by \cite{HH15,EltznereHuckemann2019} that positive curvatures may increase asymptotic fluctuation by orders of magnitude, and by \cite{hotz_sticky_2013,huckemann_sticky_2015} that infinite negative Alexandrov curvature may completely cancel asymptotic fluctuation, putting a dead end to this approach of non-Euclidean statistics. In particular, this can be the case for BHV spaces, cf. \cite{barden_central_2013,barden_limiting_2016,barden_logarithm_2018}.

  \begin{example}[Stickiness in \waldspace]\label{exa:stickiness}
    Consider two samples $F_1,F_2,F_3\in\W$ and $F_1',F_2,F_3\in\W$ with $N=4$, 
     depicted in \Cref{fig:waldspace-4-frechet-mean-samples}, where $F_1$ and $F_1'$ only differ by weights of their interior edges. By symmetry, their Fr\'echet means are of form $F$ having equal but unknown pendent edge weights $0<\lambda_{pen}<1$ and unknown interior edge weights $0\leq \lambda_{int}< 1$, as in \Cref{fig:waldspace-4-frechet-mean-samples}. It turns out that the Fr\'echet means of both samples agree in BHV with $\lambda_{int} =0$, i.e. the empirical mean \emph{sticks} to the lower dimensional \emph{star tree stratum} (featuring only pendant edges).
     
     In contrast, the two empirical Fr\'echet functions in \waldspace
     \begin{eqnarray*}
     \mc{F}(F) &=& \frac{1}{3}\left(d_\W(F,F_1)^2 + d_\W(F,F_2)^2 +d_\W(F,F_3)^2\right)\\
     \mc{F}'(F) &=& \frac{1}{3}\left(d_\W(F,F'_1)^2 + d_\W(F,F_2)^2 +d_\W(F,F_3)^2\right)
     \end{eqnarray*}
     have different minimizers, and, in particular the minimizer for $\mc{F}'$ does not stick to the star stratum but has $\lambda_{int}>0$. \Cref{fig:waldspace-4-frechet-mean-heat-plots} illustrates the values of $\mc{F}$ and $\mc{F}'$ for different values of the parameters $\lambda_\mr{pen},\lambda_\mr{int}$ of $F$ near the respective minima.
      \end{example}

  \begin{figure}[ht]
    \centering
    \begin{minipage}{0.8\linewidth}
      \includegraphics[width=1.0\linewidth]{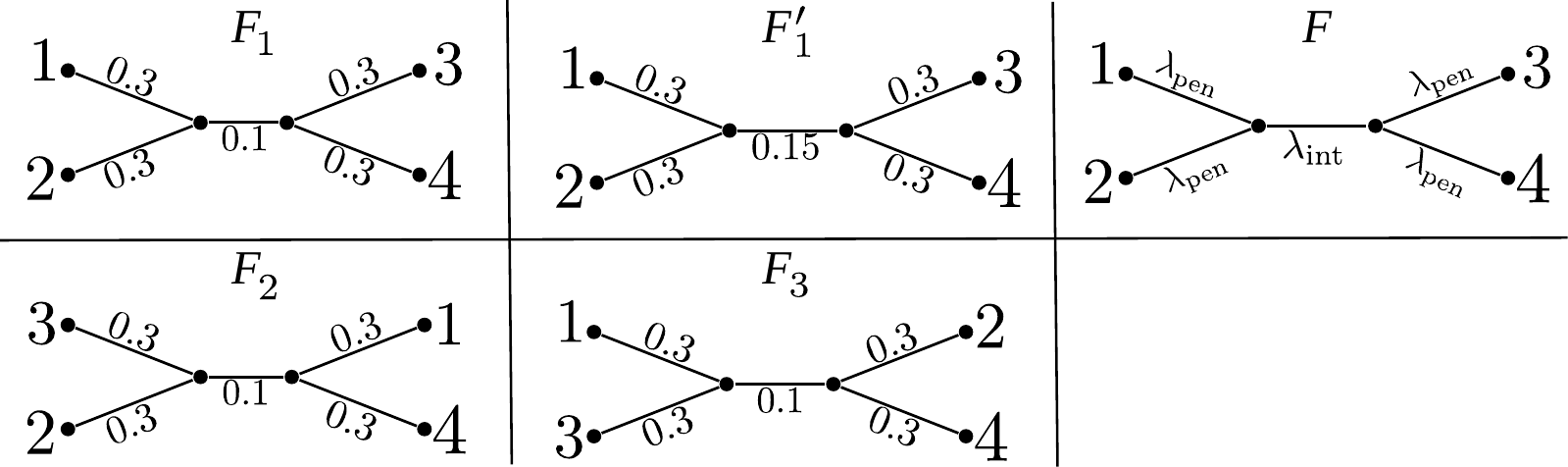}
    \end{minipage}
      \caption{Two samples of \walds: $F_1,F_2,F_3$ and $F_1',F_2,F_3$ where
      $F_1$ and $F_1'$ only differ by weights of their interior edges. By symmetry, $F$ is a candidate for each Fr\'echet mean, see \Cref{exa:stickiness}.
      } 
      \label{fig:waldspace-4-frechet-mean-samples}
  \end{figure}
  
  \begin{figure}[ht]
    \centering
    \begin{minipage}{0.45\linewidth}
        \includegraphics[width=1.0\linewidth]{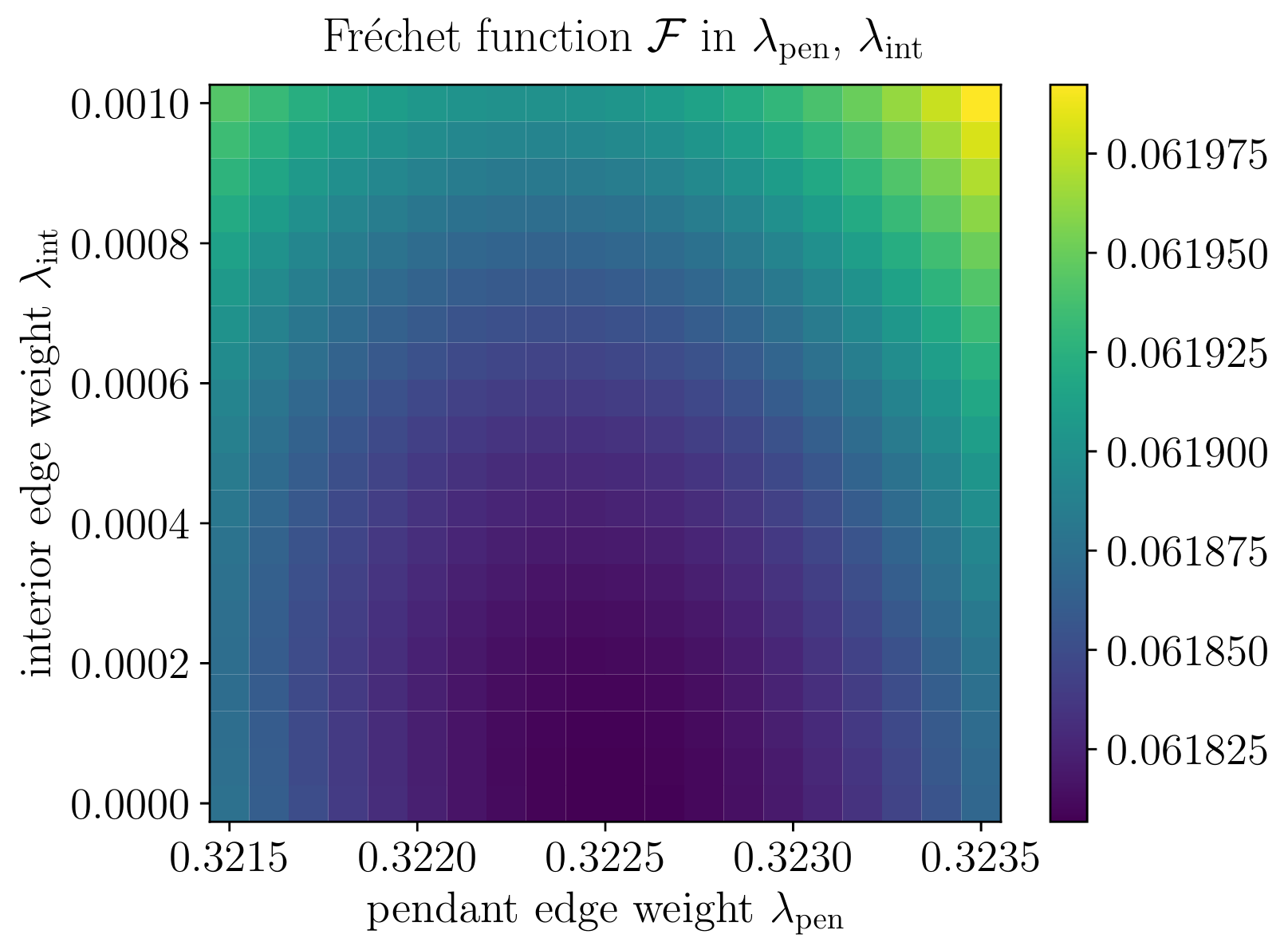}
    \end{minipage}\hfil
    \begin{minipage}{0.45\linewidth}
        \includegraphics[width=1.0\linewidth]{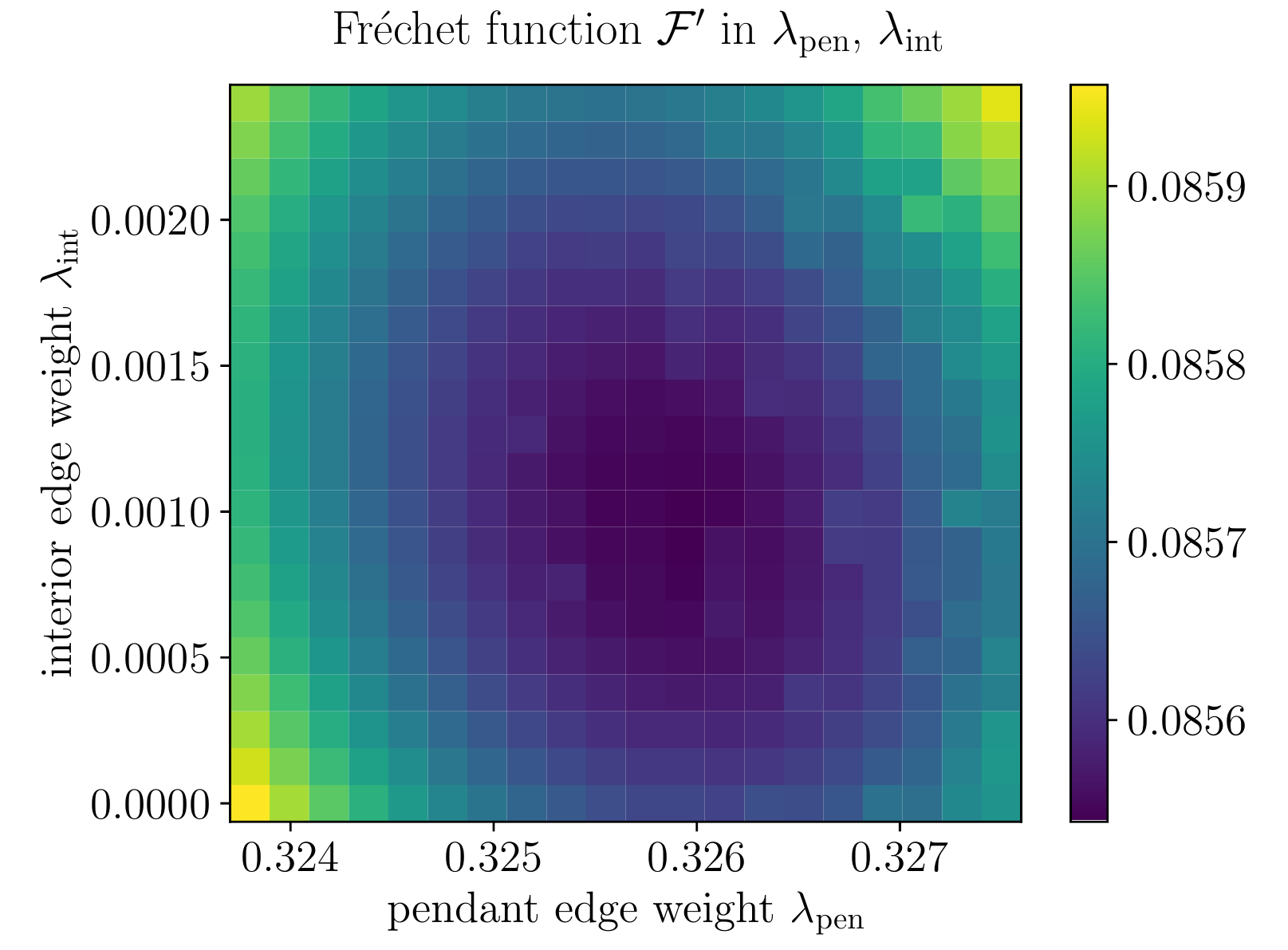}
    \end{minipage}%
      \caption{Heat map for the values of the Fr\'echet functions $\mc{F}$ (left) and $\mc{F}'$ (right) of two samples as functions of  $\lambda_\mr{pen},\lambda_\mr{int}$ determining candidate minimiziers $F$ as detailed in \Cref{exa:stickiness} and illustrated in \Cref{fig:waldspace-4-frechet-mean-samples}.
      }
      \label{fig:waldspace-4-frechet-mean-heat-plots}
  \end{figure}

  \begin{remark}
          This preliminary research indicates that effects of stickiness, which are still expected where ``too many'' lower dimensional strata hit higher dimensional strata, are less severe in \waldspace than in BHV space, thus making \waldspace more attractive for asymptotic statistics based on Fr\'echet means.
  \end{remark}
  
\section{Discussion}\label{scn:discussion}

In previous work \citep{garba_information_2021}, the wald space was introduced as a space for statistical analysis of phylogenetic trees, based on assumptions with a stronger biological motivation than existing spaces. 
In that work, the focus was primarily on geometry, whereas here we have provided a rigorous characterization of the toplogy of wald space. 
Specifically, wald space $\W$ is a disjoint union of open cubes with the Euclidean toplogy, and as topological subspaces we have 
$$\mathcal{BHV}_{N-1}\subset\W\subset\mc{E}_N\,$$
with the BHV space $\mathcal{BHV}_{N-1}$ from \cite{billera_geometry_2001} and the edge-product space $\mc{E}_N$ from \cite{moulton_peeling_2004}.
We have shown that this topology is the same as that induced by the information metric $d_\W$ defined in \cite{garba_information_2021}.
Furthermore, we have shown $\W$ is contractible, and so does not contain holes or handles of any kind. 
Examples suggest that $\W$ is a truncated cone in some sense  (see Figure~\ref{fig:cone-waldspace-slices}), but its precise formulation remains an open problem. 
As established in Theorem~\ref{thm:whitney}, boundaries between strata in wald space satisfy Whitney condition (A); whether Whitney condition (B) holds is an open problem, although we expect it to hold on the boundaries of any grove $(0,1)^E\cong\G{E}$ corresponding to the limit as one or more coordinates $\lambda_e\rightarrow 0$ (i.e. the boundaries between strata in $\mathcal{BHV}_{N-1}$). 
Our key geometrical result is that with the metric $d_\W$, wald space is a geodesic metric space, Theorem~\ref{thm:waldspace-is-geodesic-metric-space}. 
The existence of geodesics greatly enhances the potential of wald space as a home for statistical analysis.  

The approximate geodesics computed via the algorithm in Definition~\ref{def:geodesic-approximation-algorithm} provide insight into the geometry and a source of conjectures. 
For example, unlike geodesics in BHV tree space, it appears that geodesics in wald space can run for a proportion of their length along grove boundaries, even when the end points are within the interior of the same grove (see Example~\ref{exa:waldspace-3-geodesics-over-zero-boundary}). 
If wald space is uniquely geodesic (so that there is a unique geodesic between any given pair of points), its potential as a home for statistical analysis would be improved further. 
However, the presence of positive and negative sectional curvatures for different pairs of tangent vectors at the same point, and an apparent lack of global bounds on these, suggests geodesics may be non-unique, or at least makes proving uniqueness more challenging. 
Finally, Example~\ref{exa:stickiness} which involves approximate calculation of Fr\'echet means, suggests that wald space is less `sticky' than BHV tree space and hence more attractive for studying asymptotic statistics. 
 
A variety of open problems remain, and we make the following conjectures. 
\begin{enumerate}
\item All points on any geodesic between two trees are also trees. 
\item Geodesics between trees in the same grove do not leave the closure of that grove.
\item The disconnected forest $F_\infty$ 
is repulsive, in the sense that the only geodesics passing through the disconnect forest have an end point there. 
\end{enumerate}
Other open problems include the following, all mentioned elsewhere in the paper. 
\begin{enumerate}
\setcounter{enumi}{3}
\item Is wald space a truncated topological cone?
\item Does Whitney condition (B) hold at grove boundaries?
\item Most importantly for statistical applications, is wald space uniquely geodesic or 
can examples of exact non-unique geodesics be constructed? What is then the structure of cut loci?
\end{enumerate}

\section*{Acknowledgements}

The authors gratefully acknowledge helpful discussions with Fernando Galaz-Garc\'ia and support from the DFG RTG 2088 (J. Lueg), and from the DFG HU 1575-7 and the Niedersachsen Vorab of the Volkswagen Foundation (S.F. Huckemann).

\bibliographystyle{apa}

\end{document}